\documentclass[12pt]{article}
\usepackage[english]{babel}
\usepackage{float}
\usepackage[a4paper, top=2.5cm,bottom=2.5cm, left=3cm, right=3cm]{geometry}
\usepackage{amsthm}
\usepackage{amsfonts}
\usepackage{amsmath}
\usepackage{physics}
\usepackage{bbm}
\usepackage{amssymb}
\usepackage{todonotes}
\usepackage{graphicx}
\usepackage{derivative}
\usepackage{cancel}
\usepackage{chngcntr}
\usepackage{tikz-cd}
\usepackage{hyperref}
\usepackage[nameinlink,noabbrev]{cleveref} 
\usepackage{listings}
\usepackage{amssymb}
\usepackage{algpseudocode}
\usepackage{algorithm}
\usepackage{comment}
\usepackage{csquotes}
\usepackage{subcaption}
\usepackage{setspace}
\usepackage{mathtools}
\usepackage{adjustbox}
\usepackage[
backend=biber,
style=alphabetic,
sorting=nyt,
natbib=true 
]{biblatex}
\usepackage{xcolor}
\usepackage{tikz}
\usetikzlibrary{arrows,shapes,positioning}
\usepackage{fancyhdr}

\theoremstyle{definition}
\newtheorem{definition}{Definition}[section]

\theoremstyle{remark}
\newtheorem{oss}{Remark}[section]

\theoremstyle{plain}
\newtheorem{prop}{Proposition}[section]
\newtheorem{thm}{Theorem}[section]

\newtheorem{cor}{Corollary}[section]
\newtheorem{lem}{Lemma}[section]

\crefname{thm}{Theorem}{Theorems}
\Crefname{thm}{Theorem}{Theorems}
\crefname{prop}{Proposition}{Propositions}
\Crefname{prop}{Proposition}{Propositions}
\crefname{lem}{Lemma}{Lemmas}
\Crefname{lem}{Lemma}{Lemmas}
\crefname{cor}{Corollary}{Corollaries}
\Crefname{cor}{Corollary}{Corollaries}
\crefname{oss}{Remark}{Remarks}
\Crefname{oss}{Remark}{Remarks}

\DeclareMathOperator*{\diam}{diam}

\newcommand{\Ord}{\mathrm{Ord}}
\newcommand{\Ext}{\mathrm{Ext}}
\newcommand{\Rel}{\mathrm{Rel}}

\newcommand{\res}{res}
\newcommand{\dist}{dist}
\newcommand{\rch}{rch}

\newcommand{\id}{\mathrm{id}}

\newcommand{\N}{\mathbb{N}}
\newcommand{\R}{\mathbb{R}}

\title{Building confidence regions for Reeb graphs using the interleaving distance}
\author{Matteo Pegoraro\thanks{Institute of Computing, Faculty of Informatics, Università della Svizzera Italiana, Lugano, Switzerland.},
Alberto Conforti\thanks{Department of Economics, Management and Statistics,
University of Milano--Bicocca, Milan, Italy.},
Mathieu Carrière\thanks{DataShape team,
Centre Inria d'Université Côte d'Azur, Sophia-Antipolis, France.}}
\date{}

\begin{document}

\maketitle

\begin{abstract}
We develop confidence regions for Reeb graphs from finite samples using the
interleaving distance.  Given a point cloud equipped with a filter function, we
construct a finite proximity graph, extend the filter linearly, and use the
Reeb cosheaf of the resulting filtered graph as the primary estimator.  Mapper
graphs are then treated as controlled cover-based coarsenings of this 
estimator, separating the statistical approximation problem from the
visualization problem.  We prove stability bounds for the
Reeb estimators obtained both using intrinsic and extrinsic metrics, the latter under positive-reach assumptions, and derive
interleaving-distance confidence regions from either \((a,b)\)-standard sampling
assumptions or subsampling-based Hausdorff scale estimates.  We also compare this
object-level metric viewpoint with persistence-based guarantees by showing that
the extended-persistence pseudometric is bounded by twice the interleaving
distance, with sharp constant \(1\) for the \(H_0\)-related components.  Numerical
experiments illustrate how statistically significant features can be identified
and then projected to Mapper graphs for interpretation.
\end{abstract}

\section{Introduction}

Reeb graphs provide a compact topological summary of a real-valued function on a
space by recording how the connected components of its level sets appear, merge,
and disappear as the function value varies. Originating in shape analysis and
scientific computing \cite{biasotti}, they have become a central object in
topological data analysis. In practice, Reeb graphs are accessed through finite
approximations, most notably Mapper graphs \cite{singh}, which are widely used
for exploratory analysis, visualization, clustering, and feature selection in
high-dimensional settings \cite{yao,lum}. These tools have supported applications
ranging from biomolecular dynamics \cite{yao} to modern biological pipelines such
as single-cell Hi-C contact map analysis \cite{carriere2020hic}, topology-driven
clustering of global gene expression profiles \cite{nicolau_topology_2011,
jeitziner2019ttmap}, and human cell differentiation trajectories from single-cell
RNA-seq data~\cite{rizvi_single-cell_2017, kandror_enhancer_2026}.

A main challenge for using Reeb- and Mapper-type summaries in statistics is to
provide principled parameter selection and uncertainty quantification: given a
finite sample drawn from an unknown space, what can be said about the underlying
Reeb graph? A foundational contribution in this direction is the work of
Carri\`ere, Michel, and Oudot \cite{JMLR:v19:17-291}, which provides statistical
guarantees and parameter-selection procedures for Mapper and derives confidence
sets for Reeb graphs using bottleneck-type constructions. Their analysis
crucially leverages the stability of persistence-based invariants and yields
practically useful confidence statements.

A key point of departure of the present article and other recent works \cite{brown2020probabilisticconvergencestabilityrandom, bjerkevik2025reeb} with respect to \cite{JMLR:v19:17-291}
 is the metric used to compare
Reeb-type objects. The bottleneck distance derived from persistence-based
approaches yields only a \emph{pseudo}-metric between Reeb or Mapper graphs:
distinct non-isomorphic Reeb graphs can have distance zero because persistence
signatures do not separate all Reeb graphs, see for example \Cref{fig:pseudo_metric}. In contrast, the \emph{interleaving
distance} for constructible cosheaves introduced in the categorified Reeb-graph
framework of de Silva, Munch, and Patel \cite{de_Silva_2016} gives a stronger
object-level geometry. On constructible cosheaves, and hence on Reeb graphs via
the Reeb--cosheaf equivalence, distance zero is equivalent to isomorphism:
\begin{equation}\label{eq:intro_metric_zero}
d_I(\mathcal{F},\mathcal{G})=0 \quad\Longrightarrow\quad
\mathcal{F}\cong \mathcal{G}.
\end{equation}
As a consequence, an interleaving-distance confidence ball gives a much more interpretable
confidence statement for the underlying Reeb-type object.

\begin{figure}
    \centering
    \includegraphics[width=0.6\textwidth]{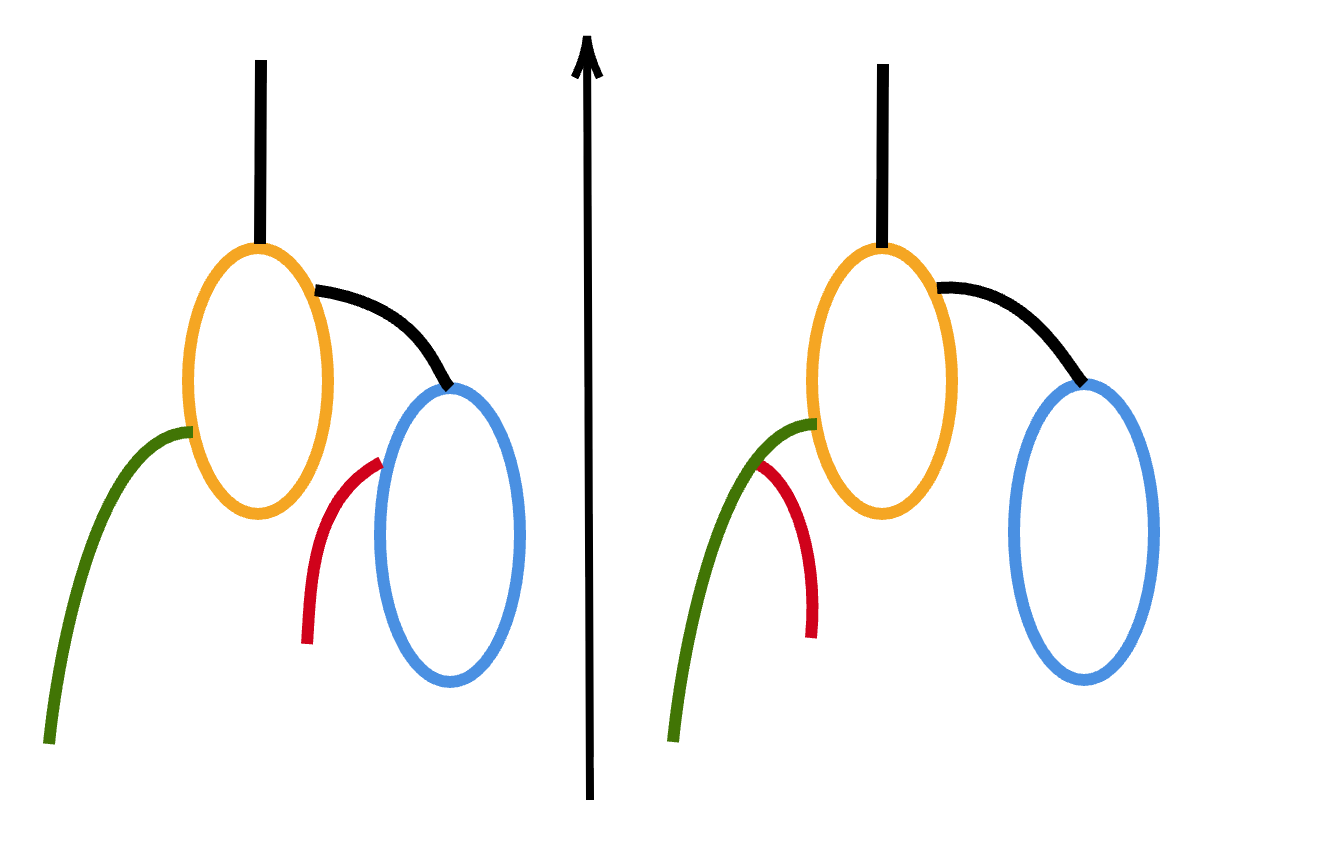}
    \caption{Two non-isomorphic Reeb graphs with identical barcode representations. Consequently, they cannot be distinguished by the bottleneck distance used in \cite{JMLR:v19:17-291}. Matching colors identify the corresponding topological features in the extended-barcode decomposition, up to the merging structure of path-connected components, which must be inferred using the elder rule.}
    \label{fig:pseudo_metric}
\end{figure}

Beyond its foundational role, the interleaving distance has become increasingly
important from computational and methodological standpoints. Recent work proposes
computable surrogates and bounds for interleavings in Mapper-like settings
\cite{munch}. Moreover, interleaving-based distances for merge trees and related summaries
have also been developed from theoretical and algorithmic viewpoints
\cite{morozov2013interleaving, agarwal2018gh, touliwang2019esa,
gasparovic2019intrinsic, curry2021decorated, teo}.

\paragraph{Contributions.}
This article develops an interleaving-distance framework for confidence regions for Reeb graphs.  The basic data object is a filtered proximity graph.  Given a finite sample \(S_n\subseteq X\), a scale \(\delta>0\), and either the intrinsic metric \(d_X\) or an ambient Euclidean metric, we form the graph \(\Gamma_\delta^{S_n,\rho}\) by connecting sample points at distance at most \(2\delta\).  The sampled filter values extend linearly to a PL function \(\hat f_n^\rho\) on this graph, and we use the Reeb cosheaf of the finite filtered graph
$(|\Gamma_\delta^{S_n,\rho}|,\hat f_n^\rho)
$
as the primary estimator.  Mapper graphs are then treated as controlled cover-based coarsenings of this PL--Reeb estimator: they are used for visualization and interpretation, while the approximation and confidence statements are made at the PL level.

The main contributions are the following.

\begin{itemize}
\item We establish a general comparison between the extended-persistence pseudometric
\(d_\Delta\) and the interleaving distance. For constructible
\(\mathbb R\)-spaces \((X,f)\) and \((Y,g)\), we prove
\[
d_\Delta\!\left(\mathcal F_{(X,f)},\mathcal F_{(Y,g)}\right)
\le
2\,d_I\!\left(\mathcal F_{(X,f)},\mathcal F_{(Y,g)}\right),
\]
together with the sharp constant-\(1\) estimate for the \(H_0\)-related
components \(\Ord_0\), \(\Ext_0\), and \(\Rel_1\), and a finer comparison through
the interleaving distances between the sublevel and superlevel merge trees of
\(f\) and \(g\). This gives a direct quantitative link between Reeb graphs,
merge trees, and persistence representations, while also connecting our
object-level confidence regions with the Mapper framework
of~\cite{JMLR:v19:17-291}.

\item We prove stability theorems for the intrinsic and extrinsic PL--Reeb estimators.  If \(\delta\) is an intrinsic covering radius for \(S_n\), the intrinsic estimator is controlled in interleaving distance by a scale \(\mu_\delta\) determined by the modulus of continuity of \(f\) and by the variation of \(f\) along the graph-edge realizations.  Under a concavity assumption on the modulus, and in particular for \(1\)-Lipschitz filters, this gives a \(\delta\)-scale bound.  Under positive reach, a Euclidean covering radius gives an extrinsic bound through the reach-distortion scale \(\eta_\tau(\delta)=2\tau\arcsin(\delta/\tau)\).

\item We separate two tasks that are often combined in Mapper constructions: approximating the target Reeb object from the sample and compressing the resulting object into a lower-resolution summary.  Once \(\mathcal R_\delta^{S_n,\rho}\) has been constructed, any finite nice open cover \(\mathcal U\) of \(\mathbb R\) gives the Mapper coarsening
$\mathcal M_{\mathcal U}(\mathcal R_\delta^{S_n,\rho}).
$

By the Mapper stability theorem of~\cite{brown2020probabilisticconvergencestabilityrandom}, the interleaving discrepancy introduced by this coarsening is at most \(\res_{\mathcal R_\delta^{S_n,\rho}}(\mathcal U)\).

\item We derive interleaving-distance confidence regions for the unknown Reeb graph.  Classical \((a,b)\)-standard assumptions and subsampling-based Hausdorff quantile estimates in the spirit of~\cite{fasy} both yield choices of \(\delta\) for the intrinsic and extrinsic PL estimators.  The corresponding Mapper coarsenings inherit the same guarantees up to the cover resolution.  Through the comparison \(d_\Delta\le 2d_I\), these regions also imply confidence statements for extended-persistence signatures.

\item We provide numerical experiments illustrating the resulting workflow: PL graphs are used for approximation and statistical inference, whereas Mapper coarsenings are used to visualize significant features.  The experiments compare intrinsic and extrinsic metrics, deterministic and probabilistic scale choices, and rate-corrected subsampling radii; the validation of the subsampling corrections is reported in the appendix.
\end{itemize}

\paragraph{Related work.}
In the statistical Reeb/Mapper literature, the closest works to ours are
\cite{JMLR:v19:17-291},
\cite{brown2020probabilisticconvergencestabilityrandom}, and
\cite{bjerkevik2025reeb}. The practical relevance of
\cite{JMLR:v19:17-291} is clear: it provides a statistical analysis of
Mapper, parameter-selection procedures, and confidence statements for topological
features. Its guarantees, however, are expressed through persistence-based
bottleneck bounds and, at the sample sizes considered in its experiments, the
resulting theoretically justified confidence regions are reported to be too
conservative to allow for interpretation. The numerical analysis therefore relies
instead on a bottleneck bootstrap whose validity for Mapper is left open. Our
framework gives object-level interleaving confidence regions, and its sharper
deterministic bounds remain practically informative in the same finite-sample
regime. The inequality \(d_\Delta\le 2d_I\), together with the constant-\(1\)
control of the \(H_0\)-related components and the finer comparison through
sublevel and superlevel merge trees, makes the relationship with
persistence-based guarantees explicit.

The approach of
\cite{brown2020probabilisticconvergencestabilityrandom} is also formulated in
interleaving distance, but relies on a density-based recovery pipeline. By
contrast, our estimators are defined directly from the observed point cloud.
The Mapper transformation of
\cite{brown2020probabilisticconvergencestabilityrandom} nevertheless enters our
framework as a controlled coarsening of the finite PL--Reeb estimator.

The recent and independently developed work
\cite{bjerkevik2025reeb} has a more direct overlap with our deterministic
stability results. It proves a general transfer principle for a space
\(X\subseteq Y\) that is obtained as a suitable deformation retract of \(Y\):
if \(g\colon Y\to\mathbb R\) and the target filter is
\(f=g|_X\), the variation of \(g\) along the retraction controls the
interleaving distance between the Reeb graphs \(\mathfrak R(Y,g)\) and
\(\mathfrak R(X,f)\). Combined with existing geometric reconstruction
theorems, this yields sample-based approximation results in which \(Y\) is a
metric thickening of the sample, including Euclidean positive-reach and
Riemannian settings. The closest overlap is therefore with our extrinsic theorem:
both approaches obtain deterministic interleaving bounds from geometric control
of a finite sample. The constructions are nevertheless different. In \cite{bjerkevik2025reeb}, the estimator is the Reeb graph of a metric thickening and requires a Lipschitz filter defined on that thickening. The authors also treat samples lying within a prescribed ambient distance of the target. In our framework, by contrast, the primary estimator is the Reeb cosheaf of a finite proximity graph equipped with the PL extension of the filter values observed on the sample, and we develop this construction for both intrinsic and extrinsic sample metrics. Lastly, although the two works
address overlapping deterministic approximation questions, they were developed
independently: the present project began in spring 2025 in the context of A.C.’s master’s thesis \cite{tesi_alberto}.

\paragraph{Outline.}
In \Cref{sec:background}, we recall Reeb graphs, constructible cosheaves, Mapper,
and the intrinsic--extrinsic metric comparison under positive reach. In
\Cref{sec:dDelta_le_2dI}, we prove the comparison \(d_\Delta\le 2d_I\). In
\Cref{sec:Mapper_cosheaf}, we define the PL-Reeb estimators and their Mapper
coarsenings, prove the corresponding stability theorems, and discuss efficient
computation. In \Cref{conf}, we derive confidence regions by choosing
\(\delta\) through probabilistic control of the relevant Hausdorff distance:
intrinsic for the intrinsic graph, and Euclidean for the extrinsic graph under
positive reach. We also record the conversion from Euclidean to intrinsic
sampling scales. Finally, \Cref{expes}
presents the numerical experiments.

\section{Background}\label{sec:background}

\subsection{Reeb graphs}

We recall a few definitions and results that will be useful in our setting. We work with a class of well-behaved topological spaces, namely constructible \(\mathbb{R}\)-spaces (see Subsection~2.2 in~\cite{de_Silva_2016}). We denote by \(\mathbb{R}\mathbf{-Top}^c\) the category whose objects are constructible \(\mathbb{R}\)-spaces.

\begin{definition}
Let \((X,f) \in \mathbb{R}\mathbf{-Top}^c\) (\(f\) is also referred to as a filter function). Define an equivalence relation \(\sim_f\) on \(X\) by declaring \(x \sim_f y\) if and only if there exists \(a\in\mathbb{R}\) such that \(x\) and \(y\) lie in the same path-connected component of the level set \(f^{-1}(a)\). Let
\[
q\colon X \to X/\!\sim_f
\]
denote the quotient map. Since \(f\) is constant on each equivalence class by construction, it induces a well-defined continuous map
\[
\bar f \colon X/\!\sim_f \to \mathbb{R}
\qquad\text{such that}\qquad
\bar f \circ q = f.
\]
The Reeb graph of \((X,f)\) is
\[
\mathfrak R(X,f)\coloneqq (X/\!\sim_f,\bar f).
\]
\end{definition}

\begin{figure}
    \centering
    \includegraphics[width = 0.6\textwidth]{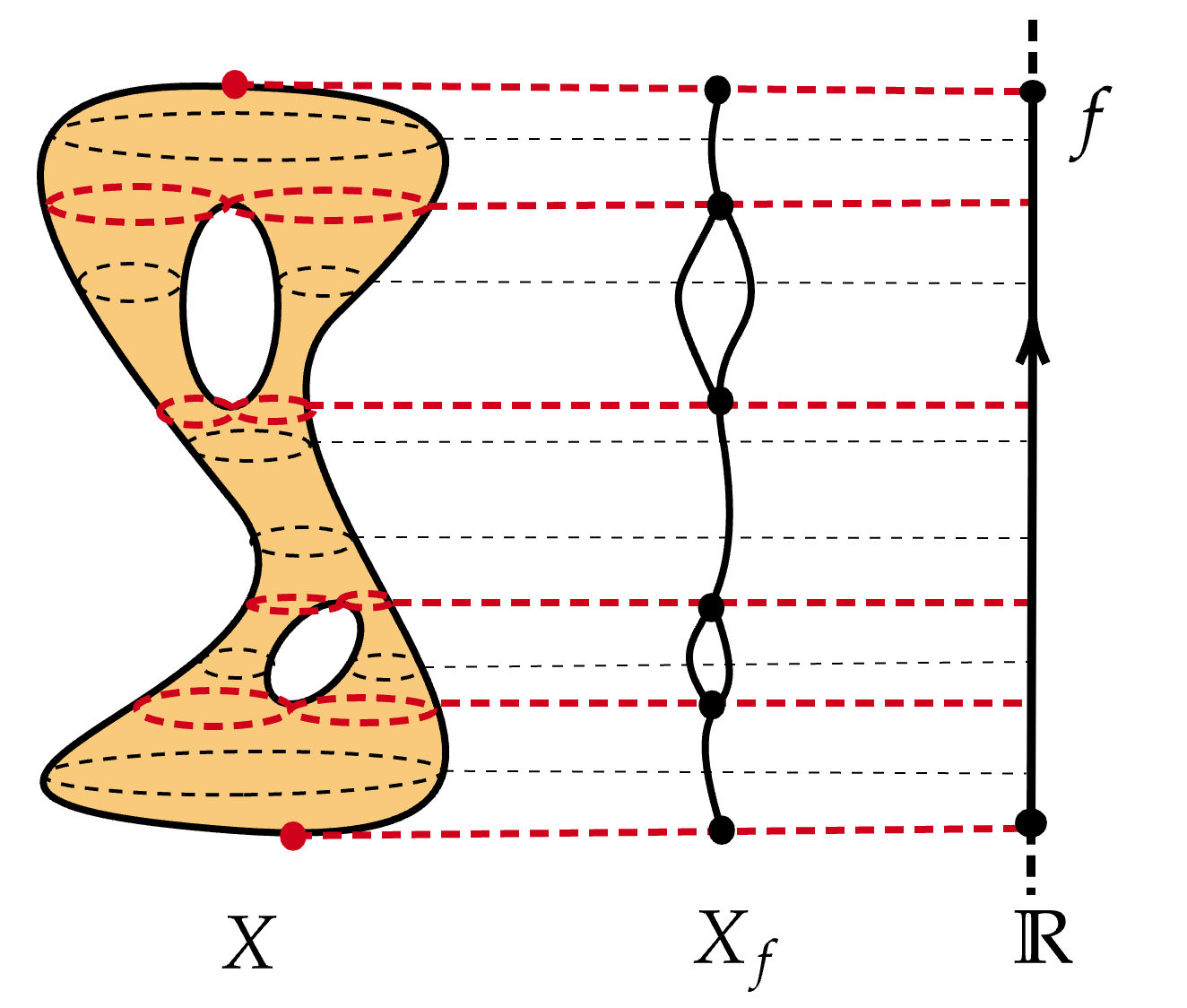}
    \caption{The Reeb graph of a torus of genus \(2\): the function used to compute the graph is the height function.}
    \label{fig:reeb}
\end{figure}

\begin{oss}
In general, the quotient \(X/\!\sim_f\) can be defined for an arbitrary topological space \(X\). The restriction to constructible \(\mathbb{R}\)-spaces ensures that the quotient inherits the structure of a combinatorial graph.
\end{oss}

Reeb graphs arising from constructible \(\mathbb{R}\)-spaces form a category, denoted \(\mathbf{Reeb}\). It is a full subcategory of \(\mathbb{R}\mathbf{-Top}^c\): objects of \(\mathbf{Reeb}\) are objects of \(\mathbb{R}\mathbf{-Top}^c\), and morphisms are the same. The Reeb construction defines a functor
\[
\mathfrak R\colon \mathbb{R}\mathbf{-Top}^c\longrightarrow\mathbf{Reeb}.
\]
Indeed, a morphism \(\varphi\colon(X,f)\to(Y,g)\) induces a unique morphism
\[
\mathfrak R(\varphi)\colon\mathfrak R(X,f)\longrightarrow\mathfrak R(Y,g)
\]
such that
\[
\mathfrak R(\varphi)\circ q_X=q_Y\circ\varphi,
\]
where
\[
q_X\colon X\longrightarrow X/\!\sim_f,
\qquad
q_Y\colon Y\longrightarrow Y/\!\sim_g
\]
are the corresponding Reeb quotient maps.

Throughout the paper, if \(R=(G,h_R)\) is a Reeb graph, we routinely identify
\(R\) with its underlying topological graph \(G\) whenever it occurs in a
topological construction. In particular, we use expressions such as
\(R_{\le t}\), \(H_k(R)\), and \(R\times[-\delta,\delta]\). The same convention
applies to \(\mathfrak R(X,f)=(X/\!\sim_f,\bar f)\).

In~\cite{de_Silva_2016} the authors also introduced the notion of a constructible cosheaf, which we will use throughout. Let \(\mathbf{Int}\) denote the category whose objects are open intervals in \(\mathbb{R}\) and whose morphisms are inclusions.

\begin{definition}
Consider \(\mathcal{F}\in\mathbf{Set}^{\mathbf{Int}}\), i.e.\ a functor \(\mathcal{F}\colon \mathbf{Int}\to \mathbf{Set}\), and let \(S\coloneqq \{a_0<a_1<\cdots<a_n\}\subseteq \mathbb{R}\) be a finite set. We say that \(\mathcal{F}\) is a constructible cosheaf (with respect to \(S\)) if:
\begin{itemize}
    \item for every open interval \(I\in \mathbf{Int}\) and every cover \(\mathcal{U}\) of \(I\) that is closed under finite intersections,
    \[
        \mathcal{F}(I)\;\cong\;\varinjlim_{U\in \mathcal{U}} \mathcal{F}(U);
    \]
    \item if \(I\subseteq J\) are open intervals and \(I\cap S = J\cap S\), then the map \(\mathcal{F}(I\subseteq J)\) is an isomorphism;
    \item if \(I\subseteq (-\infty,a_0)\) or \(I\subseteq (a_n,+\infty)\), then \(\mathcal{F}(I)=\emptyset\).
\end{itemize}
\end{definition}

We call \(S\) the set of critical values. We write \(\mathbf{Csh}^c\) for the category of constructible cosheaves, with morphisms given by natural transformations.

A fundamental example is the Reeb cosheaf~\citep{de_Silva_2016}. Given a constructible \(\mathbb{R}\)-space \(X\) with structure map \(f\colon X\to \mathbb{R}\), the associated Reeb cosheaf \(\mathcal{F}_{(X,f)}\) is defined by
\[
\mathcal{F}_{(X,f)}(I)\coloneqq \pi_0\bigl(f^{-1}(I)\bigr),
\qquad
\mathcal{F}_{(X,f)}(I\subseteq J)\coloneqq \pi_0\bigl(f^{-1}(I)\subseteq f^{-1}(J)\bigr),
\]
for every inclusion \(I\subseteq J\) of intervals. For a Reeb graph \(R=(G,h_R)\), we write
\[
\mathcal F_R(I)\coloneqq\pi_0\bigl(h_R^{-1}(I)\bigr)
\]
for its Reeb cosheaf.

In~\cite{de_Silva_2016} it is shown that the functor
\begin{align*}
\mathcal C\colon \mathbf{Reeb} &\longrightarrow \mathbf{Csh}^c,\\
R &\longmapsto \mathcal{F}_{R}
\end{align*}
is an equivalence of categories.  We denote by
\[
\mathcal D\colon \mathbf{Csh}^c \longrightarrow \mathbf{Reeb}
\]
a quasi-inverse, namely the display locale functor. In this sense, Reeb graphs admit equivalent geometric and algebraic descriptions.

Reeb quotients of constructible spaces also enjoy the following nice property.

\begin{lem}[Reeb quotients preserve path components]
\label{lem:reeb_quotient_pi0}
Let \(f\colon X\to\mathbb R\) be a constructible $\R$-space, set
$R\coloneqq\mathfrak R(X,f),
$
and let \(q\colon X\to R\) be the Reeb quotient map. Then \(q\) induces a bijection
\[
\pi_0(q)\colon \pi_0(X)\longrightarrow \pi_0(R).
\]
The same holds after restricting \(f\) to any constructible subspace \(Y\subseteq X\).
\end{lem}

\begin{proof}
Since \(q\) is continuous, \(\pi_0(q)\) is well-defined. Surjectivity follows
from the surjectivity of \(q\).

For injectivity, let \(x,x'\in X\), and suppose that \(q(x)\) and \(q(x')\)
lie in the same path component of \(R\). Choose a path
\(\alpha\colon[0,1]\to R\) from \(q(x)\) to \(q(x')\). Since \(R\) is a finite graph, there is a path from \(q(x)\) to \(q(x')\)
which is a concatenation of finitely many pieces, each contained either in a
single edge of \(R\) or in a vertex.

If such a subpath is contained in an edge, then the edge lies over an interval
of regular values. Over the corresponding regular band, constructibility
identifies the relevant component of \(X\) with a product over that interval.
Thus the subpath lifts to a path in the corresponding component of \(X\), once
its initial point in the appropriate fiber has been chosen. If a subpath is
contained in a vertex \(v\), then \(q^{-1}(v)\) is a path-connected component of
the level set \(f^{-1}(\bar f(v))\), so any two lifted endpoints over \(v\) can
be joined inside \(q^{-1}(v)\).

Starting from \(x\), we lift the subdivided subpaths successively. At each
vertex, we join the endpoint of the previous lifted subpath to the chosen
initial point of the next one inside the corresponding vertex fiber. The
concatenation gives a path in \(X\) from \(x\) to \(x'\). Hence \(x\) and
\(x'\) lie in the same path component of \(X\), proving injectivity.
\end{proof}

Applying \Cref{lem:reeb_quotient_pi0} to the constructible subspaces
\(f^{-1}(I)\subseteq X\), naturally with respect to inclusions of intervals,
gives a canonical isomorphism of constructible cosheaves
\begin{equation}\label{eq:cosheaf_identity}
\mathcal F_{(X,f)}
\cong
\mathcal F_{\mathfrak R(X,f)}.    
\end{equation}

We now recall the smoothing operation on both sides of this equivalence.

\begin{definition}
Let \(\mathcal G\in \mathbf{Csh}^c\) and \(\varepsilon\ge 0\). For an interval \(I=(a,b)\), set
\[
I^\varepsilon\coloneqq (a-\varepsilon,b+\varepsilon).
\]
The \emph{\(\varepsilon\)-smoothing} of \(\mathcal G\) is the constructible cosheaf
\[
\mathcal S_\varepsilon(\mathcal G)\colon \mathbf{Int}\to\mathbf{Set},
\qquad
\mathcal S_\varepsilon(\mathcal G)(I)\coloneqq \mathcal G(I^\varepsilon),
\]
with structure maps induced by inclusions. There is a canonical natural transformation
\[
\mathcal G\Rightarrow \mathcal S_\varepsilon(\mathcal G),
\]
given intervalwise by the structure maps \(\mathcal G(I)\to \mathcal G(I^\varepsilon)\).
\end{definition}

If \(\mathcal D(\mathcal G)=R=(G,h_R)\), then
\(\mathcal S_\varepsilon(\mathcal G)\) is geometrically realized by the Reeb
graph of the constructible \(\mathbb R\)-space
\(G\times[-\varepsilon,\varepsilon]\) equipped with the function
\[
H_\varepsilon(x,t)\coloneqq h_R(x)+t.
\]

This motivates the corresponding smoothing operation on Reeb graphs.

\begin{definition}
Let \(R=(G,h_R)\in\mathbf{Reeb}\) and let \(\varepsilon\ge 0\). The \emph{\(\varepsilon\)-smoothing} of \(R\) is the Reeb graph
\[
U_\varepsilon(R)\coloneqq \mathcal D\!\bigl(\mathcal S_\varepsilon(\mathcal C(R))\bigr).
\]
Equivalently, \(U_\varepsilon(R)\) is the Reeb graph of the constructible
\(\mathbb R\)-space \(G\times[-\varepsilon,\varepsilon]\) endowed with the
function \(H_\varepsilon(x,t)=h_R(x)+t\). The canonical natural transformation \(\mathcal C(R)\Rightarrow \mathcal S_\varepsilon(\mathcal C(R))\) corresponds, under the equivalence of categories, to a canonical morphism
\[
\zeta_R^\varepsilon\colon R\to U_\varepsilon(R).
\]
\end{definition}

We can now define the interleaving distance on cosheaves and, equivalently, on Reeb graphs.

\begin{definition}\label{def:interleaving}
Let \(\mathcal{F},\mathcal{G}\in \mathbf{Csh}^c\) and let \(\varepsilon\ge 0\). An \emph{\(\varepsilon\)-interleaving} between \(\mathcal F\) and \(\mathcal G\) consists of natural transformations
\[
\varphi\colon \mathcal F\Rightarrow \mathcal S_\varepsilon(\mathcal G),
\qquad
\psi\colon \mathcal G\Rightarrow \mathcal S_\varepsilon(\mathcal F),
\]
such that the two composites
\[
\mathcal F \xRightarrow{\ \varphi\ } \mathcal S_\varepsilon(\mathcal G)
\xRightarrow{\ \mathcal S_\varepsilon(\psi)\ } \mathcal S_{2\varepsilon}(\mathcal F)
\]
and
\[
\mathcal G \xRightarrow{\ \psi\ } \mathcal S_\varepsilon(\mathcal F)
\xRightarrow{\ \mathcal S_\varepsilon(\varphi)\ } \mathcal S_{2\varepsilon}(\mathcal G)
\]
are the canonical natural transformations induced by the structure maps \(I\subseteq I^{2\varepsilon}\).

The \emph{interleaving distance} between \(\mathcal F\) and \(\mathcal G\) is then defined by
\[
d_I(\mathcal F,\mathcal G)\coloneqq
\inf\Bigl\{\varepsilon\ge 0\ \big|\ \text{there exists an \(\varepsilon\)-interleaving between \(\mathcal F\) and \(\mathcal G\)}\Bigr\}.
\]
\end{definition}

Equivalently, an \(\varepsilon\)-interleaving between Reeb graphs \(R\) and \(S\) consists of morphisms
\[
\Phi\colon R\to U_\varepsilon(S),
\qquad
\Psi\colon S\to U_\varepsilon(R),
\]
such that
\[
U_\varepsilon(\Psi)\circ \Phi=\zeta_R^{2\varepsilon},
\qquad
U_\varepsilon(\Phi)\circ \Psi=\zeta_S^{2\varepsilon}.
\]
Under the equivalence between \(\mathbf{Reeb}\) and \(\mathbf{Csh}^c\), this geometric notion is exactly equivalent to the cosheaf-theoretic one.

In~\cite{de_Silva_2016} it is proved that \(d_I\) defines a metric on \(\mathbf{Csh}^c\) up to isomorphism. By transport through the equivalence of categories, we use the same notation \(d_I\) for the induced interleaving distance on Reeb graphs.

\subsection{Mapper graphs and covers}\label{sec:mapper}
Mapper graphs were introduced in \cite{singh} as modifications of the Reeb graphs to the setting of finite metric spaces (such as, e.g., point clouds).
We now recall the standard construction.
Let $X$ be a topological space and let $S_n\subset X$ be a point cloud equipped with pairwise dissimilarities. Given a (filter) function $f\colon X\to \mathbb{R}$, we evaluate $f$ on the points of $S_n$ and build a graph from the induced cover-and-cluster algorithm described below.

\begin{algorithm}[H]
    \caption{Mapper graph algorithm}
    \begin{algorithmic}
        \State $\bullet$ Cover the set of values $f(S_n)$ with intervals $I_1,\ldots,I_S$ such that $I_s\cap I_{s'}\neq\emptyset$ if and only if $|s-s'|=1$;
        \State $\bullet$ For each $s\in\{1,\ldots,S\}$, apply a clustering algorithm to the sample preimage $S_n\cap f^{-1}(I_s)$, thereby obtaining clusters
        \[
        Q_{s,1},\ldots,Q_{s,k_s};
        \]
        the resulting family
        \[
        \mathcal{Q}\coloneqq \{Q_{1,1},\ldots,Q_{1,k_1},\ldots,Q_{S,1},\ldots,Q_{S,k_S}\}
        \]
        is a pullback cover of $S_n$;
        \State $\bullet$ Build the Mapper graph as the \emph{nerve} of the pullback cover, i.e.: each cluster $Q_{s,k}$ corresponds to a vertex $v_{s,k}$, and two vertices $v_{s,k}$ and $v_{s',k'}$ are joined by an edge if and only if $Q_{s,k}\cap Q_{s',k'}\neq \emptyset$.
    \end{algorithmic}
\end{algorithm}

We now recall some definitions and results from~\cite{brown2020probabilisticconvergencestabilityrandom}, where the authors give a categorified version of the Mapper construction. Let $\mathcal{U}$ be an open cover of $\mathbb{R}$. Using $\mathcal{U}$, one defines an endofunctor of $\mathbf{Int}$ as follows. For $I\in \mathbf{Int}$ set
\[
\mathcal{I}_{\mathcal{U}}(I)\coloneqq \bigcup_{x\in I}\ \bigcap_{V\in \mathcal{U}_x} V,
\qquad
\text{where }\ \mathcal{U}_x \coloneqq \{\,V\in \mathcal{U}\mid x\in V\,\}.
\]
Intuitively, $\mathcal{I}_{\mathcal{U}}$ enlarges an interval by taking its closure with respect to the cover $\mathcal{U}$.

The associated Mapper transformation is defined by precomposition with $\mathcal{I}_{\mathcal{U}}$:
\[
\mathcal{M}_{\mathcal{U}}\colon \mathbf{Set}^{\mathbf{Int}} \longrightarrow \mathbf{Set}^{\mathbf{Int}},
\qquad
\mathcal{M}_{\mathcal{U}}(\mathcal{C}) \coloneqq \mathcal{C}\circ \mathcal{I}_{\mathcal{U}},
\]
that is,
\[
\mathcal{M}_{\mathcal{U}}(\mathcal{C})(I) \coloneqq \mathcal{C}\bigl(\mathcal{I}_{\mathcal{U}}(I)\bigr),
\qquad
I\in \mathbf{Int}.
\]

\begin{oss}
In~\cite{brown2020probabilisticconvergencestabilityrandom} this construction is referred to as the \emph{Mapper functor}. In our context, where the objects of interest are (constructible) cosheaves, this terminology may be slightly misleading: depending on the level of abstraction, one may refer either to the endofunctor $\mathcal{M}_{\mathcal{U}}\colon \mathbf{Set}^{\mathbf{Int}}\to \mathbf{Set}^{\mathbf{Int}}$ itself, or to its value $\mathcal{M}_{\mathcal{U}}(\mathcal{C})$ on a given cosheaf $\mathcal{C}$, as “the Mapper functor”. To avoid this ambiguity, we will refer to $\mathcal{M}_{\mathcal{U}}$ as the Mapper transformation.
\end{oss}

We now recall a result from~\cite{brown2020probabilisticconvergencestabilityrandom} that will be used in the next section.

\begin{prop}[Adapted from Proposition~4 and Theorem~1 in~\cite{brown2020probabilisticconvergencestabilityrandom}]\label{brobro}
Let $\mathcal{U}$ be a finite nice open cover of $\mathbb{R}$ (i.e., all non-empty finite intersections of elements of $\mathcal{U}$ are contractible, and $\mathcal{U}$ is locally finite). Then the Mapper transformation
$\mathcal{M}_{\mathcal{U}}\colon \mathbf{Set}^{\mathbf{Int}}\to \mathbf{Set}^{\mathbf{Int}}$
restricts to a functor from the category of cosheaves to the category $\mathbf{Csh}^c$ of constructible cosheaves.
Moreover, for any cosheaf $\mathcal{C}$, the set of critical values of $\mathcal{M}_{\mathcal{U}}(\mathcal{C})$ is contained in the set of boundary points of the open sets in $\mathcal{U}$, and
\begin{equation}\label{eq1}
\textstyle
d_I\!\bigl(\mathcal{C},\, \mathcal{M}_{\mathcal{U}}(\mathcal{C})\bigr)
\leq \res_{\mathcal{C}}(\mathcal{U}),
\end{equation}
where
\begin{equation*}
\textstyle
\res_{\mathcal{C}}(\mathcal{U})\coloneqq \max\{\diam(V)\mid V\in \mathcal{U}_{\mathcal{C}}\},
\qquad
\mathcal{U}_{\mathcal{C}} \coloneqq \{V\in \mathcal{U}\mid \mathcal{C}(V)\neq \emptyset\,\}.
\end{equation*}
\end{prop}

\subsection{Positive reach and metric equivalence between Euclidean and intrinsic distances}\label{sec:metric_equivalence}

In this subsection we recall the metric-distortion result of~\cite{boissonnatReachMetricDistortion2019} for closed sets of positive reach.
This will later allow us to convert Euclidean covering and proximity scales into
intrinsic ones.

Let \(X\subset \R^m\) be a closed set. Its \emph{medial axis} is the set of
points of \(\R^m\) having more than one closest point in \(X\). We write
\[
\dist(z,X)\coloneqq \inf_{x\in X}\|z-x\|
\]
for the Euclidean distance from \(z\in\R^m\) to \(X\). The \emph{reach} of
\(X\)~\cite{federer1959curvature} is then defined by
\[
\rch(X)\coloneqq
\inf\{\dist(z,X)\mid z \text{ belongs to the medial axis of }X\}.
\]

Following~\cite{boissonnatReachMetricDistortion2019}, for a closed set \(X\subset \R^m\) and \(x,y\in X\), we denote by \(d_X(x,y)\) the intrinsic distance in \(X\), namely
\begin{equation}\label{eq:d_X}
d_X(x,y)\coloneqq \inf\{L(\gamma)\mid \gamma:[0,1]\to X \text{ is a continuous path with }
\gamma(0)=x,\ \gamma(1)=y\},
\end{equation}
with the convention that \(d_X(x,y)=+\infty\) if there is no path in \(X\) joining \(x\) to \(y\).
Here the length of a path \(\gamma:[0,1]\to \R^m\) is
\[
L(\gamma)\coloneqq \sup\Bigl\{\sum_{i=1}^k \|\gamma(t_i)-\gamma(t_{i-1})\| \;\Bigm|\;
0=t_0<\cdots<t_k=1,\ k\in\N\Bigr\},
\]
possibly equal to \(+\infty\).
As recalled in~\cite{boissonnatReachMetricDistortion2019}, whenever there exists at least one path in \(X\) from \(x\) to \(y\) of finite length, a minimizing geodesic exists.

We need the following lemma.

\begin{lem}\label{lem:intrinsic_balls_closed}
Let \(X\subset \mathbb R^m\) be compact, and suppose that the intrinsic distance
\(d_X\), defined in~\Cref{eq:d_X}, is finite on \(X\times X\). Then, for every
\(s\in X\) and every \(\delta\ge 0\), the set
\[
B_{d_X}(s,\delta)
=
\{x\in X\mid d_X(s,x)\le \delta\}
\]
is closed in \(X\) with respect to the subspace topology.
\end{lem}

\begin{proof}
We first prove that \(d_X(s,\cdot)\) is lower semicontinuous with respect to the
subspace topology on \(X\). 
That is, we show that:
\[
d_X(s,x)\le
\liminf_{n\to\infty}d_X(s,x_n)
\]
for any \(x_n\to x\) in \(X\).

Let \(x_n\to x\) in \(X\), where convergence is with
respect to the topology inherited from \(\mathbb R^m\). If
\[
\liminf_{n\to\infty} d_X(s,x_n)=+\infty,
\]
there is nothing to prove. Otherwise, after passing to a subsequence realizing
the liminf, we may suppose that
\[
d_X(s,x_n)\to \lambda_\ast
=
\liminf_{n\to\infty} d_X(s,x_n)
<+\infty.
\]
For every \(n\), let
\[
\gamma_n\colon[0,1]\to X
\]
be a minimizing \(d_X\)-geodesic from \(s\) to \(x_n\), parametrized
proportionally to arclength. Thus
\[
L(\gamma_n)=d_X(s,x_n).
\]
In particular, the lengths \(L(\gamma_n)\) are uniformly bounded. Since each \(\gamma_n\) is parametrized proportionally to arclength, for all
\(0\le s\le t\le 1\) we have
\[
L\!\left(\gamma_n|_{[s,t]}\right)=(t-s)L(\gamma_n).
\]
Consequently,
\[
\|\gamma_n(t)-\gamma_n(s)\|
\le
L\!\left(\gamma_n|_{[s,t]}\right)
=
(t-s)L(\gamma_n)
\le
C|t-s|,
\]
where \(C\) is a uniform upper bound for the lengths \(L(\gamma_n)\). Thus the
family \(\{\gamma_n\}_n\) is uniformly \(C\)-Lipschitz as a family of maps into
\(\mathbb R^m\), and in particular it is equicontinuous.
Since \(X\) is compact, the images of the
\(\gamma_n\)'s are contained in a fixed compact subset of \(\mathbb R^m\). By the
Arzelà--Ascoli theorem, after passing to a further subsequence, the paths
\(\gamma_n\) converge uniformly to a continuous path
\[
\gamma\colon[0,1]\to X.
\]
Moreover,
\[
\gamma(0)=s,
\qquad
\gamma(1)=x,
\]
because \(\gamma_n(0)=s\), \(\gamma_n(1)=x_n\), and \(x_n\to x\) in the subspace
topology. Let
\[
0=t_0<t_1<\cdots<t_k=1
\]
be any partition of \([0,1]\). Since \(\gamma_n\to\gamma\) uniformly, we have
\[
\sum_{i=1}^k
\|\gamma(t_i)-\gamma(t_{i-1})\|
=
\lim_{n\to\infty}
\sum_{i=1}^k
\|\gamma_n(t_i)-\gamma_n(t_{i-1})\|.
\]
For every \(n\), the sum on the right is bounded above by \(L(\gamma_n)\). Hence
\[
\sum_{i=1}^k
\|\gamma(t_i)-\gamma(t_{i-1})\|
\le
\liminf_{n\to\infty}L(\gamma_n).
\]
Taking the supremum over all partitions of \([0,1]\), we obtain
\[
L(\gamma)\le \liminf_{n\to\infty}L(\gamma_n)=
\lambda_\ast.
\]

Therefore
\[
d_X(s,x)\le L(\gamma)\le \lambda_\ast
=
\liminf_{n\to\infty}d_X(s,x_n).
\]
Hence \(d_X(s,\cdot)\) is lower semicontinuous with respect to the subspace
topology on \(X\).

Now let \(x_n\in B_{d_X}(s,\delta)\) and suppose that \(x_n\to x\) in \(X\) with
respect to the subspace topology. By lower semicontinuity,
\[
d_X(s,x)
\le
\liminf_{n\to\infty}d_X(s,x_n)
\le
\delta.
\]
Thus \(x\in B_{d_X}(s,\delta)\). Since \(X\) is metrizable as a subspace of
\(\mathbb R^m\), sequential closedness is equivalent to closedness, and therefore
\(B_{d_X}(s,\delta)\) is closed in \(X\).
\end{proof}


The key metric result for us is the following theorem of~\cite{boissonnatReachMetricDistortion2019}, which characterizes the reach in terms of metric distortion.

\begin{thm}[Theorem~1 in~\cite{boissonnatReachMetricDistortion2019}]\label{thm:metric_distortion_reach}
Let \(X\subset \R^m\) be a closed set. Then
\[
\rch(X)=\sup\Bigl\{r>0 \ \Bigm| \ \forall x,y\in X,\ \|x-y\|<2r
\Longrightarrow d_X(x,y)\le 2r\arcsin\!\Bigl(\frac{\|x-y\|}{2r}\Bigr)\Bigr\},
\]
with the convention that the supremum of the empty set is \(0\).
\end{thm}

In particular, if \(X\) has positive reach \(\tau=\rch(X)>0\), then for every \(x,y\in X\) with \(\|x-y\|<2\tau\),
\begin{equation}\label{eq:metric_distortion_tau}
d_X(x,y)\le 2\tau\arcsin\!\Bigl(\frac{\|x-y\|}{2\tau}\Bigr).
\end{equation}
For later use, define
\[
\psi_\tau(\delta)
\coloneqq
2\tau\arcsin\!\left(\frac{\delta}{2\tau}\right),
\qquad
0\le \delta<2\tau.
\]
Then \eqref{eq:metric_distortion_tau} reads
\[
d_X(x,y)\le \psi_\tau(\|x-y\|)
\qquad
\text{whenever }\|x-y\|<2\tau.
\]

Since every path in \(X\) joining \(x\) and \(y\) has length at least \(\|x-y\|\), one also always has
\begin{equation}\label{eq:eucl_lower_intrinsic}
\|x-y\|\le d_X(x,y).
\end{equation}

For the scale estimates needed later in the paper, it is convenient to remove the explicit dependence on the reach, provided that one works below the reach scale.

\begin{cor}\label{cor:metric_equivalence_no_reach}
Let \(X\subset \R^m\) be a closed set with positive reach, and let \(\delta>0\) satisfy
\[
\delta\le \rch(X).
\]
Then for every \(x,y\in X\) with \(\|x-y\|\le \delta\), one has
\[
\|x-y\|\le d_X(x,y)\le 2\delta\arcsin\!\Bigl(\frac{\|x-y\|}{2\delta}\Bigr)\le \frac{\pi}{3}\,\|x-y\|.
\]
In particular, on pairs of points of \(X\) at Euclidean distance at most \(\delta\), the Euclidean and intrinsic distances are bi-Lipschitz equivalent with constants \(1\) and \(\pi/3\).
\end{cor}

\begin{proof}
Fix \(x,y\in X\) with \(\|x-y\|\le \delta\le \rch(X)\). Then clearly
\[
\|x-y\|<2\rch(X),
\]
so \Cref{thm:metric_distortion_reach} gives
\[
d_X(x,y)\le 2\rch(X)\arcsin\!\Bigl(\frac{\|x-y\|}{2\rch(X)}\Bigr).
\]
As observed in~\cite{boissonnatReachMetricDistortion2019}, for fixed \(s\ge 0\) the function
\[
r\longmapsto 2r\arcsin\!\Bigl(\frac{s}{2r}\Bigr)
\]
is decreasing on \((s/2,+\infty)\). Applying this with \(s=\|x-y\|\le \delta\) and using \(\delta\le \rch(X)\), we obtain
\[
d_X(x,y)\le 2\delta\arcsin\!\Bigl(\frac{\|x-y\|}{2\delta}\Bigr).
\]
Moreover, since \(\|x-y\|/(2\delta)\in[0,1/2]\) and the function \(u\mapsto \arcsin(u)/u\) is increasing on \((0,1)\), we have
\[
\arcsin(u)\le \frac{\pi}{3}u,\qquad u\in[0,1/2].
\]
Substituting \(u=\|x-y\|/(2\delta)\) yields
\[
2\delta\arcsin\!\Bigl(\frac{\|x-y\|}{2\delta}\Bigr)
\le \frac{\pi}{3}\,\|x-y\|.
\]
The lower bound \(\|x-y\|\le d_X(x,y)\) is \Cref{eq:eucl_lower_intrinsic}.
\end{proof}

\section{Comparing the bottleneck pseudometric \texorpdfstring{\(d_\Delta\)}{dDelta} with the interleaving distance \texorpdfstring{\(d_I\)}{dI}}
\label{sec:dDelta_le_2dI}

Throughout this section, let
\[
R=(G,h_R),
\qquad
S=(G',h_S)
\]
be finite Reeb graphs. By the convention fixed in the background, \(R\) and
\(S\) also denote their underlying topological graphs whenever they occur in
topological constructions. Note that \(h_R\) and \(h_S\) are monotone on every open edge of the corresponding
Reeb graph. Via the equivalence between \(\mathbf{Reeb}\) and
\(\mathbf{Csh}^c\), we use interchangeably the geometric smoothing notation
\(U_\varepsilon\) on Reeb graphs and the shift/smoothing notation
\(\mathcal S_\varepsilon\) on constructible cosheaves.

For a Reeb graph \(R\), we denote by
\[
\Ord_0(R),\qquad \Ext_0(R),\qquad \Rel_1(R),\qquad \Ext_1(R)
\]
the four components of its extended-persistence signature. See \cite{JMLR:v19:17-291} for more details. We write \(d_B\) for
the bottleneck distance between persistence diagrams and define
\[
d_\Delta(R,S)
\coloneqq
\max\left\{
\begin{array}{l}
d_B\!\bigl(\Ord_0(R),\Ord_0(S)\bigr),\\[0.3em]
d_B\!\bigl(\Ext_0(R),\Ext_0(S)\bigr),\\[0.3em]
d_B\!\bigl(\Rel_1(R),\Rel_1(S)\bigr),\\[0.3em]
d_B\!\bigl(\Ext_1(R),\Ext_1(S)\bigr).
\end{array}
\right\}.
\]
Equivalently, using the Reeb--cosheaf equivalence, we use the same notation
for the induced pseudometric on the associated constructible cosheaves.

The goal of this section is to compare \(d_\Delta\) with the interleaving
distance \(d_I\).

As a first step, we reduce the comparison to the case of path-connected Reeb
graphs.

\begin{prop}
\label{prop:interleaving_components}
Let
\[
R=\bigsqcup_{i=1}^m R_i,
\qquad
S=\bigsqcup_{j=1}^n S_j
\]
be the decompositions of two finite Reeb graphs into path-connected components.
Then
\[
d_I(R,S)
=
\begin{cases}
\displaystyle
\min_{\sigma:\{1,\dots,m\}\xrightarrow{\cong}\{1,\dots,n\}}\;
\max_{1\le i\le m} d_I(R_i,S_{\sigma(i)}),
& m=n,\\[1.2em]
+\infty,
& m\ne n.
\end{cases}
\]
Equivalently, one may allow matchings with the empty Reeb graph by setting
\(d_I(\varnothing,\varnothing)=0\) and
\(d_I(R_i,\varnothing)=+\infty\) for every non-empty component \(R_i\).
\end{prop}

\begin{proof}
Smoothing preserves path-connected components:
\[
U_\varepsilon(R)\cong \bigsqcup_{i=1}^m U_\varepsilon(R_i),
\qquad
U_\varepsilon(S)\cong \bigsqcup_{j=1}^n U_\varepsilon(S_j).
\]

Since each \(R_i\) is path-connected, its image under a morphism
\(\Phi\colon R\to U_\varepsilon(S)\) is path-connected. Hence it is contained in
a unique path-connected component of \(U_\varepsilon(S)\). Using the
decomposition above, there is therefore a unique index
\[
\alpha_\Phi(i)\in \{1,\dots,n\}
\]
such that
\[
\Phi(R_i)\subseteq U_\varepsilon(S_{\alpha_\Phi(i)}).
\]
Similarly, a morphism \(\Psi\colon S\to U_\varepsilon(R)\) determines a unique
map
\[
\beta_\Psi\colon \{1,\dots,n\}\longrightarrow \{1,\dots,m\}
\]
by the condition
\[
\Psi(S_j)\subseteq U_\varepsilon(R_{\beta_\Psi(j)}).
\]

Assume now that \(\Phi\) and \(\Psi\) form an \(\varepsilon\)-interleaving. For
each \(i\), the composite
\[
U_\varepsilon(\Psi)\circ \Phi\colon
R_i\longrightarrow U_{2\varepsilon}(R)
\]
has image contained in
\[
U_{2\varepsilon}\!\left(R_{\beta_\Psi(\alpha_\Phi(i))}\right),
\]
by the definitions of \(\alpha_\Phi\) and \(\beta_\Psi\). On the other hand, the
interleaving identity gives
\[
U_\varepsilon(\Psi)\circ\Phi=\zeta_R^{2\varepsilon}.
\]
Since the smoothing map \(\zeta_R^{2\varepsilon}\) sends \(R_i\) into
\(U_{2\varepsilon}(R_i)\), and since the components of \(U_{2\varepsilon}(R)\)
are precisely the \(U_{2\varepsilon}(R_k)\), we must have
\[
\beta_\Psi(\alpha_\Phi(i))=i.
\]
Thus \(\beta_\Psi\circ\alpha_\Phi=\id_{\{1,\dots,m\}}\). The same argument applied
to
\[
U_\varepsilon(\Phi)\circ\Psi=\zeta_S^{2\varepsilon}
\]
shows that
\[
\alpha_\Phi\circ\beta_\Psi=\id_{\{1,\dots,n\}}.
\]
Therefore \(\alpha_\Phi\) and \(\beta_\Psi\) are inverse bijections between the
component sets. In particular, if \(m\ne n\), no finite interleaving exists.

When \(m=n\), restricting an \(\varepsilon\)-interleaving of \(R\) and \(S\) to the
matched components gives an \(\varepsilon\)-interleaving between
\(R_i\) and \(S_{\alpha(i)}\) for every \(i\). Conversely, the disjoint union of
componentwise \(\varepsilon\)-interleavings along any bijection
\(\sigma:\{1,\dots,m\}\xrightarrow{\cong}\{1,\dots,n\}\) gives an
\(\varepsilon\)-interleaving between \(R\) and \(S\). Taking infima over
\(\varepsilon\) and then minimizing over \(\sigma\) gives the formula.

\end{proof}

As a consequence of \Cref{prop:interleaving_components}, any time we compare Reeb graphs with a different number of path-connected components, we always have $d_\Delta < d_I$, as $d_\Delta$ is always finite. Moreover, the case where the number of path-connected components is equal but  bigger than $1$, can be reduced to a comparison of path-connected Reeb graphs.
Hence, in the rest of this section we assume that \(R\) and \(S\) are
path-connected.

The comparison has two different strengths. The diagrams
governed by connected components, namely \(\Ord_0\), \(\Ext_0\), and \(\Rel_1\),
are controlled with the sharp constant \(1\). More precisely, we prove in this
section that every \(\varepsilon\)-interleaving between \(R\) and \(S\) induces
bounds
\[
\max\left\{
\begin{array}{l}
d_B\!\bigl(\Ord_0(R),\Ord_0(S)\bigr),\\
d_B\!\bigl(\Ext_0(R),\Ext_0(S)\bigr),\\
d_B\!\bigl(\Rel_1(R),\Rel_1(S)\bigr)
\end{array}
\right\}
\le \varepsilon.
\]
The loop component \(\Ext_1\) is different. Its proof requires a separate
argument, where we
show that every \(\varepsilon\)-interleaving gives
\[
d_B\!\bigl(\Ext_1(R),\Ext_1(S)\bigr)\le 2\varepsilon.
\]
Combining the estimates yields the global comparison
\[
d_\Delta(R,S)\le 2d_I(R,S).
\]

The factor \(2\) in the loop component is not merely an artifact of the proof.
The interleaving identities control the composite of two \(\varepsilon\)-smoothing
maps, and hence a loop may disappear only after \(2\varepsilon\)-smoothing. Since
smoothing shrinks an \(\Ext_1\)-bar from both ends, a \(2\varepsilon\)-smoothing
moves its upper endpoint down by \(2\varepsilon\) and its lower endpoint up by
\(2\varepsilon\). Thus a loop killed at this stage can have persistence
\(4\varepsilon\), and therefore diagonal bottleneck cost \(2\varepsilon\). In this
sense, the interleaving can eliminate the loop through two successive smoothing
steps, while the bottleneck distance pays for deleting the corresponding point in a
single diagonal match.
This two-sided pinching phenomenon does not occur for path-connected components.

\subsection{Smoothing and persistence modules}
\label{subsec:smoothing_persistence_modules}

We first record the precise smoothing facts used below. Let \(R\) be a Reeb
graph and let \(\delta\ge 0\). Recall that
\[
U_\delta(R)=
\mathcal D\!\bigl(\mathcal S_\delta(\mathcal C(R))\bigr).
\]
Set
\[
T:=R\times[-\delta,\delta],
\qquad
H(x,u):=h_R(x)+u.
\]
Then \(U_\delta(R)\) is geometrically realized as the Reeb graph of \(H\). We write
\[
q\colon T\longrightarrow U_\delta(R)
\]
for the Reeb quotient map and \(h_\delta\colon U_\delta(R)\to\mathbb R\) for
the induced function, so that
\[
h_\delta\circ q=H.
\]
The canonical natural transformation
\(\mathcal C(R)\Rightarrow \mathcal S_\delta(\mathcal C(R))\) corresponds to the
canonical smoothing morphism
\[
\zeta_R^\delta\colon R\longrightarrow U_\delta(R),
\qquad
\zeta_R^\delta=q\circ i,
\qquad
 i(x)=(x,0).
\]
Smoothing is functorial: a morphism \(F\colon R\to S\) induces
\[
U_\delta(F)\colon U_\delta(R)\longrightarrow U_\delta(S),
\]
and the square with the two smoothing morphisms commutes.  This is the
geometric form of functoriality of the cosheaf smoothing functor; see
\cite{de_Silva_2016}.

We will use the following comparison in the smoothing cylinder. It shows that
the sublevel sets of the smoothing cylinder deformation retract onto copies of
the \(\delta\)-shifted sublevel sets of the original graph. To define these
copies, we send the shifted sublevel set into the cylinder by a section which
agrees with the zero section where the zero section already lies in the
truncated cylinder; on the remaining part, the section \text{``moves diagonally''} in the
cylinder, using the available vertical room staying at the same height, as shown in \Cref{fig:iota}.

\begin{figure}
    \centering
    \includegraphics[width = 0.6\textwidth]{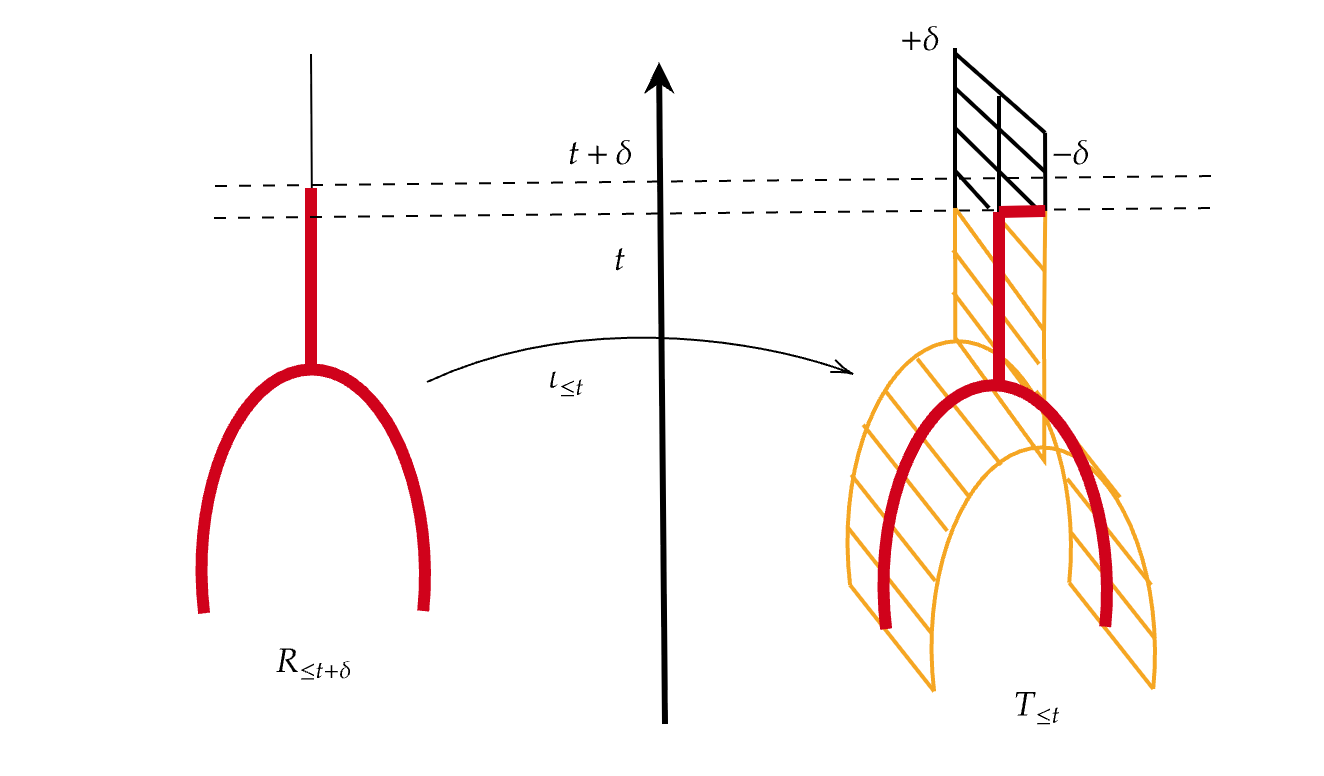}
    \caption{A visual representation of the map $\iota_{\le t}$ involved in \Cref{lem:filtered_cylinder_comparison}.}
    \label{fig:iota}
\end{figure}

\begin{lem}[Filtered cylinder comparison]
\label{lem:filtered_cylinder_comparison}
In the cylinder model for \(U_\delta(R)\), with \(H\), \(q\), and
\(\zeta_R^\delta\) as above, fix \(t\in\mathbb R\) and set
\[
T_{\le t}:=H^{-1}((-\infty,t])
=
\{(x,u)\in T:h_R(x)+u\le t\},
\]
and
\[
T_{\ge t}:=H^{-1}([t,+\infty))
=
\{(x,u)\in T:h_R(x)+u\ge t\}.
\]
Then the following hold.
\begin{enumerate}
    \item The projection to the first factor restricts to a homotopy
    equivalence
    \[
    p_{\le t}\colon T_{\le t}\longrightarrow R_{\le t+\delta}.
    \]
    A homotopy inverse is
    \[
    \iota_{\le t}\colon R_{\le t+\delta}\longrightarrow T_{\le t},
    \qquad
    \iota_{\le t}(x)=\bigl(x,\min\{0,t-h_R(x)\}\bigr).
    \]
    Moreover, the maps
    \[
    q\circ\iota_{\le t} ,
    \qquad
    \zeta_R^\delta|_{R_{\le t+\delta}}
    \]
    are homotopic as maps \(R_{\le t+\delta}\to U_\delta(R)\).

    \item The projection to the first factor restricts to a homotopy
    equivalence
    \[
    p_{\ge t}\colon T_{\ge t}\longrightarrow R_{\ge t-\delta}.
    \]
    A homotopy inverse is
    \[
    \iota_{\ge t}\colon R_{\ge t-\delta}\longrightarrow T_{\ge t},
    \qquad
    \iota_{\ge t}(x)=\bigl(x,\max\{0,t-h_R(x)\}\bigr).
    \]
    Moreover, the maps
    \[
    q\circ\iota_{\ge t},
    \qquad
    \zeta_R^\delta|_{R_{\ge t-\delta}}
    \]
    are homotopic as maps \(R_{\ge t-\delta}\to U_\delta(R)\).

    \item The quotient map induces bijections
    \[
    \pi_0(T_{\le t})
    \cong
    \pi_0\bigl((U_\delta(R))_{\le t}\bigr),
    \qquad
    \pi_0(T_{\ge t})
    \cong
    \pi_0\bigl((U_\delta(R))_{\ge t}\bigr).
    \]
    Consequently,
    \begin{equation}\label{eq:ordinary_H0_smoothing_identification}
    H_0\bigl((U_\delta(R))_{\le t};\Bbbk\bigr)
    \cong
    H_0\bigl(R_{\le t+\delta};\Bbbk\bigr),
    \end{equation}
    and
    \begin{equation}\label{eq:superlevel_H0_smoothing_identification}
    H_0\bigl((U_\delta(R))_{\ge t};\Bbbk\bigr)
    \cong
    H_0\bigl(R_{\ge t-\delta};\Bbbk\bigr).
    \end{equation}
\end{enumerate}
These identifications commute with the sublevel and superlevel structure maps
and are natural with respect to morphisms of Reeb graphs.
\end{lem}

\begin{proof}
We prove the sublevel statements first. The image of
\(p_{\le t}\colon T_{\le t}\to R\) is exactly \(R_{\le t+\delta}\). Indeed, if
\((x,u)\in T_{\le t}\), then
\[
h_R(x)\le t-u\le t+\delta.
\]
Conversely, if \(x\in R_{\le t+\delta}\), then \(t-h_R(x)\ge -\delta\), and hence
\[
\min\{0,t-h_R(x)\}\in[-\delta,\delta].
\]
Moreover,
\[
h_R(x)+\min\{0,t-h_R(x)\}\le t.
\]
Thus \(\iota_{\le t}\) is well-defined, has image in \(T_{\le t}\), and satisfies
\[
p_{\le t}\circ \iota_{\le t}=\id_{R_{\le t+\delta}}.
\]

We now check that \(\iota_{\le t}\) is a homotopy inverse to \(p_{\le t}\). Define
\[
K^-\colon T_{\le t}\times[0,1]\longrightarrow T
\]
by
\[
K^-((x,u),a)
=
\bigl(x,(1-a)u+a\min\{0,t-h_R(x)\}\bigr).
\]
We claim that \(K^-\) actually takes values in \(T_{\le t}\). If
\((x,u)\in T_{\le t}\), then
\[
u\le t-h_R(x).
\]
Also,
\[
\min\{0,t-h_R(x)\}\le t-h_R(x).
\]
Therefore every convex combination
\[
(1-a)u+a\min\{0,t-h_R(x)\}
\]
is still at most \(t-h_R(x)\). Hence
\[
h_R(x)+(1-a)u+a\min\{0,t-h_R(x)\}\le t,
\]
so the homotopy stays inside \(T_{\le t}\). Since both vertical coordinates belong
to \([-\delta,\delta]\), the whole segment also stays inside the smoothing
cylinder. At \(a=0\) this homotopy is the identity on \(T_{\le t}\), while at
\(a=1\) it is \(\iota_{\le t}\circ p_{\le t}\). Together with
\[
p_{\le t}\circ\iota_{\le t}=\id_{R_{\le t+\delta}},
\]
this proves that
\[
p_{\le t}\colon T_{\le t}\to R_{\le t+\delta}
\]
is a homotopy equivalence with homotopy inverse \(\iota_{\le t}\).

We next compare the map \(q\circ\iota_{\le t}\) with the smoothing morphism. To do
so, we compare the two maps from \(R_{\le t+\delta}\) into the cylinder: the
section \(\iota_{\le t}\) and the zero section. The latter, after composition with
the Reeb quotient map \(q\), is precisely the restriction of the smoothing
morphism \(\zeta_R^\delta\).

For \(x\in R_{\le t+\delta}\), consider the path
\[
a\longmapsto
\bigl(x,a\min\{0,t-h_R(x)\}\bigr),
\qquad a\in[0,1].
\]
This path is contained in the full cylinder \(T\). Indeed,
if \(h_R(x)\le t\), then \(\min\{0,t-h_R(x)\}=0\), so the path is constant at
\((x,0)\). If \(t<h_R(x)\le t+\delta\), then
\[
t-h_R(x)\in[-\delta,0],
\]
so the path runs from \((x,0)\) to
\[
\bigl(x,t-h_R(x)\bigr)=\iota_{\le t}(x).
\]
Thus the zero section
\[
i|_{R_{\le t+\delta}}\colon R_{\le t+\delta}\to T
\]
is homotopic in the full cylinder to \(\iota_{\le t}\). Applying the quotient map
\(q\), we obtain a homotopy in \(U_\delta(R)\) between the maps
\[
\zeta_R^\delta|_{R_{\le t+\delta}}
=
q\circ i|_{R_{\le t+\delta}},
\qquad
q\circ\iota_{\le t}.
\]
Consequently, \(q\circ\iota_{\le t}\) and
\(\zeta_R^\delta|_{R_{\le t+\delta}}\) induce the same map on \(\pi_0\), and
therefore the same map on \(H_0(-;\Bbbk)\).

The restriction
\[
q|_{T_{\le t}}\colon T_{\le t}\longrightarrow (U_\delta(R))_{\le t}
\]
is the Reeb quotient of the restricted constructible space
$(T_{\le t},H|_{T_{\le t}}).
$
By \Cref{lem:reeb_quotient_pi0}, it induces a bijection
\[
\pi_0(q|_{T_{\le t}})\colon
\pi_0(T_{\le t})
\longrightarrow
\pi_0\bigl((U_\delta(R))_{\le t}\bigr).
\]

On the other hand, the deformation retraction constructed above gives a
bijection
\[
\pi_0(p_{\le t})\colon
\pi_0(T_{\le t})
\longrightarrow
\pi_0(R_{\le t+\delta}),
\]
whose inverse is induced by \(\iota_{\le t}\). Hence the composite
\[
\pi_0(R_{\le t+\delta})
\xrightarrow{\pi_0(\iota_{\le t})}
\pi_0(T_{\le t})
\xrightarrow{\pi_0(q|_{T_{\le t}})}
\pi_0\bigl((U_\delta(R))_{\le t}\bigr)
\]
is a bijection. This composite is induced by the map
\[
q\circ\iota_{\le t}\colon R_{\le t+\delta}\longrightarrow (U_\delta(R))_{\le t}.
\]
By the homotopy constructed above, this map induces the same map on \(\pi_0\)
as
$\zeta_R^\delta|_{R_{\le t+\delta}}.
$
Hence, we obtain the canonical identification
\[
H_0\bigl(R_{\le t+\delta};\Bbbk\bigr)
\cong
H_0\bigl((U_\delta(R))_{\le t};\Bbbk\bigr),
\]
or equivalently
\[
H_0\bigl((U_\delta(R))_{\le t};\Bbbk\bigr)
\cong
H_0\bigl(R_{\le t+\delta};\Bbbk\bigr).
\]
This is \eqref{eq:ordinary_H0_smoothing_identification}.

The superlevel statements are obtained by applying the same argument to the
height function \(-h_R\). Equivalently, one uses
\[
\iota_{\ge t}(x)=\bigl(x,\max\{0,t-h_R(x)\}\bigr),
\qquad x\in R_{\ge t-\delta}.
\]
The vertical homotopy from \((x,u)\) to \(\iota_{\ge t}(x)\) inside \(T_{\ge t}\) is
obtained by the same formula with \(\max\{0,t-h_R(x)\}\) in place of
\(\min\{0,t-h_R(x)\}\). Repeating the same arguments as above yields
\[
H_0\bigl((U_\delta(R))_{\ge t};\Bbbk\bigr)
\cong
H_0\bigl(R_{\ge t-\delta};\Bbbk\bigr),
\]
which is \eqref{eq:superlevel_H0_smoothing_identification}.

Finally, we verify naturality. If \(s\le t\), the inclusions of truncated
sublevel cylinders and sublevel sets give a commutative diagram
\[
\begin{tikzcd}
T_{\le s} \arrow[r,hook] \arrow[d,"p_{\le s}"']
&
T_{\le t} \arrow[d,"p_{\le t}"]
\\
R_{\le s+\delta} \arrow[r,hook]
&
R_{\le t+\delta}.
\end{tikzcd}
\]
Likewise, the restrictions of the Reeb quotient map fit into the commutative
diagram
\[
\begin{tikzcd}
T_{\le s} \arrow[r,hook] \arrow[d,"q"']
&
T_{\le t} \arrow[d,"q"]
\\
(U_\delta(R))_{\le s} \arrow[r,hook]
&
(U_\delta(R))_{\le t}.
\end{tikzcd}
\]
Since the vertical maps induce bijections on path components, the
identifications
\[
H_0\bigl((U_\delta(R))_{\le t};\Bbbk\bigr)
\cong
H_0\bigl(R_{\le t+\delta};\Bbbk\bigr)
\]
commute with the sublevel structure maps. The same argument applies to the
superlevel identifications.

Now let \(F\colon R\to S\) be a morphism of Reeb graphs. Set
\[
T_R:=R\times[-\delta,\delta],
\qquad
T_S:=S\times[-\delta,\delta].
\]
Since \(F\) preserves the height functions, the map
\[
\widetilde F
\coloneqq
F\times\id_{[-\delta,\delta]}
\colon
T_R
\longrightarrow
T_S
\]
preserves the induced height functions on the smoothing cylinders. Hence, for
every \(t\), it restricts to maps
\[
\widetilde F_{\le t}\colon (T_R)_{\le t}\to (T_S)_{\le t},
\qquad
\widetilde F_{\ge t}\colon (T_R)_{\ge t}\to (T_S)_{\ge t}.
\]
For the sublevel sets, these maps fit into the commutative diagrams
\[
\begin{tikzcd}
(T_R)_{\le t}
    \arrow[r,"\widetilde F_{\le t}"]
    \arrow[d,"p_{\le t}^R"']
&
(T_S)_{\le t}
    \arrow[d,"p_{\le t}^S"]
\\
R_{\le t+\delta}
    \arrow[r,"F"]
&
S_{\le t+\delta}
\end{tikzcd}
\qquad
\begin{tikzcd}
(T_R)_{\le t}
    \arrow[r,"\widetilde F_{\le t}"]
    \arrow[d,"q_R"']
&
(T_S)_{\le t}
    \arrow[d,"q_S"]
\\
(U_\delta(R))_{\le t}
    \arrow[r,"U_\delta(F)"]
&
(U_\delta(S))_{\le t}.
\end{tikzcd}
\]
The first diagram follows from
\[
p_{\le t}^S\circ\widetilde F_{\le t}
=
F\circ p_{\le t}^R,
\]
while the second is the restriction of naturality of the Reeb quotient 
\[
U_\delta(F)\circ q_R
=
q_S\circ\widetilde F.
\]
Since the projections are homotopy equivalences and the quotient maps induce
bijections on path components, passing to \(H_0\) and composing the induced
isomorphisms gives the commutative diagram
\[
\begin{tikzcd}
H_0\bigl((U_\delta(R))_{\le t};\Bbbk\bigr)
    \arrow[r,"(U_\delta(F))_*"]
    \arrow[d,"\cong"']
&
H_0\bigl((U_\delta(S))_{\le t};\Bbbk\bigr)
    \arrow[d,"\cong"]
\\
H_0(R_{\le t+\delta};\Bbbk)
    \arrow[r,"F_*"]
&
H_0(S_{\le t+\delta};\Bbbk).
\end{tikzcd}
\]
The same argument applies to the superlevel identifications. Hence the
identifications are natural with respect to morphisms of Reeb graphs.
\end{proof}

Before continuing, we note that the canonical \(H_0\)-identifications above are
specific to path components.  In general there is no canonical smooth-to-shift
morphism on the extended \(H_1\)-module without choosing interval decompositions.

\subsection{The \texorpdfstring{\(H_0\)}{H0}-related components}

Let \(\Phi\colon R\to U_\varepsilon(S)\) and
\(\Psi\colon S\to U_\varepsilon(R)\) be an \(\varepsilon\)-interleaving.  We first consider ordinary sublevel \(H_0\).
Using \eqref{eq:ordinary_H0_smoothing_identification} with \(\delta=\varepsilon\),
the map \(\Phi\) induces maps
\[
H_0(R_{\le t})
\longrightarrow
H_0\bigl((U_\varepsilon(S))_{\le t}\bigr)
\cong
H_0(S_{\le t+\varepsilon}),
\]
and \(\Psi\) induces maps in the opposite direction.  By functoriality of
smoothing and the interleaving identities
\[
U_\varepsilon(\Psi)\circ\Phi=\zeta_R^{2\varepsilon},
\qquad
U_\varepsilon(\Phi)\circ\Psi=\zeta_S^{2\varepsilon},
\]
the two composites are the ordinary structure maps
\[
H_0(R_{\le t})\longrightarrow H_0(R_{\le t+2\varepsilon}),
\qquad
H_0(S_{\le t})\longrightarrow H_0(S_{\le t+2\varepsilon}).
\]
Thus the ordinary sublevel \(H_0\)-modules of \(R\) and \(S\) are
\(\varepsilon\)-interleaved.  Algebraic stability gives
\begin{equation}\label{eq:ord0_bound}
d_B\!\bigl(\Ord_0(R),\Ord_0(S)\bigr)\le \varepsilon .
\end{equation}

The same argument applied to the functions \(-h_R\) and \(-h_S\), or equivalently
to the superlevel identifications \eqref{eq:superlevel_H0_smoothing_identification},
gives an \(\varepsilon\)-interleaving of the ordinary \(H_0\)-modules of the
superlevel filtrations.  Under the standard extended-persistence identification,
these modules encode the relative component \(\Rel_1\).  Hence
\begin{equation}\label{eq:rel1_bound}
d_B\!\bigl(\Rel_1(R),\Rel_1(S)\bigr)\le \varepsilon .
\end{equation}

Finally, since the Reeb graphs are connected in this section, the extended
\(H_0\)-component consists of the single point
\[
\Ext_0(R)=\{(\min h_R,\max h_R)\},
\qquad
\Ext_0(S)=\{(\min h_S,\max h_S)\}.
\]
The existence of a morphism \(R\to U_\varepsilon(S)\) implies
\[
\min h_S-\varepsilon\le \min h_R,
\qquad
\max h_R\le \max h_S+\varepsilon,
\]
and the morphism \(S\to U_\varepsilon(R)\) gives the reverse inequalities.
Therefore
\[
|\min h_R-\min h_S|\le\varepsilon,
\qquad
|\max h_R-\max h_S|\le\varepsilon,
\]
and so
\begin{equation}\label{eq:ext0_bound}
d_B\!\bigl(\Ext_0(R),\Ext_0(S)\bigr)\le \varepsilon .
\end{equation}

\subsubsection{Conclusion of the \(H_0\)-related comparison}

The arguments above prove the sharp estimate for the \(H_0\)-related components.
Namely, for every \(\varepsilon\)-interleaving
\[
\Phi\colon R\to U_\varepsilon(S),
\qquad
\Psi\colon S\to U_\varepsilon(R),
\]
combining \eqref{eq:ord0_bound}, \eqref{eq:rel1_bound}, and
\eqref{eq:ext0_bound}, we obtain
\begin{equation}\label{eq:extended_H0_bound}
\max\left\{
\begin{array}{l}
d_B\!\bigl(\Ord_0(R),\Ord_0(S)\bigr),\\[0.3em]
d_B\!\bigl(\Ext_0(R),\Ext_0(S)\bigr),\\[0.3em]
d_B\!\bigl(\Rel_1(R),\Rel_1(S)\bigr)
\end{array}
\right\}
\le \varepsilon.
\end{equation}
Letting \(\varepsilon\downarrow d_I(R,S)\), we obtain
\[
\max\left\{
\begin{array}{l}
d_B\!\bigl(\Ord_0(R),\Ord_0(S)\bigr),\\[0.3em]
d_B\!\bigl(\Ext_0(R),\Ext_0(S)\bigr),\\[0.3em]
d_B\!\bigl(\Rel_1(R),\Rel_1(S)\bigr)
\end{array}
\right\}
\le d_I(R,S).
\]

\begin{oss}[Merge trees]
For clarity and completeness, we record a finer interpretation of the preceding
argument, which gives a complete description of the situation for the
\(H_0\)-related components. This observation is not needed for the remainder of
the paper, so we do not introduce merge trees or their interleaving distance in
detail.

Let \(M(h_R)\) and \(M(h_S)\) denote the merge trees associated with the sublevel-set
filtrations of \(h_R\) and \(h_S\), respectively, and define \(M(-h_R)\) and \(M(-h_S)\)
analogously. The proof of \Cref{lem:filtered_cylinder_comparison}, together with the interleaving
argument above, takes place already at the level of path-connected components.
It therefore shows that every \(\varepsilon\)-interleaving between \(R\) and
\(S\) induces \(\varepsilon\)-interleavings
\[
M(h_R)\longleftrightarrow M(h_S),
\qquad
M(-h_R)\longleftrightarrow M(-h_S).
\]
Consequently,
\begin{equation*}\label{eq:extended_H0_merge_tree_bound}
\max\left\{
\begin{array}{l}
d_B\!\bigl(\Ord_0(R),\Ord_0(S)\bigr),\\[0.3em]
d_B\!\bigl(\Ext_0(R),\Ext_0(S)\bigr),\\[0.3em]
d_B\!\bigl(\Rel_1(R),\Rel_1(S)\bigr)
\end{array}
\right\}
\le
\max\left\{
\begin{array}{l}
d_I\!\bigl(M(h_R),M(h_S)\bigr),\\[0.3em]
d_I\!\bigl(M(-h_R),M(-h_S)\bigr)
\end{array}
\right\}
\le
d_I(R,S).
\end{equation*}
\end{oss}

\subsection{The \texorpdfstring{\(H_1\)}{H1}-component}
\label{subsec:h1_component_newborn_corrected}

We now treat the \(\Ext_1\)-component. As already mentioned, no canonical
smooth-to-shift morphism on the \(\Ext_1\)-persistence module is available. We
instead combine the two-step smoothing identity with a Hall-rank argument: long
bars survive the \(2\varepsilon\)-smoothing composite, and bad intermediate bars
can only contribute to strictly persistence-decreasing matrix entries.

In the results below, we will need the following assumption: in each of the
four sublevel-set filtrations associated with \(h_R\), \(-h_R\), \(h_S\), and
\(-h_S\), the ordinary \(H_1\)-birth times are pairwise distinct, and, for each
of \(R\) and \(S\), no vertex can simultaneously support an ordinary
\(H_1\)-birth for the filtration of the structure map and an ordinary
\(H_1\)-birth for the filtration of its negative.
The genericity of this assumption is proved in
\Cref{subsec:h1_genericity_newborn}.

\paragraph{Cycle representatives for \(\Ext_1\)-bars.}
Let \(R\) be a Reeb graph. For a \(1\)-cycle
\(\gamma\in Z_1(R;\Bbbk)\), set
\[
 b(\gamma)\coloneqq \max h_R(\operatorname{supp}\gamma),
 \qquad
 d(\gamma)\coloneqq \min h_R(\operatorname{supp}\gamma).
\]
We shall use the following structural description of the \(\Ext_1\)-barcode. The
point is that one can choose cycle representatives whose homology classes are
adapted simultaneously to the sublevel and superlevel flags.

\begin{lem}[Cycle representatives for \(\Ext_1\)-bars]
\label{lem:cycle_representatives_ext1_bars}
Let \(R\) be a Reeb graph. Assume that the ordinary \(H_1\)-birth times of
the sublevel filtration of \(h_R\) are pairwise distinct, and that the ordinary
\(H_1\)-birth times of the sublevel filtration of \(-h_R\) are pairwise distinct.
Then the points of \(\Ext_1(R)\) can be indexed as
\((b_i,d_i)\), \(i=1,\ldots,m\), with \(b_i>d_i\), and admit \(1\)-cycle
representatives
\[
\gamma_1,\ldots,\gamma_m\in Z_1(R;\Bbbk)
\]
such that
\[
b_i=b(\gamma_i),
\qquad
 d_i=d(\gamma_i).
\]
Moreover:
\begin{enumerate}
    \item ordered by increasing \(b_i\), the classes \([\gamma_i]\) form a
    newborn-cycle basis for the ordinary sublevel \(H_1\)-persistence of \(h_R\);
    equivalently, for every \(t\),
    \[
    \operatorname{im}\bigl(H_1(R_{\le t};\Bbbk)\to H_1(R;\Bbbk)\bigr)
    =
    \operatorname{span}\{[\gamma_i]: b_i\le t\};
    \]
    \item ordered by decreasing \(d_i\), the same classes \([\gamma_i]\) form a
    newborn-cycle basis for the superlevel filtration of \(h_R\), equivalently for
    the ordinary sublevel persistence of \(-h_R\): for every \(t\),
    \[
    \operatorname{im}\bigl(H_1(R_{\ge t};\Bbbk)\to H_1(R;\Bbbk)\bigr)
    =
    \operatorname{span}\{[\gamma_i]: d_i\ge t\}.
    \]
\end{enumerate}
\end{lem}

\begin{proof}
Let \(V=H_1(R;\Bbbk)\). We write
\[
F_R^{\le t}
:=
\operatorname{im}\bigl(H_1(R_{\le t};\Bbbk)\to V\bigr),
\qquad
G_R^{\ge t}
:=
\operatorname{im}\bigl(H_1(R_{\ge t};\Bbbk)\to V\bigr).
\]
Choose a barcode decomposition of the degree-one extended-persistence module of
\(R\). The summands corresponding to \(\Ext_1\)-points give a basis
\(\xi_1,\ldots,\xi_m\) of \(V\), indexed by the points \((b_i,d_i)\) such that
\[
F_R^{\le t}=\operatorname{span}\{\xi_i:b_i\le t\},
\qquad
G_R^{\ge t}=\operatorname{span}\{\xi_i:d_i\ge t\}.
\]
The second equality is just the same barcode decomposition read for \(-h_R\): a
birth at level \(-d_i\) for \(-h_R\) corresponds, in the original height parameter,
to appearance in the superlevel at threshold \(d_i\).

Since \(R\) is a
graph, it has no \(2\)-simplices, so \(B_1(R;\Bbbk)=0\). Thus each class in
\(H_1(R;\Bbbk)\) has a unique simplicial \(1\)-cycle representative. Let
\(\gamma_i\) be the representative of \(\xi_i\). Because
\(\xi_i\in F_R^{\le b_i}\) and \(\xi_i\notin F_R^{<b_i}\), the cycle
\(\gamma_i\) is supported in \(R_{\le b_i}\) and in no strict sublevel set below
\(b_i\). Hence
$\max h_R(\operatorname{supp}\gamma_i)=b_i.
$

Similarly, \(\xi_i\in G_R^{\ge d_i}\) and \(\xi_i\notin G_R^{>d_i}\), so
\(\gamma_i\) is supported in \(R_{\ge d_i}\) and in no strict superlevel set above
\(d_i\). Hence
$\min h_R(\operatorname{supp}\gamma_i)=d_i.
$
\end{proof}

\paragraph{Smoothing and long \(\Ext_1\)-bars.}
We now recall the smoothing formula for extended persistence in
\cite[Corollary~7]{alharbi2024realizable}: the \(\Ext_1\)-barcode of
\((U_\delta(R),h_\delta)\) is obtained from the \(\Ext_1\)-barcode of
\(R\) as follows: an \(\Ext_1\)-point \((b,d)\), with \(b>d\), is sent, at
the level of barcodes, to \((b-\delta,d+\delta)\) if \(b-d>2\delta\), and is
removed if \(b-d\le 2\delta\).

We stress that this is used only as a statement about the barcodes of the two
objects \(R\) and \((U_\delta(R),h_\delta)\): to the best of our
understanding, \cite{alharbi2024realizable} does not prove that the map on
persistence modules induced by the smoothing morphism sends each barcode summand
of \(R\) to the corresponding shifted barcode summand of
\((U_\delta(R),h_\delta)\).

For the Hall argument below, we shall use \(\delta=2\varepsilon\). We call an
\(\Ext_1\)-bar \emph{long} if its persistence is larger than \(4\varepsilon\),
and \emph{short} otherwise. Thus the long bars are precisely the bars which
survive \(2\varepsilon\)-smoothing at the barcode level.

\begin{lem}[Short \(\Ext_1\)-bars are diagonal-disposable]
\label{lem:short_ext1_bars_diagonal_disposable}
Let \(R\) be a Reeb graph and let \((b,d)\in\Ext_1(R)\) be an \(\Ext_1\)-point. If
$b-d\le 4\varepsilon,
$
then this point has diagonal bottleneck cost at most \(2\varepsilon\). Moreover,
it is removed by \(2\varepsilon\)-smoothing.
\end{lem}

\begin{proof}
The cost of matching \((b,d)\) to the diagonal is $(b-d)/2$,
which is at most \(2\varepsilon\) by assumption. By the smoothing formula in \cite{alharbi2024realizable},
\(2\varepsilon\)-smoothing keeps an \(\Ext_1\)-bar only if its persistence is
larger than \(4\varepsilon\). Hence every bar with \(b-d\le 4\varepsilon\) is
removed by \(2\varepsilon\)-smoothing.
\end{proof}

\paragraph{Filtered surjectivity of smoothing.}
We next prove the homological input which replaces any direct assumption about
how smoothing acts on barcode summands.

Let \(R\) be a constructible Reeb graph. For \(t\in\mathbb R\), set
\[
F_R^{\le t}
:=
\operatorname{im}\Bigl(
H_1(R_{\le t};\Bbbk)\to H_1(R;\Bbbk)
\Bigr),
\qquad
F_R^{<t}
:=
\operatorname{im}\Bigl(
H_1(R_{<t};\Bbbk)\to H_1(R;\Bbbk)
\Bigr),
\]
and
\[
\widetilde F_{R,\delta}^{\le t}
:=
\operatorname{im}\Bigl(
H_1((U_\delta(R))_{\le t};\Bbbk)\to H_1(U_\delta(R);\Bbbk)
\Bigr),
\]
\[
\widetilde F_{R,\delta}^{<t}
:=
\operatorname{im}\Bigl(
H_1((U_\delta(R))_{<t};\Bbbk)\to H_1(U_\delta(R);\Bbbk)
\Bigr).
\]
Similarly, define the superlevel images
\[
G_R^{\ge t}
:=
\operatorname{im}\Bigl(
H_1(R_{\ge t};\Bbbk)\to H_1(R;\Bbbk)
\Bigr),
\qquad
G_R^{>t}
:=
\operatorname{im}\Bigl(
H_1(R_{>t};\Bbbk)\to H_1(R;\Bbbk)
\Bigr),
\]
and
\[
\widetilde G_{R,\delta}^{\ge t}
:=
\operatorname{im}\Bigl(
H_1((U_\delta(R))_{\ge t};\Bbbk)\to H_1(U_\delta(R);\Bbbk)
\Bigr),
\]
\[
\widetilde G_{R,\delta}^{>t}
:=
\operatorname{im}\Bigl(
H_1((U_\delta(R))_{>t};\Bbbk)\to H_1(U_\delta(R);\Bbbk)
\Bigr).
\]

\begin{lem}[Reeb quotient maps are surjective on \(H_1\)]
\label{lem:reeb_quotient_surjective_H1}
Let \((X,f)\) be a constructible \(\mathbb R\)-space, set
$R\coloneqq\mathfrak R(X,f),
$
and let \(q_X\colon X\to R\) be the Reeb quotient map. Then
\[
(q_X)_*\colon H_1(X;\Bbbk)\to H_1(R;\Bbbk)
\]
is surjective.
\end{lem}

\begin{proof}
Choose a simplicial structure on \(R\) whose vertices contain all points over
critical values of \(f\), so that each open edge of \(R\) lies over an interval
of regular values. Let
$z=\sum_e c_e e
$
be a simplicial \(1\)-cycle in \(R\), with each edge \(e\) oriented from
\(s(e)\) to \(t(e)\).

We use the same constructibility argument as in the proof of
\Cref{lem:reeb_quotient_pi0}. If \(e\) is an open edge of \(R\), then there is
an interval \(I\subset\mathbb R\) of regular values such that \(e\) is a
path-connected component of \(\bar f^{-1}(I)\), while \(q_X^{-1}(e)\) is the
corresponding path-connected component of \(f^{-1}(I)\). By constructibility,
this component is a product over \(I\). Hence, for every \(e\) with
\(c_e\neq 0\), we may choose a path
\[
\widetilde e\subset q_X^{-1}(e)
\]
joining a point \(\widetilde s_e\in q_X^{-1}(s(e))\) to a point
\(\widetilde t_e\in q_X^{-1}(t(e))\).

The lifted edge chain has boundary
\[
\partial\sum_e c_e\widetilde e
=
\sum_e c_e(\widetilde t_e-\widetilde s_e),
\]
which lies in the fibers over the vertices of \(R\). Fix such a vertex \(v\).
By definition of the Reeb quotient, \(q_X^{-1}(v)\) is a path-connected
component of the level set \(f^{-1}(\bar f(v))\). 
Choose a point \(p_v\in q_X^{-1}(v)\). For each edge with \(s(e)=v\), choose
a path \([p_v,\widetilde s_e]\subset q_X^{-1}(v)\), and for each edge with
\(t(e)=v\), choose a path \([\widetilde t_e,p_v]\subset q_X^{-1}(v)\). Set
\[
C_v
\coloneqq
\sum_{s(e)=v} c_e [p_v,\widetilde s_e]
+
\sum_{t(e)=v} c_e [\widetilde t_e,p_v].
\]
Then
\[
\partial C_v
=
\sum_{s(e)=v} c_e(\widetilde s_e-p_v)
+
\sum_{t(e)=v} c_e(p_v-\widetilde t_e).
\]
Adding \(\partial C_v\) to the boundary of the lifted edge chain cancels all
terms supported at the lifted endpoints over \(v\). The remaining
coefficient of \(p_v\) is
$\sum_{t(e)=v} c_e-\sum_{s(e)=v} c_e,
$
which is zero because \(z\) is a cycle.

Therefore
\[
\widetilde z
\coloneqq
\sum_e c_e\widetilde e+\sum_v C_v
\]
is a \(1\)-cycle in \(X\). The correction chains \(C_v\) map to vertices of
\(R\), hence contribute no simplicial \(1\)-chain in \(R\), while each
\(\widetilde e\) maps to \(e\). Thus
\[
(q_X)_*[\widetilde z]=[z].
\]
Since \(z\) was arbitrary, \((q_X)_*\) is surjective.
\end{proof}

\begin{lem}[Filtered surjectivity of smoothing]
\label{lem:filtered_surjectivity_smoothing}
For every \(t\in\mathbb R\),
\[
(\zeta_R^\delta)_*
\bigl(F_R^{\le t+\delta}\bigr)
=
\widetilde F_{R,\delta}^{\le t},
\qquad 
(\zeta_R^\delta)_*
\bigl(F_R^{<t+\delta}\bigr)
=
\widetilde F_{R,\delta}^{<t}.
\]
Similarly,
\[
(\zeta_R^\delta)_*
\bigl(G_R^{\ge t-\delta}\bigr)
=
\widetilde G_{R,\delta}^{\ge t},
\qquad
(\zeta_R^\delta)_*
\bigl(G_R^{>t-\delta}\bigr)
=
\widetilde G_{R,\delta}^{>t}.
\]
\end{lem}

\begin{proof}
We prove the closed sublevel statement first. Set
\[
T:=R\times[-\delta,\delta],
\qquad
H(x,u):=h_R(x)+u,
\]
and let
\[
q\colon T\to U_\delta(R)
\]
be the Reeb quotient map. Let
\[
T_{\le t}:=H^{-1}((-\infty,t])
=\{(x,u)\in T:h_R(x)+u\le t\}
\]
denote the truncated sublevel cylinder.
By \Cref{lem:filtered_cylinder_comparison}, the projection identifies
\(T_{\le t}\) up to homotopy with \(R_{\le t+\delta}\), with homotopy inverse
\[
\iota_{\le t}(x)=\bigl(x,\min\{0,t-h_R(x)\}\bigr).
\]
Moreover, \(q\circ\iota_{\le t}\) is homotopic in \(U_\delta(R)\) to
\(\zeta_R^\delta|_{R_{\le t+\delta}}\).

We first prove the inclusion
\[
(\zeta_R^\delta)_*
\bigl(F_R^{\le t+\delta}\bigr)
\subseteq
\widetilde F_{R,\delta}^{\le t}.
\]
Let \(\gamma\) be a cycle contained in \(R_{\le t+\delta}\). Then
\[
(\zeta_R^\delta)_*[\gamma]=q_*i_*[\gamma].
\]
The maps \(i|_{R_{\le t+\delta}}\) and \(\iota_{\le t}\) are homotopic in the full
cylinder, and therefore
\[
q_*i_*[\gamma]=q_*(\iota_{\le t})_*[\gamma]
\]
in \(H_1(U_\delta(R);\Bbbk)\). But
\[
q\circ\iota_{\le t}(R_{\le t+\delta})
\subseteq
(U_\delta(R))_{\le t}.
\]
Hence \((\zeta_R^\delta)_*[\gamma]\) is represented by a cycle contained in
\((U_\delta(R))_{\le t}\), and so belongs to
\(\widetilde F_{R,\delta}^{\le t}\).

Conversely, let \(\xi\in \widetilde F_{R,\delta}^{\le t}\).
Then \(\xi\) is represented by a cycle in
\[
(U_\delta(R))_{\le t}=q(T_{\le t}).
\]
The restriction
\[
q|_{T_{\le t}}\colon T_{\le t}\to q(T_{\le t})
\]
is the Reeb quotient map of the constructible space \((T_{\le t},H|_{T_{\le t}})\).
Therefore, by \Cref{lem:reeb_quotient_surjective_H1}, there exists a cycle
\(\Gamma\in Z_1(T_{\le t};\Bbbk)\) such that
$q_*[\Gamma]=\xi.
$

Since \(T_{\le t}\) deformation retracts onto \(\iota_{\le t}(R_{\le t+\delta})\), the
class \([\Gamma]\) is represented by a cycle of the form
$(\iota_{\le t})_*\gamma
$
for some \(\gamma\in Z_1(R_{\le t+\delta};\Bbbk)\).
Thus
\[
\xi=q_*(\iota_{\le t})_*[\gamma].
\]
Finally, \(q\circ\iota_{\le t}\) is homotopic in \(U_\delta(R)\) to
\(q\circ i=\zeta_R^\delta\) on \(R_{\le t+\delta}\). Hence
\[
q_*(\iota_{\le t})_*[\gamma]
=
(\zeta_R^\delta)_*[\gamma]
\]
in \(H_1(U_\delta(R);\Bbbk)\). Therefore
\[
\xi\in
(\zeta_R^\delta)_*
\bigl(F_R^{\le t+\delta}\bigr).
\]
This proves
\[
(\zeta_R^\delta)_*
\bigl(F_R^{\le t+\delta}\bigr)
=
\widetilde F_{R,\delta}^{\le t}.
\]

The strict sublevel statement follows by taking unions.  Indeed, if a cycle is supported in \(R_{<t+\delta}\), then its support is
compact, and therefore the maximum of \(h_R\) on its support is strictly smaller
than \(t+\delta\). Hence the cycle is already supported in \(R_{\le r+\delta}\)
for some \(r<t\). The same argument applies to cycles in
\((U_\delta(R))_{<t}\). Hence,
\[
F_R^{<t+\delta}
=
\bigcup_{r<t}F_R^{\le r+\delta},
\qquad
\widetilde F_{R,\delta}^{<t}
=
\bigcup_{r<t}\widetilde F_{R,\delta}^{\le r}.
\]
Taking the union of the closed-sublevel equality over all \(r<t\) gives
\[
(\zeta_R^\delta)_*
\bigl(F_R^{<t+\delta}\bigr)
=
\widetilde F_{R,\delta}^{<t}.
\]

The superlevel equality is proved in the same way, using the superlevel
part of \Cref{lem:filtered_cylinder_comparison}.  
\end{proof}

\paragraph{Triangularity of self-smoothing on surviving bars.}
Let \(\gamma_1,\dots,\gamma_m\) be the \(\Ext_1\)-cycle representatives of \(R\) given by
\Cref{lem:cycle_representatives_ext1_bars}. Write \(p_i=(b_i,d_i)\in\Ext_1(R)\) for the bar represented by \(\gamma_i\). For every \(i\) such that
$b_i-d_i>2\delta,
$
let \(p_i^\delta:=(b_i-\delta,d_i+\delta)\) be the corresponding surviving bar of \((U_\delta(R),h_\delta)\), and choose cycle representatives \(\gamma_i^\delta\in Z_1(U_\delta(R);\Bbbk)\) given by \Cref{lem:cycle_representatives_ext1_bars} for the smoothed graph.
Thus \([\gamma_i^\delta]\) is the newborn-cycle basis element in
\(H_1(U_\delta(R);\Bbbk)\) indexed by the surviving bar \(p_i^\delta\).

\begin{lem}[Triangularity of the self-smoothing map]
\label{lem:triangularity_self_smoothing_ext1}
Let \(J\subseteq \{i: b_i-d_i>2\delta\}\).
Set
\[
V_J:=\operatorname{span}\{[\gamma_i]:i\in J\}
\subseteq H_1(R;\Bbbk),
\]
and
\[
V_J^\delta:=\operatorname{span}\{[\gamma_i^\delta]:i\in J\}
\subseteq H_1(U_\delta(R);\Bbbk).
\]
Let \(\pi_J^\delta\colon H_1(U_\delta(R);\Bbbk)\to V_J^\delta\) be the coordinate projection with respect to the newborn-cycle basis of
\(H_1(U_\delta(R);\Bbbk)\). Then the map
\[
\pi_J^\delta\circ(\zeta_R^\delta)_*|_{V_J}
\colon V_J\to V_J^\delta
\]
has full rank \(\#J\). More precisely, order \(J\) by nonincreasing persistence, and use the 
ordered bases
\[
\{[\gamma_i]:i\in J\}
\quad\text{of }V_J,
\qquad
\{[\gamma_i^\delta]:i\in J\}
\quad\text{of }V_J^\delta.
\]
With respect to these ordered bases, the matrix of
$\pi_J^\delta\circ(\zeta_R^\delta)_*|_{V_J}
$
is lower triangular with nonzero diagonal entries.
\end{lem}

\begin{proof}
Fix \(i\in J\). Since
\[
[\gamma_i]\in F_R^{\le b_i}\setminus F_R^{<b_i},
\]
\Cref{lem:filtered_surjectivity_smoothing}, applied with \(t=b_i-\delta\), gives
\[
(\zeta_R^\delta)_*[\gamma_i]
\in
\widetilde F_{R,\delta}^{\le b_i-\delta}.
\]
The same lemma, in its strict version, gives
\[
(\zeta_R^\delta)_*F_R^{<b_i}
=
\widetilde F_{R,\delta}^{<b_i-\delta}.
\]
The induced map
\[
F_R^{\le b_i}/F_R^{<b_i}
\longrightarrow
\widetilde F_{R,\delta}^{\le b_i-\delta}/
\widetilde F_{R,\delta}^{<b_i-\delta}
\]
is surjective. Both spaces are one-dimensional: the source by the genericity
assumption, and the target by the barcode-level smoothing formula of \cite{alharbi2024realizable} together
with the distinctness of the birth values. Hence this induced map is nonzero.
Consequently, the coefficient of \([\gamma_i^\delta]\) in
$(\zeta_R^\delta)_*[\gamma_i]
$
is nonzero.

Moreover, since
\[
[\gamma_i]\in G_R^{\ge d_i},
\]
the superlevel part of \Cref{lem:filtered_surjectivity_smoothing}, applied with
\(t=d_i+\delta\), gives
\[
(\zeta_R^\delta)_*[\gamma_i]
\in
\widetilde G_{R,\delta}^{\ge d_i+\delta}.
\]
We now use that the newborn-cycle basis in \(H_1(U_\delta(R);\Bbbk)\) is adapted
simultaneously to the sublevel and superlevel flags of the smoothed graph. Namely,
\[
\widetilde F_{R,\delta}^{\le t}
=
\operatorname{span}
\{[\gamma_k^\delta]: b_k-d_k>2\delta\text{ and } b_k-\delta\le t\},
\]
\[
\widetilde G_{R,\delta}^{\ge t}
=
\operatorname{span}
\{[\gamma_k^\delta]: b_k-d_k>2\delta\text{ and } d_k+\delta\ge t\}.
\]
Therefore, consider $k$ such that the coefficient of \([\gamma_k^\delta]\) in
$(\zeta_R^\delta)_*[\gamma_i]
$
is nonzero; then membership in
\(\widetilde F_{R,\delta}^{\le b_i-\delta}\) forces
\[
b_k-\delta\le b_i-\delta,
\]
while membership in \(\widetilde G_{R,\delta}^{\ge d_i+\delta}\) forces
\[
d_k+\delta\ge d_i+\delta.
\]
Equivalently,
\[
b_k\le b_i,
\qquad
 d_k\ge d_i.
\]
Hence
\[
\operatorname{pers}(p_k)=b_k-d_k\le b_i-d_i=\operatorname{pers}(p_i).
\]
Moreover, suppose $\operatorname{pers}(p_i)=\operatorname{pers}(p_k)$.
Then both inequalities above are equalities. Since the birth
values are distinct, this forces \(k=i\). Thus every off-diagonal nonzero
coefficient goes from an input bar to an output bar of strictly smaller
persistence.

Ordering \(J\) by nonincreasing persistence, and using the corresponding ordered
bases the matrix of
$\pi_J^\delta\circ(\zeta_R^\delta)_*|_{V_J}
$
is therefore lower triangular. Its diagonal entries are precisely the nonzero
coefficients of \([\gamma_i^\delta]\) in \((\zeta_R^\delta)_*[\gamma_i]\). Hence
the matrix is lower triangular with nonzero diagonal, and so has rank \(\#J\).
\end{proof}

\paragraph{The Hall-rank argument.}
Let \(\Phi\colon R\to U_\varepsilon(S)\) and
\(\Psi\colon S\to U_\varepsilon(R)\) be an \(\varepsilon\)-interleaving. Choose, for \(R\) and \(S\), the cycle
representative families given by \Cref{lem:cycle_representatives_ext1_bars},
\[
\{\gamma_i^R\}_{i\in I_R},
\qquad
\{\gamma_j^S\}_{j\in I_S}.
\]
Write
\[
p_i=(b_i^R,d_i^R)
=
\bigl(b(\gamma_i^R),d(\gamma_i^R)\bigr),
\qquad
q_j=(b_j^S,d_j^S)
=
\bigl(b(\gamma_j^S),d(\gamma_j^S)\bigr).
\]

For \(J\subseteq \Ext_1(R)\), set
\[
N_{2\varepsilon}(J)
\coloneqq
\{q_j\in\Ext_1(S):\exists p_i\in J
\text{ such that }\|p_i-q_j\|_\infty\le 2\varepsilon\}.
\]

\begin{lem}[Hall rank inequality for \(\Ext_1\)]
\label{lem:ext1_hall_rank_inequality}
Let \(J\subseteq \Ext_1(R)\) be a set of long \(\Ext_1\)-bars of \(R\),
i.e. bars of persistence larger than \(4\varepsilon\). Then
\[
\#J\le \#N_{2\varepsilon}(J).
\]
The same statement holds with \(R\) and \(S\) exchanged.
\end{lem}

\begin{proof}
Let
\[
V_J
:=
\operatorname{span}\{[\gamma_i^R]:p_i\in J\}
\subseteq H_1(R;\Bbbk).
\]
The interleaving identity gives
\[
R
\xrightarrow{\Phi}
U_\varepsilon(S)
\xrightarrow{U_\varepsilon(\Psi)}
U_{2\varepsilon}(R),
\]
and this composite is the canonical \(2\varepsilon\)-smoothing morphism
\[
\zeta_{R}^{2\varepsilon}\colon R\to U_{2\varepsilon}(R).
\]

For each \(p_i\in J\), let \((\gamma_i^R)^{2\varepsilon}\) denote the
newborn-cycle representative in \(H_1(U_{2\varepsilon}(R);\Bbbk)\)
corresponding, by the barcode-level coupling, to the surviving bar
\[
p_i^{2\varepsilon}
=
(b_i^R-2\varepsilon,d_i^R+2\varepsilon).
\]
Set
\[
V_J^{2\varepsilon}
:=
\operatorname{span}\{[(\gamma_i^R)^{2\varepsilon}]:p_i\in J\}
\subseteq H_1(U_{2\varepsilon}(R);\Bbbk),
\]
and let
\(\pi_J\colon H_1(U_{2\varepsilon}(R);\Bbbk)\to V_J^{2\varepsilon}\)
be the coordinate projection with respect to the newborn-cycle basis of
\(H_1(U_{2\varepsilon}(R);\Bbbk)\).

Define
\[
\Theta_J
:=
\pi_J\circ (U_\varepsilon(\Psi))_*\circ \Phi_*|_{V_J}.
\]
Since
\[
(U_\varepsilon(\Psi))_*\circ\Phi_*=(\zeta_{R}^{2\varepsilon})_*,
\]
we have
\[
\Theta_J
=
\pi_J\circ(\zeta_{R}^{2\varepsilon})_*|_{V_J}.
\]
Because every bar in \(J\) is long, it survives \(2\varepsilon\)-smoothing.
Therefore, by \Cref{lem:triangularity_self_smoothing_ext1}, applied to
\(R\) with \(\delta=2\varepsilon\), the matrix of \(\Theta_J\), ordered by
nonincreasing persistence, is lower triangular with nonzero diagonal entries. Hence
$\operatorname{rank} \Theta_J=\#J.
$

Now consider the intermediate space \(W:=H_1(U_\varepsilon(S);\Bbbk)\).
The newborn-cycle basis of \(W\) is indexed by the \(\Ext_1\)-bars of \(S\) that
survive \(\varepsilon\)-smoothing. If a bar \(q_j=(b_j^S,d_j^S)\) of \(S\)
survives \(\varepsilon\)-smoothing, let \((\gamma_j^S)^\varepsilon\) be the
corresponding newborn-cycle representative in \(W\). Its coordinates in
the smoothed barcode are
\[
q_j^\varepsilon=(b_j^S-\varepsilon,d_j^S+\varepsilon).
\]

Define
\[
W_{\mathrm{good}}
:=
\operatorname{span}\{[(\gamma_j^S)^\varepsilon]:
q_j\in N_{2\varepsilon}(J),\ q_j
\text{ survives }\varepsilon\text{-smoothing}\},
\]
\[
W_{\mathrm{bad}}
:=
\operatorname{span}\{[(\gamma_j^S)^\varepsilon]:
q_j\notin N_{2\varepsilon}(J),\ q_j
\text{ survives }\varepsilon\text{-smoothing}\}.
\]
Thus
\[
W=W_{\mathrm{good}}\oplus W_{\mathrm{bad}}
\]
with respect to the newborn-cycle basis of \(W\). Let
\(\rho_{\mathrm{good}}\colon W\to W_{\mathrm{good}}\) and
\(\rho_{\mathrm{bad}}\colon W\to W_{\mathrm{bad}}\) be the corresponding
coordinate projections. We define
\[
C_{\mathrm{good}}
:=
\pi_J\circ (U_\varepsilon(\Psi))_*
\circ \rho_{\mathrm{good}}\circ \Phi_*|_{V_J},
\]
and
\[
C_{\mathrm{bad}}
:=
\pi_J\circ (U_\varepsilon(\Psi))_*
\circ \rho_{\mathrm{bad}}\circ \Phi_*|_{V_J}.
\]
Since
\[
\rho_{\mathrm{good}}+\rho_{\mathrm{bad}}=\operatorname{id}_W,
\]
we have
\[
\Theta_J=C_{\mathrm{good}}+C_{\mathrm{bad}}.
\]

Since \(C_{\mathrm{good}}\) factors through \(W_{\mathrm{good}}\), and
\[
\dim W_{\mathrm{good}}
\le
\#N_{2\varepsilon}(J),
\]
we have
\[
\operatorname{rank} C_{\mathrm{good}}
\le
\#N_{2\varepsilon}(J).
\]
We prove that
\[
\operatorname{rank} C_{\mathrm{good}}=\#J.
\]

Order the bars in \(J\) by nonincreasing persistence, and use the same order on
the target bars
$p_i^{2\varepsilon},
p_i\in J.
$

We claim that, with respect to this ordering, the matrix of
\(C_{\mathrm{bad}}\) is strictly triangular.

Let \(p_i\in J\), let \(p_k\in J\), and let \(q_j\notin N_{2\varepsilon}(J)\)
be a bar of \(S\) which survives \(\varepsilon\)-smoothing. Suppose that the
coefficient of \([(\gamma_j^S)^\varepsilon]\) in
\[
\Phi_*[\gamma_i^R]\in H_1(U_\varepsilon(S);\Bbbk)
\]
is nonzero, and that the coefficient of
\([(\gamma_k^R)^{2\varepsilon}]\) in
\[
(U_\varepsilon(\Psi))_*[(\gamma_j^S)^\varepsilon]
\in H_1(U_{2\varepsilon}(R);\Bbbk)
\]
is nonzero.  We now show that
\[
\operatorname{pers}(p_k)
<
\operatorname{pers}(p_i)-2\varepsilon.
\]

First, we use the sublevel and superlevel constraints coming from the first
coefficient. Since
\[
[\gamma_i^R]\in F_{R}^{\le b_i^R}\cap G_{R}^{\ge d_i^R},
\]
and \(\Phi\) is a height-preserving morphism,
we have
\[
\Phi_*[\gamma_i^R]
\in
\widetilde F_{S,\varepsilon}^{\le b_i^R}
\cap
\widetilde G_{S,\varepsilon}^{\ge d_i^R}.
\]
With respect to the newborn-cycle basis of
\(H_1(U_\varepsilon(S);\Bbbk)\), these two subspaces are described by
\[
\widetilde F_{S,\varepsilon}^{\le t}
=
\operatorname{span}
\{[(\gamma_\ell^S)^\varepsilon]:
b_\ell^S-d_\ell^S>2\varepsilon
\text{ and } b_\ell^S-\varepsilon\le t\},
\]
and
\[
\widetilde G_{S,\varepsilon}^{\ge t}
=
\operatorname{span}
\{[(\gamma_\ell^S)^\varepsilon]:
b_\ell^S-d_\ell^S>2\varepsilon
\text{ and } d_\ell^S+\varepsilon\ge t\}.
\]
Therefore, if the coefficient of
\([(\gamma_j^S)^\varepsilon]\) in \(\Phi_*[\gamma_i^R]\) is nonzero, then this
basis element must occur in both of the above spans with
\(t=b_i^R\) and \(t=d_i^R\), respectively. Hence
\[
b_j^S-\varepsilon\le b_i^R,
\qquad
d_j^S+\varepsilon\ge d_i^R.
\]
Equivalently,
\[
b_j^S\le b_i^R+\varepsilon,
\qquad
 d_j^S\ge d_i^R-\varepsilon.
\]
Since \([(\gamma_j^S)^\varepsilon]\) is bad, that is
$q_j\notin N_{2\varepsilon}(J)
$
we have,
\[
\|q_j-p_i\|_\infty>2\varepsilon.
\]
Together with the two inequalities above, this forces either
\[
b_j^S<b_i^R-2\varepsilon
\qquad\text{or}\qquad
 d_j^S>d_i^R+2\varepsilon.
\]
Indeed, the alternatives
\[
b_j^S>b_i^R+2\varepsilon
\qquad\text{or}\qquad
 d_j^S<d_i^R-2\varepsilon
\]
are excluded by
\[
b_j^S\le b_i^R+\varepsilon,
\qquad
 d_j^S\ge d_i^R-\varepsilon.
\]
In the first remaining case,
\[
b_j^S-d_j^S
<
(b_i^R-2\varepsilon)-(d_i^R-\varepsilon)
=
(b_i^R-d_i^R)-\varepsilon.
\]
In the second remaining case,
\[
b_j^S-d_j^S
<
(b_i^R+\varepsilon)-(d_i^R+2\varepsilon)
=
(b_i^R-d_i^R)-\varepsilon.
\]
Thus
\[
\operatorname{pers}(q_j)
<
\operatorname{pers}(p_i)-\varepsilon.
\]

Second, we use the sublevel and superlevel constraints coming from the second
coefficient. Since \(U_\varepsilon(\Psi)\) is height-preserving 
and
\[
[(\gamma_j^S)^\varepsilon]
\in
\widetilde F_{S,\varepsilon}^{\le b_j^S-\varepsilon}
\cap
\widetilde G_{S,\varepsilon}^{\ge d_j^S+\varepsilon},
\]
we have
\[
(U_\varepsilon(\Psi))_*[(\gamma_j^S)^\varepsilon]
\in
\widetilde F_{R,2\varepsilon}^{\le b_j^S-\varepsilon}
\cap
\widetilde G_{R,2\varepsilon}^{\ge d_j^S+\varepsilon}.
\]
With respect to the newborn-cycle basis of
\(H_1(U_{2\varepsilon}(R);\Bbbk)\), these two subspaces are described by
\[
\widetilde F_{R,2\varepsilon}^{\le t}
=
\operatorname{span}
\{[(\gamma_\ell^R)^{2\varepsilon}]:
b_\ell^R-d_\ell^R>4\varepsilon
\text{ and } b_\ell^R-2\varepsilon\le t\},
\]
and
\[
\widetilde G_{R,2\varepsilon}^{\ge t}
=
\operatorname{span}
\{[(\gamma_\ell^R)^{2\varepsilon}]:
b_\ell^R-d_\ell^R>4\varepsilon
\text{ and } d_\ell^R+2\varepsilon\ge t\}.
\]
Therefore, if the coefficient of
\([(\gamma_k^R)^{2\varepsilon}]\) in
$(U_\varepsilon(\Psi))_*[(\gamma_j^S)^\varepsilon]
$
is nonzero, then this basis element must occur in both of the above spans with
\(t=b_j^S-\varepsilon\) and \(t=d_j^S+\varepsilon\), respectively. Hence
\[
b_k^R-2\varepsilon\le b_j^S-\varepsilon,
\qquad
d_k^R+2\varepsilon\ge d_j^S+\varepsilon.
\]
Equivalently,
\[
b_j^S\ge b_k^R-\varepsilon,
\qquad
d_j^S\le d_k^R+\varepsilon.
\]
Again, since \(q_j\notin N_{2\varepsilon}(J)\), we have in particular
\[
\|q_j-p_k\|_\infty>2\varepsilon.
\]
Together with the previous two inequalities, this forces either
\[
b_j^S>b_k^R+2\varepsilon
\qquad\text{or}\qquad
 d_j^S<d_k^R-2\varepsilon.
\]
Indeed, the alternatives
\[
b_j^S<b_k^R-2\varepsilon
\qquad\text{or}\qquad
 d_j^S>d_k^R+2\varepsilon
\]
are excluded by
\[
b_j^S\ge b_k^R-\varepsilon,
\qquad
 d_j^S\le d_k^R+\varepsilon.
\]
In the first remaining case,
\[
b_j^S-d_j^S
>
(b_k^R+2\varepsilon)-(d_k^R+\varepsilon)
=
(b_k^R-d_k^R)+\varepsilon.
\]
In the second remaining case,
\[
b_j^S-d_j^S
>
(b_k^R-\varepsilon)-(d_k^R-2\varepsilon)
=
(b_k^R-d_k^R)+\varepsilon.
\]
Thus
\[
\operatorname{pers}(q_j)
>
\operatorname{pers}(p_k)+\varepsilon.
\]

Combining the two persistence inequalities gives
\[
\operatorname{pers}(p_k)
<
\operatorname{pers}(p_i)-2\varepsilon.
\]
Therefore every nonzero matrix entry of \(C_{\mathrm{bad}}\) goes from an input
bar to an output bar of strictly smaller persistence. In particular, using
persistence-ordered bases, a nonzero entry of \(C_{\mathrm{bad}}\) can occur only
strictly below the diagonal. Hence \(C_{\mathrm{bad}}\) is strictly lower
triangular.

On the other hand, by \Cref{lem:triangularity_self_smoothing_ext1}, the matrix
of \(\Theta_J\) in the same ordered bases is lower triangular with nonzero diagonal
entries. Since
$C_{\mathrm{good}}=\Theta_J-C_{\mathrm{bad}},
$
and since \(C_{\mathrm{bad}}\) has zero diagonal, the matrix of
\(C_{\mathrm{good}}\) is lower triangular and has the same diagonal entries as
\(\Theta_J\). In particular, all diagonal entries of \(C_{\mathrm{good}}\) are
nonzero. Therefore
\[
\operatorname{rank} C_{\mathrm{good}}=\#J.
\]
Combining this with
\[
\operatorname{rank} C_{\mathrm{good}}
\le
\#N_{2\varepsilon}(J)
\]
gives
\[
\#J\le \#N_{2\varepsilon}(J).
\]

The proof with \(R\) and \(S\) exchanged is identical, using the other
interleaving composite
\[
S\to U_\varepsilon(R)\to U_{2\varepsilon}(S).
\]
\end{proof}

\begin{lem}[Two-sided Hall criterion for partial bottleneck matchings]
\label{lem:two_sided_hall_partial_matching}
Let \(A\) and \(B\) be finite sets, let \(A_0\subseteq A\) and
\(B_0\subseteq B\), and let \(E\subseteq A\times B\) be a bipartite graph. For
\(I\subseteq A\) and \(J\subseteq B\), write
\[
N(I):=\{b\in B:\exists a\in I\text{ with }(a,b)\in E\},
\]
and
\[
N(J):=\{a\in A:\exists b\in J\text{ with }(a,b)\in E\}.
\]
Assume that
\[
\#I\le \#N(I)
\qquad\text{for every }I\subseteq A_0,
\]
and
\[
\#J\le \#N(J)
\qquad\text{for every }J\subseteq B_0.
\]
Then there exists a matching \(M\subseteq E\) which saturates every vertex of
\(A_0\) and every vertex of \(B_0\).
\end{lem}

\begin{proof}
Recall that a matching in \(E\) is a subset \(M\subseteq E\) such that no two
edges in \(M\) share an endpoint. A vertex is saturated by \(M\) if it is
incident to an edge of \(M\). Thus saying that \(M\) saturates \(A_0\) means
that every \(a\in A_0\) is matched by \(M\) to some vertex of \(B\), not
necessarily to a vertex of \(B_0\).

By the first Hall inequality and the ordinary Hall theorem, there exists a
matching \(M\subseteq E\) which saturates \(A_0\). Among all matchings
\(M\subseteq E\) saturating \(A_0\), choose one which saturates the largest
possible number of vertices of \(B_0\).

We claim that \(M\) also saturates \(B_0\). Suppose not, and choose an
unsaturated vertex \(b_0\in B_0\). Consider \(M\)-alternating paths starting at
\(b_0\) whose first edge is not in \(M\). That is, a path
\[
b_0,a_1,b_1,a_2,b_2,\ldots
\]
such that its edges alternately lie outside and inside \(M\), starting with an
edge outside \(M\).

Let \(B'\subseteq B_0\) be the set of vertices of \(B_0\) reachable from
\(b_0\) by such paths, and let \(A'\coloneqq N(B')\subseteq A\).

If some \(a\in A'\) were unsaturated by \(M\), then \(a\notin A_0\), because
\(M\) saturates \(A_0\). Replacing the edges of \(M\) along the alternating path
from \(b_0\) to \(a\) by the edges outside \(M\) along the same path would then
produce another matching in \(E\) still saturating \(A_0\), but saturating one
more vertex of \(B_0\), contradicting the choice of \(M\). Hence every vertex of
\(A'\) is saturated by \(M\).

Moreover, if \(a\in A'\) is matched by \(M\) to a vertex \(b\in B\), then \(b\) is
reachable from \(b_0\) by an alternating path: first reach a vertex of \(B'\)
adjacent to \(a\), then traverse the matched edge from \(a\) to \(b\). If
\(b\notin B_0\), then replacing this matched edge by the preceding unmatched
edge and stopping at \(b\) would again give a matching in \(E\) saturating
\(A_0\) and saturating one more vertex of \(B_0\), contradicting the choice of
\(M\). Thus \(b\in B'\).

Therefore every vertex of \(A'=N(B')\) is matched to a vertex of \(B'\). Since
\(b_0\in B'\) is not saturated, these matched partners lie in the proper subset
\(B'\setminus\{b_0\}\). Hence
\[
\#N(B')=\#A'\le \#B'-1<\#B',
\]
contradicting the second Hall inequality applied to \(B'\subseteq B_0\). Thus
\(M\) saturates all vertices of \(B_0\). Since \(M\) was chosen among matchings
saturating \(A_0\), it saturates every vertex of \(A_0\) and every vertex of
\(B_0\).
\end{proof}

\begin{prop}[The \(\Ext_1\)-component]
\label{prop:ext1_component_2epsilon_bound}
For every \(\varepsilon\)-interleaving
\[
\Phi\colon R\to U_\varepsilon(S),
\qquad
\Psi\colon S\to U_\varepsilon(R),
\]
one has
\[
d_B\bigl(\Ext_1(R),\Ext_1(S)\bigr)
\le 2\varepsilon.
\]
\end{prop}

\begin{proof}
Let \(P_R:=\Ext_1(R)\) and \(P_S:=\Ext_1(S)\).
Let
\[
P_R^{\mathrm{long}}
:=
\{p\in P_R:\operatorname{pers}(p)>4\varepsilon\},
\qquad
P_S^{\mathrm{long}}
:=
\{q\in P_S:\operatorname{pers}(q)>4\varepsilon\}.
\]
Consider the bipartite graph with left vertex set \(P_R\), right vertex set
\(P_S\), and an edge between \(p\in P_R\) and \(q\in P_S\) whenever
$\|p-q\|_\infty\le 2\varepsilon.
$
For a subset
$J\subseteq P_R^{\mathrm{long}},
$
its neighbourhood in this graph is exactly \(N_{2\varepsilon}(J)\). Hence
\Cref{lem:ext1_hall_rank_inequality} gives
\[
\#J\le \#N_{2\varepsilon}(J)
\qquad
\text{for every }J\subseteq P_R^{\mathrm{long}}.
\]
The symmetric part of \Cref{lem:ext1_hall_rank_inequality} gives the analogous
inequality for every subset of \(P_S^{\mathrm{long}}\). Therefore
\Cref{lem:two_sided_hall_partial_matching} gives a matching between points of
\(P_R\) and points of \(P_S\) which saturates all long bars on both sides, and
every matched pair has \(\ell^\infty\)-distance at most \(2\varepsilon\).

The only off-diagonal points not saturated by this matching are short bars, that
is, bars of persistence at most \(4\varepsilon\). By
\Cref{lem:short_ext1_bars_diagonal_disposable}, each such bar has diagonal cost
at most \(2\varepsilon\). Matching all remaining short bars to the diagonal
therefore extends the previous partial matching to a bottleneck matching of cost
at most \(2\varepsilon\). Hence
$d_B(P_R,P_S)\le 2\varepsilon.
$
\end{proof}

\subsection{Conclusion of the comparison}
\label{app:conclusion_full_comparison}

We now combine the \(H_0\)-related estimates with the \(\Ext_1\)-estimate proved above. For every
\(\varepsilon\)-interleaving
\[
\Phi\colon R\to U_\varepsilon(S),
\qquad
\Psi\colon S\to U_\varepsilon(R),
\]
by \Cref{eq:extended_H0_bound}, and
\Cref{prop:ext1_component_2epsilon_bound} we have $d_\Delta(R,S)\le2\varepsilon.
$

Letting \(\varepsilon\downarrow d_I(R,S)\), we obtain
\[
d_\Delta(R,S)\le2d_I(R,S)
\]
for combinatorially generic Reeb graphs.

We now remove the genericity assumption. Let \(R^{(n)}\to R\) and
\(S^{(n)}\to S\) be combinatorially generic perturbations in functional
distortion, whose existence follows from the density result proved in \Cref{subsec:h1_genericity_newborn}. Set
\[
a_n\coloneqq d_{FD}(R,R^{(n)}),
\qquad
b_n\coloneqq d_{FD}(S,S^{(n)}),
\]
so that \(a_n,b_n\to0\). By the comparison estimates with \(d_{FD}\),
\[
d_I(R,R^{(n)}),\,d_\Delta(R,R^{(n)})\le 3a_n,
\qquad
d_I(S,S^{(n)}),\,d_\Delta(S,S^{(n)})\le 3b_n.
\]
Using the triangle inequality for \(d_\Delta\), the generic comparison, and then
the triangle inequality for \(d_I\), we obtain
\[
\begin{aligned}
d_\Delta(R,S)
&\le
d_\Delta(R,R^{(n)})
+d_\Delta(R^{(n)},S^{(n)})
+d_\Delta(S^{(n)},S)\\
&\le
3a_n+2d_I(R^{(n)},S^{(n)})+3b_n\\
&\le
3a_n+2\bigl(d_I(R^{(n)},R)+d_I(R,S)
+d_I(S,S^{(n)})\bigr)+3b_n\\
&\le
2d_I(R,S)+9a_n+9b_n.
\end{aligned}
\]
Letting \(n\to\infty\) gives
\[
d_\Delta(R,S)\le 2d_I(R,S).
\]

Equivalently, via the equivalence between Reeb graphs and constructible cosheaves,
we have proved the following comparison.

\begin{thm}\label{thm:dDelta_le_2dI}
For any Reeb graphs \(R\) and \(S\),
\[
d_\Delta\!\left(\mathcal F_R,\mathcal F_S\right)
\le
2\,d_I\!\left(\mathcal F_R,\mathcal F_S\right).
\]
Equivalently,
\[
d_\Delta(R,S)
\le
2d_I(R,S).
\]
In particular, the same conclusion holds for
\(R=\mathfrak R(X,f)\) and \(S=\mathfrak R(Y,g)\), for any constructible
\(\mathbb R\)-spaces \((X,f)\) and \((Y,g)\). Moreover, the components
\(\Ord_0\), \(\Ext_0\), and \(\Rel_1\) satisfy the sharper constant \(1\), while
the factor \(2\) is only used for the \(\Ext_1\)-component. The constant \(2\)
is sharp.
\end{thm}

\section{PL-Reeb estimators and Mapper coarsenings}
\label{sec:Mapper_cosheaf}

This section constructs sample-based Reeb-type estimators from proximity graphs on
a finite sample $S_n$, and then applies Mapper as a controlled coarsening of these estimators.
We keep the intrinsic and extrinsic proximity metrics in parallel, since they
lead to different covering assumptions and stability bounds.

\subsection{Setting}
\label{sec:delta_mapper_setting}

Throughout the section, let \(X\subset \mathbb{R}^m\) be a compact subset such
that the intrinsic distance \(d_X\), defined in~\Cref{eq:d_X}, is finite on
\(X\times X\), and let $S_n\subseteq X$ be a finite sample. Let
$f\colon X\to \mathbb{R}
$
be a continuous function such that \((X,f)\) is a constructible
\(\mathbb{R}\)-space. Throughout this section, we write
$R\coloneqq\mathfrak R(X,f)
$
for its Reeb graph. We set  $f_n\coloneqq f|_{S_n}$.
 We fix a non-decreasing modulus of continuity
\[
\omega_f\colon [0,+\infty)\to[0,+\infty)
\]
for \(f\) with respect to \(d_X\), meaning that
\[
\omega_f(0)=0,
\qquad
|f(x)-f(y)|\le \omega_f\bigl(d_X(x,y)\bigr)
\qquad
\text{for every }x,y\in X.
\]

For the stability estimates we need to control the filter along short intrinsic
geodesics, and also compare it with affine interpolation on graph edges. For
\(\delta>0\), define
\[
\theta_f(\delta)
\coloneqq
\sup
\max\left\{
\frac12 |f(x)-f(y)|,\,
\sup_{t\in[0,1]}
\left|
f(\gamma_{xy}(t))
-
\bigl((1-t)f(x)+t f(y)\bigr)
\right|
\right\},
\]
where the supremum is taken over all pairs \(x,y\in X\) with
\(d_X(x,y)\le 2\delta\), and where \(\gamma_{xy}\) is a chosen minimizing
\(d_X\)-geodesic from \(x\) to \(y\), parametrized proportionally to arc length.
Such geodesics exist for the pairs considered below by the discussion in
\Cref{sec:metric_equivalence}. We set
\[
\mu_\delta\coloneqq
\max\{\omega_f(\delta),\theta_f(\delta)\}.
\]

We record some estimates encoded in this notation. Let
\(x,y\in X\), let \(L=d_X(x,y)\le 2\delta\), and let \(\gamma_{xy}\) be
parametrized proportionally to arc length. If \(p\) is the point with affine
coordinate \(t\in[0,1]\) on an abstract edge \([x,y]\), and if the edge is sent
to \(\gamma_{xy}\), then
\begin{equation}\label{eq:theta_edge_estimate}
\left|
f(\gamma_{xy}(t))
-
\bigl((1-t)f(x)+t f(y)\bigr)
\right|
\le
\theta_f(\delta).
\end{equation}
Moreover, if \(v\in\{x,y\}\) is an endpoint such that the subsegment from \(v\)
to \(\gamma_{xy}(t)\) has length at most \(L/2\), then
\begin{equation}\label{eq:theta_endpoint_estimate}
\left|
\bigl((1-t)f(x)+t f(y)\bigr)-f(v)
\right|
\le
\frac12 |f(x)-f(y)|
\le
\theta_f(\delta).
\end{equation}

The same quantities are always controlled, more coarsely, by the modulus at the
doubled scale. Indeed, for \(z=\gamma_{xy}(t)\),
\[
d_X(z,x)=tL,
\qquad
d_X(z,y)=(1-t)L.
\]
Therefore
\[
\begin{aligned}
&\left|
f(\gamma_{xy}(t))
-
\bigl((1-t)f(x)+t f(y)\bigr)
\right|\\
&\qquad=
\left|
(1-t)\bigl(f(z)-f(x)\bigr)
+
t\bigl(f(z)-f(y)\bigr)
\right|\\
&\qquad\le
(1-t)\,\omega_f(tL)
+
t\,\omega_f((1-t)L)\\
&\qquad\le
(1-t)\,\omega_f(L)+t\,\omega_f(L)
=
\omega_f(L)
\le
\omega_f(2\delta).
\end{aligned}
\]
Together with
\[
\frac12 |f(x)-f(y)|
\le
\omega_f(L)
\le
\omega_f(2\delta),
\]
this gives
\begin{equation}\label{eq:theta_doubled_modulus}
\theta_f(\delta)\le \omega_f(2\delta).
\end{equation}

If \(\omega_f\) is concave, then the preceding estimate improves to
\begin{equation}\label{eq:theta_concave_modulus}
\theta_f(\delta)\le\omega_f(\delta),
\end{equation}
and hence \(\mu_\delta=\omega_f(\delta)\). Indeed,
\[
\frac12|f(x)-f(y)|
\le
\frac12\omega_f(L)
\le
\omega_f(L/2)
\le
\omega_f(\delta),
\]
where we used concavity and \(\omega_f(0)=0\). For the affine-interpolation term,
the estimate displayed above gives
\[
\left|
f(\gamma_{xy}(t))
-
\bigl((1-t)f(x)+t f(y)\bigr)
\right|
\le
(1-t)\,\omega_f(tL)+t\,\omega_f((1-t)L).
\]
By concavity of \(\omega_f\),
\[
(1-t)\,\omega_f(tL)+t\,\omega_f((1-t)L)
\le
\omega_f\bigl((1-t)tL+t(1-t)L\bigr)
=
\omega_f\bigl(2t(1-t)L\bigr).
\]
Since \(2t(1-t)\le 1/2\), monotonicity gives
\[
\omega_f\bigl(2t(1-t)L\bigr)
\le
\omega_f(L/2)
\le
\omega_f(\delta).
\]
Thus \(\theta_f(\delta)\le\omega_f(\delta)\). In particular, if \(f\) is
\(1\)-Lipschitz with respect to \(d_X\), then \(\mu_\delta\le \delta\).

If \(X\) has positive reach \(\tau=\rch(X)>0\) and \(0<\delta<\tau\), we also set
\[
\eta_\tau(\delta)\coloneqq\psi_\tau(2\delta)
=
2\tau\arcsin\!\Bigl(\frac{\delta}{\tau}\Bigr).
\]

To conclude this section, we point out the following fact which will be needed in the following.

\begin{oss}[Intersection graph of a finite closed cover]
\label{rmk:finite_closed_cover_intersection_graph}
Let \(Y\) be a connected space, and let
\[
Y=F_1\cup\cdots\cup F_N
\]
be a finite cover by non-empty closed subsets. We call the \(1\)-skeleton of the
nerve of this cover its intersection graph: its vertices are the sets \(F_i\),
and two vertices \(F_i,F_j\) are joined by an edge if and only if
$F_i\cap F_j\neq\varnothing.
$
Then this intersection graph is connected. Indeed, if it were disconnected, the
sets \(F_i\) could be partitioned into two non-empty classes with no
intersections between the two classes. The unions of the sets in the two classes
would then give a separation of \(Y\) into two disjoint non-empty closed subsets,
contradicting connectedness.
\end{oss}

\subsection{PL-Reeb estimators}
\label{sec:pl_reeb_estimators}

For \(\rho\in\{d_X,\|\cdot\|\}\), let \(\Gamma_\delta^{S_n,\rho}\) be the graph
with vertex set \(S_n\), in which two distinct vertices \(x,y\in S_n\) are joined
by an edge if and only if
$\rho(x,y)\le 2\delta.
$

Let
$\hat f_n^\rho\colon |\Gamma_\delta^{S_n,\rho}|\to\mathbb R
$
be the piecewise-linear extension of \(f_n\) to the geometric realization. We define the intrinsic and extrinsic PL-Reeb
estimators by
\[
\mathcal R_{\delta}^{S_n,d_X}
\coloneqq
\mathcal F_{(|\Gamma_\delta^{S_n,d_X}|,\hat f_n^{d_X})},
\qquad
\mathcal R_{\delta}^{S_n,\|\cdot\|}
\coloneqq
\mathcal F_{(|\Gamma_\delta^{S_n,\|\cdot\|}|,\hat f_n^{\|\cdot\|})}.
\]
Equivalently, for every \(I\in\mathbf{Int}\),
\[
\mathcal R_{\delta}^{S_n,\rho}(I)
=
\pi_0\!\left((\hat f_n^\rho)^{-1}(I)\right).
\]
By definition, \(\mathcal R_{\delta}^{S_n,\rho}\) is a constructible cosheaf:
it is the Reeb cosheaf of the finite filtered graph
\((|\Gamma_\delta^{S_n,\rho}|,\hat f_n^\rho)\). The raw filtered graph need not
itself be a Reeb graph in the strict sense, since \(\hat f_n^\rho\) may be
constant on some edges. Equivalently, one may pass to the Reeb quotient
$\mathfrak R(|\Gamma_\delta^{S_n,\rho}|,\hat f_n^\rho),
$
which collapses all connected horizontal level components and gives the Reeb
graph representing the same cosheaf. For this reason, in the arguments below we work mostly with the filtered
proximity graph \((|\Gamma_\delta^{S_n,\rho}|,\hat f_n^\rho)\) and its
associated Reeb cosheaf, relying on \Cref{eq:cosheaf_identity}, rather than
explicitly with its Reeb quotient.

\subsection{Stability of the PL-Reeb estimators}

Now we present the stability results for the PL-Reeb estimators.

\begin{thm}[Intrinsic PL-Reeb stability]\label{thm:intrinsic_reeb_stability}
Under the setting of \Cref{sec:delta_mapper_setting}, assume that \(\delta\)
is an intrinsic covering radius for \(S_n\), namely
\[
\forall x\in X,\ \exists s\in S_n
\qquad
 d_X(x,s)\le \delta.
\]
Then
\[
d_I\!\left(
\mathcal F_R,
\mathcal R_{\delta}^{S_n,d_X}
\right)
\le
\mu_\delta.
\]
If \(\omega_f\) is concave and \(\omega_f(0)=0\), then
\[
d_I\!\left(
\mathcal F_R,
\mathcal R_{\delta}^{S_n,d_X}
\right)
\le
\omega_f(\delta).
\]
In particular, if \(f\) is \(1\)-Lipschitz with respect to \(d_X\), then
\[
d_I\!\left(
\mathcal F_R,
\mathcal R_{\delta}^{S_n,d_X}
\right)
\le
\delta.
\]
\end{thm}

\begin{proof}
Write
\[
\Gamma_\delta=\Gamma_\delta^{S_n,d_X},
\qquad
\hat f_n=\hat f_n^{d_X},
\qquad
\mu=\mu_\delta.
\]
For every edge \([x,y]\) of \(\Gamma_\delta\), choose the minimizing
\(d_X\)-geodesic \(\gamma_{xy}\) used in the definition of \(\theta_f\),
parametrized proportionally to arc length. These choices define a continuous map
\[
\Phi\colon |\Gamma_\delta|\to X
\]
which is the identity on vertices and sends the point with affine coordinate
\(t\) on the edge \([x,y]\) to \(\gamma_{xy}(t)\). Since every edge of
\(\Gamma_\delta\) satisfies \(d_X(x,y)\le 2\delta\),
\eqref{eq:theta_edge_estimate} gives, for every \(p\in|\Gamma_\delta|\),
\[
|f(\Phi(p))-\hat f_n(p)|\le \theta_f(\delta)\le \mu.
\]

We construct an interleaving between \(\mathcal R_{\delta}^{S_n,d_X}\) and
\(\mathcal F_R\). See also \Cref{fig:stability} (left) for a visual representation. We first define the morphism from the graph estimator to
the Reeb cosheaf of \((X,f)\). Let \(I\in\mathbf{Int}\), and let \(A\) be a
path-connected component of
\[
(\hat f_n)^{-1}(I)\subseteq |\Gamma_\delta|.
\]
For every \(p\in A\), we have \(\hat f_n(p)\in I\), and the estimate above gives
$f(\Phi(p))\in I^\mu.
$

Hence
\[
\Phi(A)\subseteq f^{-1}(I^\mu).
\]
Moreover, \(\Phi(A)\) is path-connected because \(A\) is path-connected and
\(\Phi\) is continuous. Therefore \(\Phi(A)\) is contained in a unique
path-connected component of \(f^{-1}(I^\mu)\). Denote this component by
$\alpha_I(A).
$

The assignment is compatible with inclusions of intervals. Indeed, if
\(I\subseteq J\) and \(A'\) is the path-connected component of
\((\hat f_n)^{-1}(J)\) containing \(A\), then
\[
\Phi(A)\subseteq \Phi(A').
\]
Thus the component \(\alpha_I(A)\) maps to the component \(\alpha_J(A')\) under
the canonical map
\[
\pi_0(f^{-1}(I^\mu))\longrightarrow \pi_0(f^{-1}(J^\mu)).
\]
Hence the maps \(\alpha_I\) define a natural transformation
\[
\alpha\colon
\mathcal R_{\delta}^{S_n,d_X}
\Rightarrow
\mathcal S_\mu\mathcal F_R.
\]

We now define the morphism in the opposite direction. Let \(I\in\mathbf{Int}\),
and let \(C\) be a path-connected component of \(f^{-1}(I)\). Set
\[
S(C)\coloneqq
\{s\in S_n\mid B_{d_X}(s,\delta)\cap C\neq\varnothing\}.
\]
The set \(S(C)\) is non-empty by the covering-radius assumption. If \(s\in S(C)\),
choose \(c_s\in C\) with \(d_X(s,c_s)\le\delta\). Since \(f_n(s)=f(s)\), we have
\[
|f_n(s)-f(c_s)|
=
|f(s)-f(c_s)|
\le \omega_f(\delta)
\le \mu.
\]
Because \(f(c_s)\in I\), it follows that
$f_n(s)\in I^\mu.
$

We claim that \(S(C)\) spans a connected subgraph of \(\Gamma_\delta\), and that
this subgraph is contained in \((\hat f_n)^{-1}(I^\mu)\). Since \(S_n\) is finite,
the family
\[
\{C\cap B_{d_X}(s,\delta)\}_{s\in S(C)}
\]
is a finite closed cover of the connected space \(C\), by
\Cref{lem:intrinsic_balls_closed}, and each member is non-empty. By \Cref{rmk:finite_closed_cover_intersection_graph}, its intersection graph is connected.
Thus, for any two elements \(s,s'\in S(C)\), there is a finite chain
$s=s_0,s_1,\ldots,s_q=s'
$
in \(S(C)\) such that
\[
B_{d_X}(s_r,\delta)\cap B_{d_X}(s_{r+1},\delta)\neq\varnothing
\]
for every \(r\). Hence
$d_X(s_r,s_{r+1})\le 2\delta,
$
so consecutive vertices are joined by edges of \(\Gamma_\delta\). Moreover, the
endpoint values of each such edge lie in \(I^\mu\). Since \(\hat f_n\) is affine
on the edge and \(I^\mu\) is an interval, the whole edge lies in
$(\hat f_n)^{-1}(I^\mu).
$

Therefore the subgraph spanned by \(S(C)\) is connected and lies in
\((\hat f_n)^{-1}(I^\mu)\). It is consequently contained in a unique
path-connected component of \((\hat f_n)^{-1}(I^\mu)\). Denote this component by
$\beta_I(C).
$

This assignment is again compatible with inclusions. If \(I\subseteq J\), and if
\(C'\) is the path-connected component of \(f^{-1}(J)\) containing \(C\), then
$S(C)\subseteq S(C').
$

Hence the subgraph spanned by \(S(C)\) is contained in the subgraph spanned by
\(S(C')\), and the component \(\beta_I(C)\) maps to the component
\(\beta_J(C')\) under the canonical map
\[
\pi_0((\hat f_n)^{-1}(I^\mu))
\longrightarrow
\pi_0((\hat f_n)^{-1}(J^\mu)).
\]
Thus the maps \(\beta_I\) define a natural transformation
\[
\beta\colon
\mathcal F_R
\Rightarrow
\mathcal S_\mu\mathcal R_{\delta}^{S_n,d_X}.
\]

It remains to check the two interleaving identities. First, let \(I\in
\mathbf{Int}\), and let \(A\) be a path-connected component of
\((\hat f_n)^{-1}(I)\). Let
\[
C_A\coloneqq \alpha_I(A)
\]
be the path-connected component of \(f^{-1}(I^\mu)\) containing \(\Phi(A)\). We
show that the canonical image of \(A\) in
$\pi_0((\hat f_n)^{-1}(I^{2\mu}))
$
coincides with
$\beta_{I^\mu}(\alpha_I(A)).
$

By definition, \(\beta_{I^\mu}(C_A)\) is the component of
\((\hat f_n)^{-1}(I^{2\mu})\) containing the subgraph spanned by \(S(C_A)\).
Thus it is enough to show that \(A\) meets the path-connected component of
\((\hat f_n)^{-1}(I^{2\mu})\) containing this subgraph. Indeed, \(A\) is
path-connected and
\[
A\subseteq(\hat f_n)^{-1}(I)\subseteq(\hat f_n)^{-1}(I^{2\mu}).
\]

Choose \(p\in A\). If \(p\) is a vertex, then \(\Phi(p)=p\in C_A\). Hence
\(p\in S(C_A)\), so \(p\) already belongs to the subgraph spanned by \(S(C_A)\),
and the desired conclusion follows. We may therefore suppose that \(p\) lies in
the interior of an edge \([x,y]\), and set
$q=\Phi(p).
$

The chosen geodesic from \(x\) to \(y\) has length at most \(2\delta\). Therefore
at least one endpoint \(v\in\{x,y\}\) satisfies
$d_X(v,q)\le\delta.
$

Since \(q\in \Phi(A)\subseteq C_A\), this implies \(v\in S(C_A)\). Moreover, by \eqref{eq:theta_endpoint_estimate},
\[
|\hat f_n(p)-f_n(v)|
\le
\frac12 |f_n(x)-f_n(y)|
\le
\theta_f(\delta)
\le
\mu.
\]
Since \(\hat f_n(p)\in I\), both endpoint values of the graph segment from \(p\)
to \(v\) lie in \(I^\mu\). By affineness of \(\hat f_n\) on the edge, the whole
segment lies in
\[
(\hat f_n)^{-1}(I^\mu)
\subseteq
(\hat f_n)^{-1}(I^{2\mu}).
\]
Thus \(A\) meets the path-connected component of
\((\hat f_n)^{-1}(I^{2\mu})\) containing the subgraph spanned by \(S(C_A)\).
Therefore
$\beta_{I^\mu}\circ\alpha_I
$
is the canonical structure map
\[
\mathcal R_{\delta}^{S_n,d_X}(I)
\longrightarrow
\mathcal R_{\delta}^{S_n,d_X}(I^{2\mu}).
\]

Conversely, let \(I\in\mathbf{Int}\), and let \(C\) be a path-connected component
of \(f^{-1}(I)\). Let
\[
A_C\coloneqq \beta_I(C)
\]
be the path-connected component of \((\hat f_n)^{-1}(I^\mu)\) containing the
subgraph spanned by \(S(C)\). We show that the canonical image of \(C\) in
$\pi_0(f^{-1}(I^{2\mu}))
$
coincides with
$\alpha_{I^\mu}(\beta_I(C)).
$

By definition, \(\alpha_{I^\mu}(A_C)\) is the path-connected component of
\(f^{-1}(I^{2\mu})\) containing \(\Phi(A_C)\). It is therefore enough to show
that \(C\) and \(\Phi(A_C)\) lie in the same path-connected component of
\(f^{-1}(I^{2\mu})\).

First, every point of \(C\) is connected inside \(f^{-1}(I^\mu)\) to a vertex in
\(S(C)\). Indeed, let \(c\in C\). By the covering-radius assumption, choose
\(s\in S_n\) with
$d_X(c,s)\le\delta.
$

Then \(s\in S(C)\), and a minimizing \(d_X\)-geodesic from \(c\) to \(s\) has
length at most \(\delta\). Along this geodesic, the \(f\)-values remain in
$I^{\omega_f(\delta)}
\subseteq
I^\mu.
$
Thus \(c\) is connected inside \(f^{-1}(I^\mu)\) to \(s\).

It remains to compare the graph component \(A_C\) with its image under \(\Phi\).
Since \(A_C\subseteq(\hat f_n)^{-1}(I^\mu)\), the estimate
\[
|f(\Phi(p))-\hat f_n(p)|\le\mu
\]
implies
\[
\Phi(A_C)\subseteq f^{-1}(I^{2\mu}).
\]
Moreover, \(A_C\) is path-connected, hence \(\Phi(A_C)\) is path-connected. The
image \(\Phi(A_C)\) contains every vertex of \(S(C)\), because \(\Phi\) is the
identity on vertices. Since every point of \(C\) is connected inside
\(f^{-1}(I^\mu)\subseteq f^{-1}(I^{2\mu})\) to a vertex of \(S(C)\), and since
\(\Phi(A_C)\) is path-connected inside \(f^{-1}(I^{2\mu})\), the sets \(C\) and
\(\Phi(A_C)\) lie in the same path-connected component of
\(f^{-1}(I^{2\mu})\). 

Therefore
$\alpha_{I^\mu}\circ\beta_I
$
is the canonical structure map
\[
\mathcal F_R(I)
\longrightarrow
\mathcal F_R(I^{2\mu}).
\]

Thus \(\alpha\) and \(\beta\) define a \(\mu\)-interleaving, and hence
\[
d_I\!\left(
\mathcal F_R,
\mathcal R_{\delta}^{S_n,d_X}
\right)
\le \mu_\delta.
\]
The concave-modulus and \(1\)-Lipschitz bounds follow from the estimates on
\(\mu_\delta\) established above.
\end{proof}

\begin{cor}[Intrinsic PL-Reeb stability from a Euclidean covering radius]
\label{cor:intrinsic_pl_stability_from_euclidean_radius}
Under the setting of \Cref{sec:delta_mapper_setting}, assume moreover that
\(X\subset\mathbb R^m\) has positive reach \(\tau=\rch(X)>0\). Let
\(0<\delta<2\tau\), and assume that \(\delta\) is a Euclidean covering radius for
\(S_n\), namely
\[
\forall x\in X,\ \exists s\in S_n
\qquad
\|x-s\|\le \delta.
\]
Set
\[
\delta^X
\coloneqq
\psi_\tau(\delta)
=
2\tau\arcsin\!\left(\frac{\delta}{2\tau}\right).
\]
Then
\[
d_I\!\left(
\mathcal F_R,
\mathcal R_{\delta^X}^{S_n,d_X}
\right)
\le
\mu_{\delta^X}.
\]
If \(\omega_f\) is concave and \(\omega_f(0)=0\), then \(\mu_{\delta^X}\) can be
replaced by \(\omega_f(\delta^X)\). In particular, if \(f\) is \(1\)-Lipschitz
with respect to \(d_X\), then
\[
d_I\!\left(
\mathcal F_R,
\mathcal R_{\delta^X}^{S_n,d_X}
\right)
\le
\delta^X.
\]
If moreover \(\delta\le\tau\), then
\[
\delta^X\le \frac{\pi}{3}\delta,
\]
and hence, in the \(1\)-Lipschitz case,
\[
d_I\!\left(
\mathcal F_R,
\mathcal R_{\delta^X}^{S_n,d_X}
\right)
\le
\frac{\pi}{3}\delta.
\]
\end{cor}

\begin{proof}
Let \(x\in X\). By the Euclidean covering-radius assumption, there exists
\(s\in S_n\) such that
$\|x-s\|\le \delta.
$

Since \(S_n\subseteq X\) and \(\delta<2\tau\), \Cref{thm:metric_distortion_reach}
gives
\[
d_X(x,s)
\le
2\tau\arcsin\!\left(\frac{\|x-s\|}{2\tau}\right)
\le
2\tau\arcsin\!\left(\frac{\delta}{2\tau}\right)
=
\delta^X.
\]
Thus \(\delta^X\) is an intrinsic covering radius for \(S_n\). The first bound is
therefore exactly \Cref{thm:intrinsic_reeb_stability} applied with
\(\delta^X\) in place of \(\delta\). The concave-modulus and \(1\)-Lipschitz
variants follow from the corresponding intrinsic stability variants. Finally,
if \(\delta\le\tau\), then \(\delta^X\le(\pi/3)\delta\).
\end{proof}

\begin{thm}[Extrinsic PL-Reeb stability]\label{thm:extrinsic_reeb_stability}
Under the setting of \Cref{sec:delta_mapper_setting}, assume moreover that
\(X\) has positive reach \(\tau=\rch(X)>0\), that \(0<\delta<\tau\), and that
\(\delta\) is a Euclidean covering radius for \(S_n\), namely
\[
\forall x\in X,\ \exists s\in S_n\quad \|x-s\|\le \delta.
\]
Then
\[
d_I\!\left(
\mathcal F_R,
\mathcal R_{\delta}^{S_n,\|\cdot\|}
\right)
\le
\omega_f(\eta_\tau(\delta)).
\]
In particular, if \(f\) is \(1\)-Lipschitz with respect to \(d_X\) and
\(2\delta<\tau\), then
\[
d_I\!\left(
\mathcal F_R,
\mathcal R_{\delta}^{S_n,\|\cdot\|}
\right)
\le
\frac{2\pi}{3}\,\delta.
\]
\end{thm}

\begin{proof}
The argument follows the proof of \Cref{thm:intrinsic_reeb_stability}, but the
use of the ambient Euclidean graph requires some additional metric-comparison
steps.

Write
\[
\Gamma_\delta=\Gamma_\delta^{S_n,\|\cdot\|},
\qquad
\hat f_n=\hat f_n^{\|\cdot\|},
\qquad
\eta=\eta_\tau(\delta)=\psi_\tau(2\delta),
\qquad
\mu=\omega_f(\eta).
\]
If \([x,y]\) is an edge of \(\Gamma_\delta\), then
$\|x-y\|\le 2\delta.
$
Since \(\delta<\tau\), \Cref{thm:metric_distortion_reach} gives
$d_X(x,y)
\le
\eta.
$

For every edge \([x,y]\) of \(\Gamma_\delta\), choose a minimizing
\(d_X\)-geodesic \(\gamma_{xy}\) from \(x\) to \(y\), parametrized
proportionally to arc length. These choices define a continuous map
\[
\Phi\colon |\Gamma_\delta|\to X
\]
which is the identity on vertices and sends the point with affine coordinate
\(t\) on the edge \([x,y]\) to \(\gamma_{xy}(t)\). Since every edge satisfies
\(d_X(x,y)\le \eta=2(\eta/2)\), \eqref{eq:theta_edge_estimate} gives
\[
|f(\Phi(p))-\hat f_n(p)|
\le
\theta_f(\eta/2)
\]
for every \(p\in|\Gamma_\delta|\). The doubled-scale modulus estimate gives
\[
\theta_f(\eta/2)\le\omega_f(\eta)=\mu,
\]
and therefore
\[
|f(\Phi(p))-\hat f_n(p)|\le\mu
\qquad
\text{for every }p\in|\Gamma_\delta|.
\]
We highlight that this is where the asymmetry with \Cref{thm:intrinsic_reeb_stability} appears:
the \(\theta_f\)-term is evaluated at half the intrinsic edge scale \(\eta\), and
is already dominated by the modulus at scale \(\eta\). The same modulus term is
also needed below for the sample-to-space estimates, so the final extrinsic
constant is \(\omega_f(\eta_\tau(\delta))\), not a maximum of two separate terms as in \Cref{thm:intrinsic_reeb_stability}  with $\mu_\delta$.

We now construct the first morphism of the interleaving. See also \Cref{fig:stability} (left) for a visual representation. Let \(I\in\mathbf{Int}\),
and let \(A\) be a path-connected component of
\[
(\hat f_n)^{-1}(I)\subseteq|\Gamma_\delta|.
\]
For every \(p\in A\), we have \(\hat f_n(p)\in I\), and the estimate above gives
\[
f(\Phi(p))\in I^\mu.
\]
Hence
\[
\Phi(A)\subseteq f^{-1}(I^\mu).
\]
Moreover, \(\Phi(A)\) is path-connected because \(A\) is path-connected and
\(\Phi\) is continuous. Therefore \(\Phi(A)\) is contained in a unique
path-connected component of \(f^{-1}(I^\mu)\). Denote this component by
$\alpha_I(A).
$

This assignment is compatible with inclusions of intervals. Indeed, if
\(I\subseteq J\) and \(A'\) is the path-connected component of
\((\hat f_n)^{-1}(J)\) containing \(A\), then
\[
\Phi(A)\subseteq \Phi(A').
\]
Thus the component \(\alpha_I(A)\) maps to the component \(\alpha_J(A')\) under
the canonical map
\[
\pi_0(f^{-1}(I^\mu))\longrightarrow \pi_0(f^{-1}(J^\mu)).
\]
Hence the maps \(\alpha_I\) define a natural transformation
\[
\alpha\colon
\mathcal R_{\delta}^{S_n,\|\cdot\|}
\Rightarrow
\mathcal S_\mu\mathcal F_R.
\]

We now define the morphism in the opposite direction. Let \(I\in\mathbf{Int}\),
and let \(C\) be a path-connected component of \(f^{-1}(I)\). Set
\[
S(C)\coloneqq
\{s\in S_n\mid B_{\|\cdot\|}(s,\delta)\cap C\neq\varnothing\},
\]
where the ball is taken in the ambient Euclidean metric. The set \(S(C)\) is
non-empty by the Euclidean covering-radius assumption. If \(s\in S(C)\), choose
\(c_s\in C\) with
\[
\|s-c_s\|\le\delta.
\]
Since \(S_n\subseteq X\), \(c_s\in X\), and \(\delta<\tau\),
\Cref{thm:metric_distortion_reach} gives
\[
d_X(s,c_s)
\le
\psi_\tau(\delta)
\le
\psi_\tau(2\delta)
=
\eta.
\]
Therefore
\[
|f_n(s)-f(c_s)|
=
|f(s)-f(c_s)|
\le
\omega_f(\eta)
=
\mu.
\]
Because \(f(c_s)\in I\), it follows that
$f_n(s)\in I^\mu.
$

We claim that \(S(C)\) spans a connected subgraph of \(\Gamma_\delta\), and that
this subgraph is contained in \((\hat f_n)^{-1}(I^\mu)\). Since \(S_n\) is finite,
the family
\[
\{C\cap B_{\|\cdot\|}(s,\delta)\}_{s\in S(C)}
\]
is a finite closed cover of the connected space \(C\), and each member is
non-empty. By \Cref{rmk:finite_closed_cover_intersection_graph}, its intersection graph
is connected. Thus, for any two elements \(s,s'\in S(C)\), there is a finite
chain
$s=s_0,s_1,\ldots,s_q=s'
$
in \(S(C)\) such that
\[
C\cap B_{\|\cdot\|}(s_r,\delta)
\cap B_{\|\cdot\|}(s_{r+1},\delta)
\neq\varnothing
\]
for every \(r\). In particular,
\[
\|s_r-s_{r+1}\|\le 2\delta,
\]
so consecutive vertices are joined by edges of \(\Gamma_\delta\). Moreover, the
endpoint values of each such edge lie in \(I^\mu\). Since \(\hat f_n\) is affine
on the edge and \(I^\mu\) is an interval, the whole edge lies in
$(\hat f_n)^{-1}(I^\mu).
$

Therefore the subgraph spanned by \(S(C)\) is connected and lies in
\((\hat f_n)^{-1}(I^\mu)\). It is consequently contained in a unique
path-connected component of \((\hat f_n)^{-1}(I^\mu)\). Denote this component by
$\beta_I(C).
$

This assignment is compatible with inclusions. If \(I\subseteq J\), and if \(C'\)
is the path-connected component of \(f^{-1}(J)\) containing \(C\), then
$S(C)\subseteq S(C').
$
Hence the subgraph spanned by \(S(C)\) is contained in the subgraph spanned by
\(S(C')\), and the component \(\beta_I(C)\) maps to the component
\(\beta_J(C')\) under the canonical map
\[
\pi_0((\hat f_n)^{-1}(I^\mu))
\longrightarrow
\pi_0((\hat f_n)^{-1}(J^\mu)).
\]
Thus the maps \(\beta_I\) define a natural transformation
\[
\beta\colon
\mathcal F_R
\Rightarrow
\mathcal S_\mu\mathcal R_{\delta}^{S_n,\|\cdot\|}.
\]

It remains to check the two interleaving identities. First, let \(I\in
\mathbf{Int}\), and let \(A\) be a path-connected component of
\((\hat f_n)^{-1}(I)\). Let
\[
C_A\coloneqq \alpha_I(A)
\]
be the path-connected component of \(f^{-1}(I^\mu)\) containing \(\Phi(A)\). We
show that the canonical image of \(A\) in
$\pi_0((\hat f_n)^{-1}(I^{2\mu}))
$
coincides with
$\beta_{I^\mu}(\alpha_I(A)).
$

By definition, \(\beta_{I^\mu}(C_A)\) is the component of
\((\hat f_n)^{-1}(I^{2\mu})\) containing the subgraph spanned by \(S(C_A)\).
Thus it is enough to show that \(A\) meets the path-connected component of
\((\hat f_n)^{-1}(I^{2\mu})\) containing this subgraph. Indeed, \(A\) is
path-connected and
\[
A\subseteq(\hat f_n)^{-1}(I)\subseteq(\hat f_n)^{-1}(I^{2\mu}).
\]

Choose \(p\in A\). If \(p\) is a vertex, then \(\Phi(p)=p\in C_A\). Hence
\(p\in S(C_A)\), so \(p\) already belongs to the subgraph spanned by \(S(C_A)\),
and the desired conclusion follows. We may therefore suppose that \(p\) lies in
the interior of an edge \([x,y]\), and set
$q=\Phi(p).
$

Let \(v\in\{x,y\}\) be an endpoint for which the subsegment of the chosen
geodesic from \(v\) to \(q\) has length at most \(L/2\), where
\(L=d_X(x,y)\). Since \(L\le\eta\), we have
$d_X(v,q)\le \eta/2.
$
Moreover, by \eqref{eq:theta_endpoint_estimate},
\[
|\hat f_n(p)-f_n(v)|
\le
\theta_f(\eta/2)
\le
\omega_f(\eta)
=
\mu.
\]
Since \(\hat f_n(p)\in I\), it follows that \(f_n(v)\in I^\mu\). Therefore the
graph segment from \(p\) to \(v\) lies in
\[
(\hat f_n)^{-1}(I^\mu)
\subseteq
(\hat f_n)^{-1}(I^{2\mu}),
\]
because \(\hat f_n\) is affine on the edge.

Note that, unlike in the intrinsic case, we cannot conclude from this that
\(v\in S(C_A)\): the set \(S(C_A)\) is defined using ambient Euclidean
\(\delta\)-balls, whereas the estimate above only gives the intrinsic bound
\(d_X(v,q)\le \eta/2\), which need not imply \(\|v-q\|\le\delta\). 
Thus, we need to connect \(v\) to the subgraph spanned by \(S(C_A)\), staying inside
\((\hat f_n)^{-1}(I^{2\mu})\). See also \Cref{fig:stability} (right) for a visual representation.

Consider the \(d_X\)-geodesic subsegment from
\(v\) to \(q\), and set
\[
S(v,q)\coloneqq
\{s\in S_n\mid B_{\|\cdot\|}(s,\delta)
\text{ meets this subsegment}\}.
\]
The sets
\[
\text{subsegment}\cap B_{\|\cdot\|}(s,\delta),
\qquad s\in S(v,q),
\]
form a finite closed cover of the subsegment by non-empty sets. By \Cref{rmk:finite_closed_cover_intersection_graph}, their intersection graph is connected.
If two such sets intersect, then the corresponding sample points are at Euclidean
distance at most \(2\delta\), and therefore the corresponding vertices are joined
by an edge of \(\Gamma_\delta\). Thus \(S(v,q)\) spans a connected subgraph of
\(\Gamma_\delta\).

This subgraph contains \(v\), since \(v\in S_n\) and the subsegment contains
\(v\). It also contains a vertex of \(S(C_A)\): indeed, \(q\in C_A\), and the
Euclidean covering-radius assumption gives \(s\in S_n\) with
$\|s-q\|\le\delta.
$

Since \(q\) belongs to the subsegment, this \(s\) belongs to \(S(v,q)\); since
also \(q\in C_A\), it belongs to \(S(C_A)\).

We now check that the connecting subgraph spanned by \(S(v,q)\) lies in
\((\hat f_n)^{-1}(I^{2\mu})\). Let \(s\in S(v,q)\). Choose a point \(r\) on the
subsegment from \(v\) to \(q\) such that
$\|s-r\|\le\delta.
$
By \Cref{thm:metric_distortion_reach},
\[
d_X(s,r)\le \psi_\tau(\delta).
\]
Also, since \(r\) lies on the \(d_X\)-geodesic subsegment from \(v\) to \(q\),
\[
d_X(r,q)\le d_X(v,q)\le \eta/2.
\]
Moreover,
\[
\psi_\tau(\delta)
=
2\tau\arcsin\!\Bigl(\frac{\delta}{2\tau}\Bigr)
\le
\tau\arcsin\!\Bigl(\frac{\delta}{\tau}\Bigr)
=
\eta/2,
\]
where the inequality is \(2\arcsin(u/2)\le \arcsin(u)\) for \(u\in[0,1]\).
Hence
\[
d_X(s,q)
\le
d_X(s,r)+d_X(r,q)
\le
\psi_\tau(\delta)+\eta/2
\le
\eta.
\]
Since \(q\in C_A\subseteq f^{-1}(I^\mu)\), we have \(f(q)\in I^\mu\), and
therefore
\[
|f_n(s)-f(q)|
=
|f(s)-f(q)|
\le
\omega_f(\eta)
=
\mu.
\]
Thus
$f_n(s)\in I^{2\mu}.
$

So all vertices of the subgraph spanned by \(S(v,q)\) have \(f_n\)-values in
\(I^{2\mu}\). Since \(\hat f_n\) is affine on edges and \(I^{2\mu}\) is an
interval, the whole subgraph lies in
$(\hat f_n)^{-1}(I^{2\mu}).
$

We have therefore connected \(p\) to the subgraph spanned by \(S(C_A)\) inside
\((\hat f_n)^{-1}(I^{2\mu})\). Hence
$\beta_{I^\mu}\circ\alpha_I
$
is the canonical structure map
\[
\mathcal R_{\delta}^{S_n,\|\cdot\|}(I)
\longrightarrow
\mathcal R_{\delta}^{S_n,\|\cdot\|}(I^{2\mu}).
\]

Conversely, let \(I\in\mathbf{Int}\), and let \(C\) be a path-connected component
of \(f^{-1}(I)\). Let
\[
A_C\coloneqq \beta_I(C)
\]
be the path-connected component of \((\hat f_n)^{-1}(I^\mu)\) containing the
subgraph spanned by \(S(C)\). We show that the canonical image of \(C\) in
$\pi_0(f^{-1}(I^{2\mu}))
$
coincides with
$\alpha_{I^\mu}(\beta_I(C)).
$

By definition, \(\alpha_{I^\mu}(A_C)\) is the path-connected component of
\(f^{-1}(I^{2\mu})\) containing \(\Phi(A_C)\). It is therefore enough to show
that \(C\) and \(\Phi(A_C)\) lie in the same path-connected component of
\(f^{-1}(I^{2\mu})\).

First, every point of \(C\) is connected inside \(f^{-1}(I^\mu)\) to a vertex in
\(S(C)\). Indeed, let \(c\in C\). By the Euclidean covering-radius assumption,
choose \(s\in S_n\) with
$\|c-s\|\le\delta.
$

Then \(s\in S(C)\). By \Cref{thm:metric_distortion_reach},
\[
d_X(c,s)\le\psi_\tau(\delta)\le\eta.
\]
Let \(\gamma\) be a minimizing \(d_X\)-geodesic from \(c\) to \(s\). For every
point \(r\) on this geodesic,
\[
d_X(r,c)\le d_X(c,s)\le\eta.
\]
Since \(f(c)\in I\), we have
\[
|f(r)-f(c)|\le\omega_f(\eta)=\mu,
\]
and hence \(f(r)\in I^\mu\). Thus \(c\) is connected to \(s\) inside
\(f^{-1}(I^\mu)\).

It remains to compare the graph component \(A_C\) with its image under \(\Phi\).
Since
\[
A_C\subseteq(\hat f_n)^{-1}(I^\mu),
\]
the estimate
\[
|f(\Phi(p))-\hat f_n(p)|\le\mu
\]
implies
\[
\Phi(A_C)\subseteq f^{-1}(I^{2\mu}).
\]
Moreover, \(A_C\) is path-connected, hence \(\Phi(A_C)\) is path-connected. The
image \(\Phi(A_C)\) contains every vertex of \(S(C)\), because \(\Phi\) is the
identity on vertices and \(A_C\) contains the subgraph spanned by \(S(C)\).
Since every point of \(C\) is connected inside
\[
f^{-1}(I^\mu)\subseteq f^{-1}(I^{2\mu})
\]
to a vertex of \(S(C)\), and since all vertices of \(S(C)\) lie in the
path-connected set \(\Phi(A_C)\subseteq f^{-1}(I^{2\mu})\), the sets \(C\) and
\(\Phi(A_C)\) lie in the same path-connected component of
\(f^{-1}(I^{2\mu})\). Therefore
$\alpha_{I^\mu}\circ\beta_I
$
is the canonical structure map
\[
\mathcal F_R(I)
\longrightarrow
\mathcal F_R(I^{2\mu}).
\]

Thus \(\alpha\) and \(\beta\) define a \(\mu\)-interleaving, and hence
\[
d_I\!\left(
\mathcal F_R,
\mathcal R_{\delta}^{S_n,\|\cdot\|}
\right)
\le
\mu
=
\omega_f(\eta_\tau(\delta)).
\]
Finally, if \(f\) is \(1\)-Lipschitz with respect to \(d_X\), then we may take
\(\omega_f(r)=r\), so \(\mu=\eta_\tau(\delta)\). If \(2\delta<\tau\), then
\[
\eta_\tau(\delta)=\psi_\tau(2\delta)\le \frac{2\pi}{3}\delta
\]
by \Cref{cor:metric_equivalence_no_reach}. This gives
\[
d_I\!\left(
\mathcal F_R,
\mathcal R_{\delta}^{S_n,\|\cdot\|}
\right)
\le
\frac{2\pi}{3}\delta.
\]
\end{proof}

\begin{figure}
    \centering
    \begin{subfigure}[b]{0.48\textwidth}
        \centering
        \includegraphics[width=\textwidth]{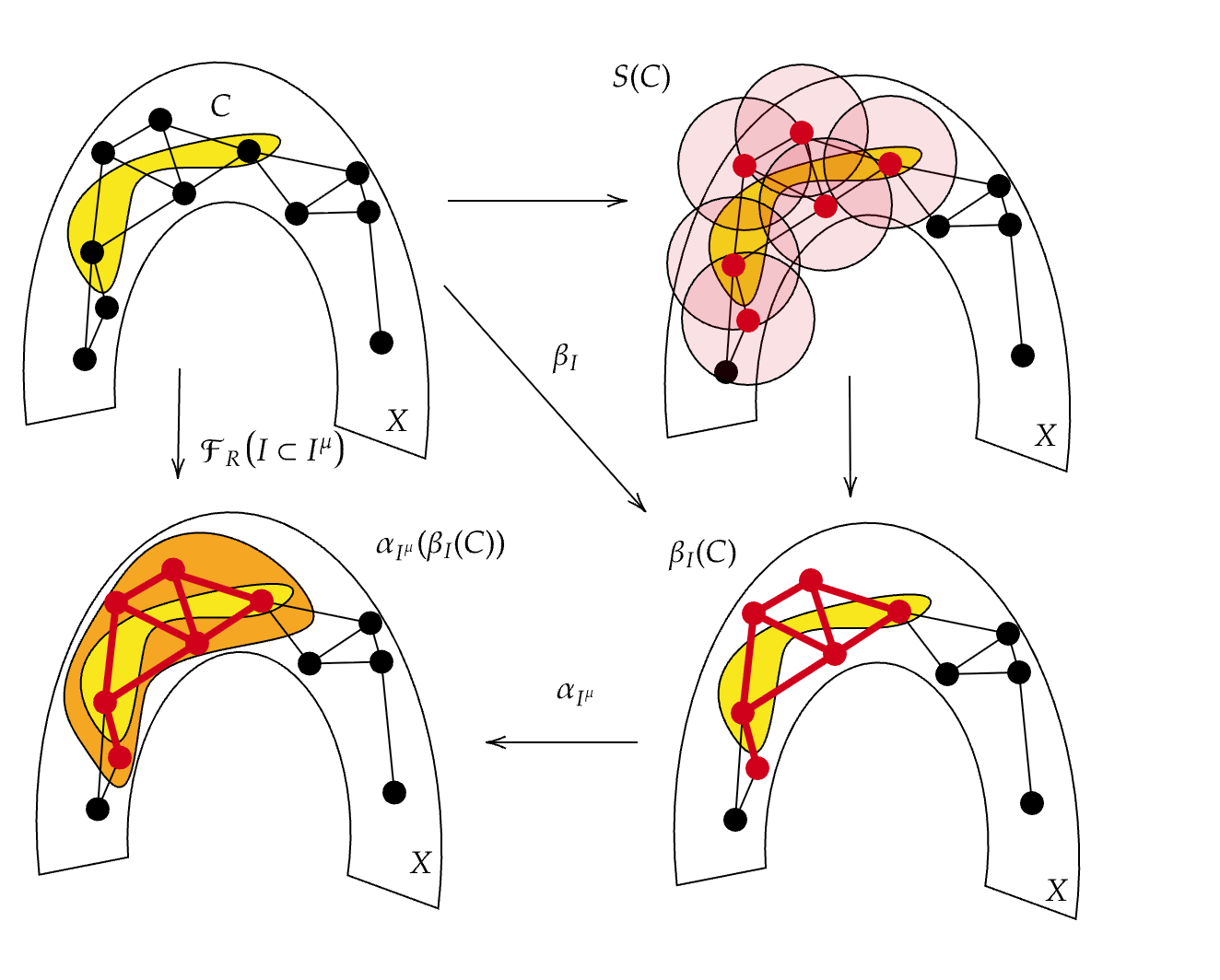}
        \caption{Visual representation of the interleaving maps. The arrows describe the association  $C\in \pi_0(f^{-1}(I))\mapsto \beta_I(C)\mapsto \alpha_{I^\mu}(\beta_I(C))$.}
    \end{subfigure}
    \hfill
    \begin{subfigure}[b]{0.48\textwidth}
        \centering
        \includegraphics[width=\textwidth]{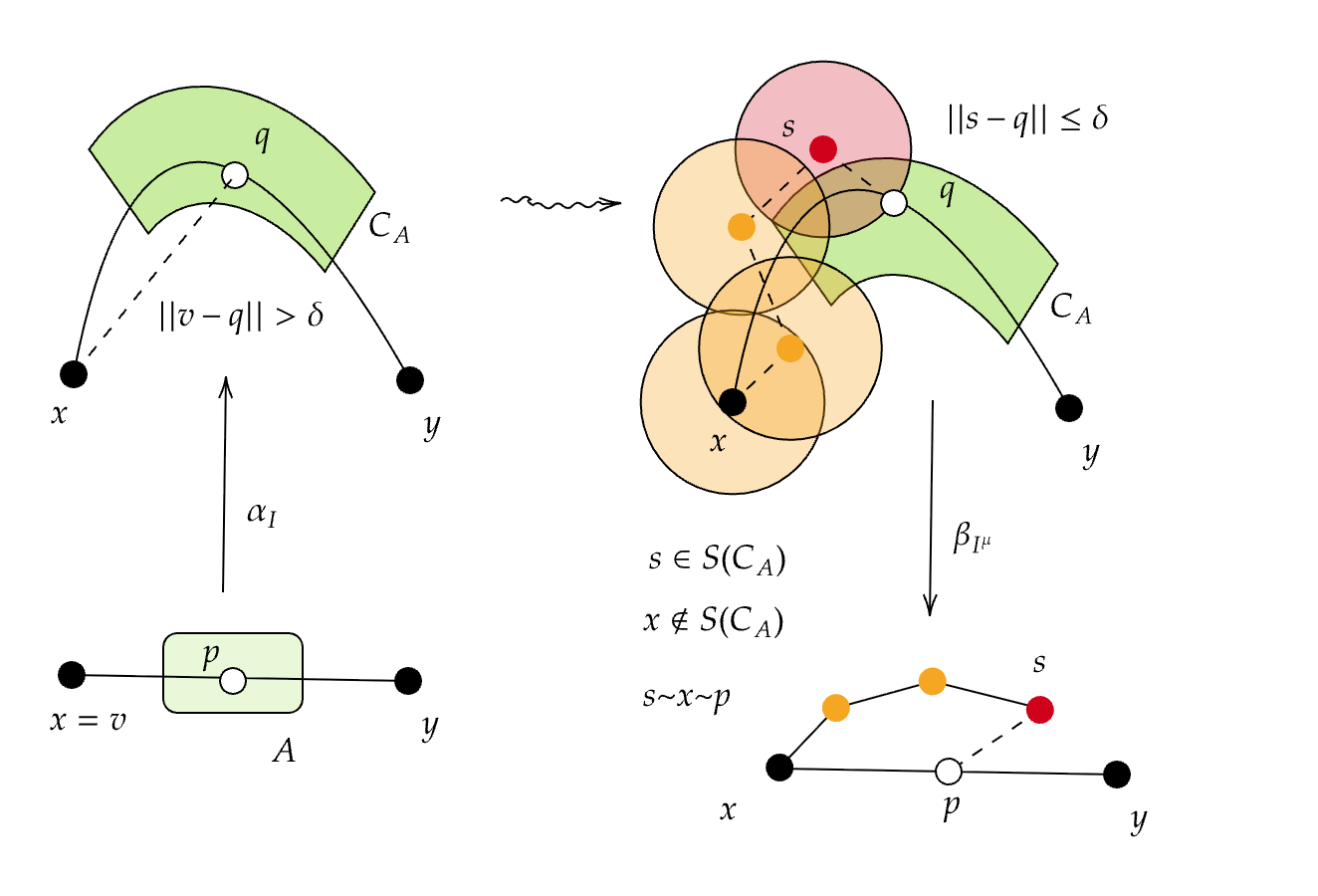}
        \caption{Visual representation of a technical step in the proof of \Cref{thm:extrinsic_reeb_stability}: verifying that there is a path in $\mathcal R_{\delta}^{S_n,\|\cdot\|}$ from  $p\in A, A\in \pi_0 (\hat f_n^{-1}(I))$, to  a vertex in $S(C_A)$.}
    \end{subfigure}
    \caption{Schematic illustrations for the proofs of the stability theorems.}
    \label{fig:stability}
\end{figure}

\subsection{Mapper coarsenings and Reeb-to-Mapper multivalued maps}
\label{sec:mapper_coarsening}

The PL-Reeb estimator and its Mapper coarsenings play complementary roles. The
PL-Reeb estimator carries the object-level stability estimate, while the Mapper
coarsening provides a lower-resolution object that can be substantially smaller
and easier to visualize. The loss incurred by this coarsening is explicit: it is
controlled by the resolution of the chosen cover with respect to the PL-Reeb
cosheaf.

Let \(\mathcal U\) be a finite nice open cover of \(\mathbb R\). We apply the
Mapper transformation associated with \(\mathcal U\) to the PL-Reeb estimators and
set
\[
\widetilde{\mathcal M}^{S_n,\rho}_{\delta,\mathcal U}
\coloneqq
\mathcal M_{\mathcal U}\bigl(\mathcal R_{\delta}^{S_n,\rho}\bigr),
\qquad
\rho\in\{d_X,\|\cdot\|\}.
\]
Equivalently, for \(I\in\mathbf{Int}\),
\[
\widetilde{\mathcal M}^{S_n,\rho}_{\delta,\mathcal U}(I)
=
\pi_0\!\left(
(\hat f_n^\rho)^{-1}(\mathcal I_{\mathcal U}(I))
\right).
\]
The Mapper object is therefore a cover-level coarsening of the PL-Reeb estimator.
The cover is arbitrary in the theory; its contribution to the error bound is
only through its resolution with respect to the PL-Reeb cosheaf.

\begin{prop}[Mapper coarsening error]\label{prop:mapper_coarsening_error}
Let \(\rho\in\{d_X,\|\cdot\|\}\), and let \(\mathcal U\) be a finite nice open
cover of \(\mathbb R\). Then
\[
d_I\!\left(
\mathcal R_{\delta}^{S_n,\rho},
\widetilde{\mathcal M}^{S_n,\rho}_{\delta,\mathcal U}
\right)
\le
\res_{\mathcal R_{\delta}^{S_n,\rho}}(\mathcal U).
\]
In particular, if
\[
\res_{\mathcal R_{\delta}^{S_n,\rho}}(\mathcal U)\le \chi,
\]
then the coarsening error is at most \(\chi\).
\end{prop}

\begin{proof}
This is exactly the stability estimate for the Mapper transformation recalled in
\Cref{brobro}, applied to the cosheaf
\(\mathcal R_{\delta}^{S_n,\rho}\).
\end{proof}

Consequently, any stability estimate for a PL-Reeb estimator immediately gives a
stability estimate for its Mapper coarsening by the triangle inequality. For
example, under the assumptions of \Cref{thm:intrinsic_reeb_stability},
\[
d_I\!\left(
\mathcal F_R,
\widetilde{\mathcal M}^{S_n,d_X}_{\delta,\mathcal U}
\right)
\le
\mu_\delta+
\res_{\mathcal R_{\delta}^{S_n,d_X}}(\mathcal U).
\]
If \(\omega_f\) is concave and \(\omega_f(0)=0\), then \(\mu_\delta\) can be
replaced by \(\omega_f(\delta)\); in particular, if \(f\) is \(1\)-Lipschitz with
respect to \(d_X\), then the right-hand side is bounded by
\[
\delta+
\res_{\mathcal R_{\delta}^{S_n,d_X}}(\mathcal U).
\]
The analogous intrinsic-from-Euclidean and extrinsic bounds are obtained in the
same way, by adding the corresponding Mapper-resolution term to the PL-Reeb
bounds of \Cref{cor:intrinsic_pl_stability_from_euclidean_radius,thm:extrinsic_reeb_stability}.

This gives a precise way to use Mapper as a visualization layer for potentially
complex PL-Reeb estimators. The estimator itself retains the sharper
object-level approximation guarantee, while the cover \(\mathcal U\) determines
the scale at which this estimator is displayed. The additional error controlled by
$\res_{\mathcal R_{\delta}^{S_n,\rho}}(\mathcal U),
$
so the dimensionality reduction and the loss in the guarantee are controlled by
the same explicit quantity.

The controlled link between the two objects can be expressed by a natural
multivalued map from the coarsened Mapper object back to the PL-Reeb estimator.
For every interval \(I\), the inclusion
$I\subseteq\mathcal I_{\mathcal U}(I)
$
induces the structure map
\[
\mathcal R_{\delta}^{S_n,\rho}(I)
\longrightarrow
\mathcal R_{\delta}^{S_n,\rho}(\mathcal I_{\mathcal U}(I))
=
\widetilde{\mathcal M}^{S_n,\rho}_{\delta,\mathcal U}(I).
\]
For a coarsened component
\(A\in\widetilde{\mathcal M}^{S_n,\rho}_{\delta,\mathcal U}(I)\), define
\[
\Xi_I(A)
\coloneqq
\left\{
B\in\mathcal R_{\delta}^{S_n,\rho}(I)
\mid
B\mapsto A
\right\}.
\]
In geometric terms, \(A\) is a component of
\((\hat f_n^\rho)^{-1}(\mathcal I_{\mathcal U}(I))\), and \(\Xi_I(A)\) is the
collection of components of \((\hat f_n^\rho)^{-1}(I)\) contained in it. Hence a
node or component of the coarsened Mapper object can be visualized together with
the PL-Reeb components that it represents.

\section{Confidence regions based on the interleaving distance}\label{conf}

Throughout this section, we retain the setting and notation of
\Cref{sec:delta_mapper_setting}; in particular,
$R\coloneqq\mathfrak R(X,f).
$

The stability results of
\Cref{thm:intrinsic_reeb_stability,thm:extrinsic_reeb_stability} turn covering
radius bounds for the sample into confidence regions for the target Reeb cosheaf.
We present the general mechanism once, and then specialize it below to the
relevant probabilistic regimes.

Let \(\widehat{\mathcal F}\) be a random cosheaf-valued estimator of
\(\mathcal F_R\). Suppose that \(\delta\ge 0\) is a deterministic radius and
that \(E\) is an event such that, on \(E\),
\[
d_I\!\left(\mathcal F_R,\widehat{\mathcal F}\right)
\le \delta.
\]
If \(\mathbb P(E)\ge 1-\alpha\), then the random interleaving ball
\[
B_{d_I}(\widehat{\mathcal F},\delta)
\coloneqq
\left\{
\mathcal G
\middle|
 d_I\!\left(\mathcal G,\widehat{\mathcal F}\right)
 \le \delta
\right\}
\]
is a \((1-\alpha)\)-confidence region for \(\mathcal F_R\). Indeed, on the
event \(E\), the target cosheaf \(\mathcal F_R\) belongs to this ball.

The same principle applies after Mapper coarsening. If \(\mathcal U\) is a
finite nice open cover and
\[
\widehat{\mathcal M}
\coloneqq
\mathcal M_{\mathcal U}(\widehat{\mathcal F}),
\]
then \(\Cref{prop:mapper_coarsening_error}\) gives
\[
d_I\!\left(\widehat{\mathcal F},\widehat{\mathcal M}\right)
\le
\res_{\widehat{\mathcal F}}(\mathcal U).
\]
Therefore, on the same event \(E\),
\[
d_I\!\left(\mathcal F_R,\widehat{\mathcal M}\right)
\le
\delta+\res_{\widehat{\mathcal F}}(\mathcal U).
\]
Thus
$B_{d_I}\!\left(
\widehat{\mathcal M},
\delta+\res_{\widehat{\mathcal F}}(\mathcal U)
\right)$
is the corresponding Mapper-level confidence region. If one uses a deterministic
upper bound \(\chi\ge\res_{\widehat{\mathcal F}}(\mathcal U)\), the radius can be
replaced by \(\delta+\chi\). Below we state only the PL-Reeb confidence radii; the corresponding
Mapper-level radii are obtained by this additive rule.

Let \(P\) be a probability measure supported on \(X\), and assume that
\(S_n=\{X_1,\ldots,X_n\}\) is an i.i.d. sample from \(P\). We write
\[
d_H^{X}(S_n,X)\coloneqq \sup_{x\in X}\inf_{y\in S_n} d_X(x,y)
\]
for the one-sided Hausdorff covering radius in the intrinsic metric. When
\(X\subset\mathbb R^m\), we also write
\[
d_H^{E}(S_n,X)\coloneqq \sup_{x\in X}\inf_{y\in S_n} \|x-y\|
\]
for the corresponding Euclidean covering radius. Since \(S_n\subseteq X\), these
are the relevant one-sided Hausdorff distances from \(X\) to the sample.

Since the intrinsic and extrinsic stability bounds above are driven by covering
radii of \(S_n\), the remaining task is to control these radii probabilistically.
We do this in two standard ways: first through an \((a,b)\)-standard assumption,
which gives an explicit finite-sample tail bound for the relevant Hausdorff
covering radius, and then through the subsampling method of~\cite{fasy}.

\subsection{Using \texorpdfstring{\((a,b)\)}{(a,b)}-standard assumptions}

We first consider the classical \((a,b)\)-standard assumption. The following
Hausdorff tail bound is the one proved in~\cite{chazal15a}; it applies to the
metric in which the \((a,b)\)-standard condition is imposed.

\begin{thm}[Theorem~2 in~\cite{chazal15a}]\label{thm:chazal_hausdorff}
Let \((\mathbb M,\rho)\) be a metric space, and let \(P\) be a probability
measure with compact support \(X\subset\mathbb M\). Suppose that \(P\)
satisfies the \((a,b)\)-standard assumption in the metric \(\rho\), namely
\[
P(B_\rho(x,r))\ge \min\{ar^b,1\}
\qquad
\text{for all }x\in X\text{ and all }r>0.
\]
Let \(S_n=\{X_1,\ldots,X_n\}\) be an i.i.d. sample from \(P\), and set
\[
d_H^\rho(S_n,X)\coloneqq \sup_{x\in X}\inf_{y\in S_n} \rho(x,y).
\]
Then, for every \(\delta>0\),
\[
\mathbb P\!\left(d_H^\rho(S_n,X)>\delta\right)
\le
\min\!\left\{
1,\,
\frac{4^b}{a\delta^b}\exp\!\left(-a\left(\frac{\delta}{2}\right)^b n\right)
\right\}
\eqqcolon \alpha_\delta.
\]
\end{thm}

\begin{cor}[Intrinsic confidence region under \((a,b)\)-standard assumptions]\label{cor:ab_intrinsic}
Assume that the hypotheses of \Cref{thm:chazal_hausdorff} hold with
\(\rho=d_X\). Then, for every \(\delta>0\),
\[
\mathbb P\!\left(
 d_I\!\left(\mathcal F_R,\mathcal R_{\delta}^{S_n,d_X}\right)
 \le
 \mu_\delta
\right)
\ge 1-\alpha_\delta.
\]
If \(\omega_f\) is concave and \(\omega_f(0)=0\), then \(\mu_\delta\) can be
replaced by \(\omega_f(\delta)\); in particular, if \(f\) is \(1\)-Lipschitz with
respect to \(d_X\), then the confidence radius can be taken to be \(\delta\).
\end{cor}

\begin{cor}[Extrinsic confidence region under \((a,b)\)-standard assumptions]\label{cor:ab_extrinsic}
Assume that the hypotheses of \Cref{thm:chazal_hausdorff} hold with
\(\rho=\|\cdot\|\), and suppose moreover that \(X\) has positive reach
\(\tau=\rch(X)>0\). Then, for every \(\delta\in(0,\tau)\),
\[
\mathbb P\!\left(
 d_I\!\left(\mathcal F_R,\mathcal R_{\delta}^{S_n,\|\cdot\|}\right)
 \le
 \omega_f(\eta_\tau(\delta))
\right)
\ge 1-\alpha_\delta.
\]
If \(f\) is \(1\)-Lipschitz and \(2\delta<\tau\), the confidence radius
\(\omega_f(\eta_\tau(\delta))\) can be replaced by \((2\pi/3)\delta\).
\end{cor}

In practice, once the confidence level \(\alpha\) is fixed, one may choose
\(\delta_\alpha\) as the smallest solution of \(\alpha_\delta\le \alpha\) in the
relevant metric, and then build the intrinsic or extrinsic PL-Reeb estimator at
that scale.

\subsection{Using subsampling \`a la~\cite{fasy}}\label{sub}

We now turn to the subsampling strategy of~\cite{fasy}. The relevant Hausdorff
distance in that work is the ambient Euclidean Hausdorff distance. Accordingly,
throughout this subsection we assume the hypotheses of Theorem~3 in~\cite{fasy};
in particular, \(X=M\) is a compact manifold without boundary, embedded in
\(\mathbb R^D\), with positive reach, and the sampling distribution is supported
exactly on \(M\) in the noiseless setting.

Let \(b_n\) be such that
\[
b_n=o\!\left(\frac{n}{\log n}\right)
\qquad\text{and}\qquad
b_n\to +\infty.
\]
Draw all \(N=\binom{n}{b_n}\) subsamples of \(S_n\) of cardinality \(b_n\), and
denote them by \(S_{b_n,n}^1,\ldots,S_{b_n,n}^N\).  Define
\[
L_{b_n}(\delta)\coloneqq
\frac{1}{N}\sum_{j=1}^N
\mathbbm{1}_{\{d_H^E(S_{b_n,n}^j,S_n)>\delta\}},
\]
and set
\begin{equation}\label{quant_sub}
\delta_{b_n,\alpha}^{\rm raw}\coloneqq 2\,L_{b_n}^{-1}(\alpha).
\end{equation}
By Theorem~3 in~\cite{fasy}, for all sufficiently large \(n\),
\[
\mathbb P\!\left(d_H^E(S_n,X)>\delta_{b_n,\alpha}^{\rm raw}\right)
\le
\alpha+O\!\left(\frac{b_n}{n}\right)^{1/4}.
\]
This gives the following PL-Reeb confidence regions.

\begin{cor}[Extrinsic confidence region via subsampling]\label{cor:sub_extrinsic}
Assume the standing hypotheses of this subsection. Then, for all sufficiently
large \(n\) such that \(\delta_{b_n,\alpha}^{\rm raw}<\rch(X)\) almost surely,
\[
\mathbb P\!\left(
 d_I\!\left(\mathcal F_R,
 \mathcal R_{\delta_{b_n,\alpha}^{\rm raw}}^{S_n,\|\cdot\|}\right)
 \le
 \omega_f(\eta_{\rch(X)}(\delta_{b_n,\alpha}^{\rm raw}))
\right)
\ge
1-\alpha-O\!\left(\frac{b_n}{n}\right)^{1/4}.
\]
If \(f\) is \(1\)-Lipschitz and \(2\delta_{b_n,\alpha}^{\rm raw}<\rch(X)\) almost surely, the
confidence radius may be replaced by \((2\pi/3)\delta_{b_n,\alpha}^{\rm raw}\).
\end{cor}

The construction of \(\delta_{b_n,\alpha}^{\rm raw}\) is extrinsic. We can also
use it to obtain an intrinsic covering radius by applying the reach distortion
estimate. If \(d_H^E(S_n,M)\le \delta_{b_n,\alpha}^{\rm raw}\) and
\(\delta_{b_n,\alpha}^{\rm raw}<2\rch(M)\), then
\[
d_H^X(S_n,M)
\le
\psi_{\rch(M)}(\delta_{b_n,\alpha}^{\rm raw})
=
2\rch(M)\arcsin\!\left(\frac{\delta_{b_n,\alpha}^{\rm raw}}{2\rch(M)}\right).
\]
Set
\[
\delta_{b_n,\alpha}^{X,{\rm raw}}
\coloneqq
\psi_{\rch(M)}(\delta_{b_n,\alpha}^{\rm raw}).
\]
If one wants a reach-free constant at smaller scales, then
\(\delta_{b_n,\alpha}^{X,{\rm raw}}\le (\pi/3)\delta_{b_n,\alpha}^{\rm raw}\)
whenever \(\delta_{b_n,\alpha}^{\rm raw}\le \rch(M)\), by
\Cref{cor:metric_equivalence_no_reach}.

\begin{cor}[Intrinsic confidence region via subsampling]\label{cor:sub_intrinsic}
Assume the standing hypotheses of this subsection. Then, for all sufficiently
large \(n\) such that \(\delta_{b_n,\alpha}^{\rm raw}<2\rch(X)\) almost surely,
\[
\mathbb P\!\left(
 d_I\!\left(\mathcal F_R,
 \mathcal R_{\delta_{b_n,\alpha}^{X,{\rm raw}}}^{S_n,d_X}\right)
 \le
 \mu_{\delta_{b_n,\alpha}^{X,{\rm raw}}}
\right)
\ge
1-\alpha-O\!\left(\frac{b_n}{n}\right)^{1/4}.
\]
If \(f\) is \(1\)-Lipschitz and the concavity simplification applies, one may use
\(\delta_{b_n,\alpha}^{X,{\rm raw}}\) as the confidence radius.
\end{cor}

Equivalently, the raw subsampling construction gives the formal confidence balls
\[
\left\{
\mathcal G
\middle|
 d_I\!\left(\mathcal G,
 \mathcal R_{\delta_{b_n,\alpha}^{\rm raw}}^{S_n,\|\cdot\|}\right)
 \le
 \omega_f(\eta_{\rch(X)}(\delta_{b_n,\alpha}^{\rm raw}))
\right\}, \qquad
\left\{
\mathcal G
\middle|
 d_I\!\left(\mathcal G,
 \mathcal R_{\delta_{b_n,\alpha}^{X,{\rm raw}}}^{S_n,d_X}\right)
 \le
 \mu_{\delta_{b_n,\alpha}^{X,{\rm raw}}}
\right\},
\]
with confidence at least
\(1-\alpha-O((b_n/n)^{1/4})\), under the assumptions above.

The raw subsampling radius is the theorem-supported choice, but it is often very
conservative in finite samples. This conservativeness is structural. The
empirical distribution \(L_{b_n}\) is built from the quantities
$d_H^E(S_{b_n,n}^j,S_n),
$
so each subsample of size \(b_n\) is asked to cover the full sample \(S_n\). The
resulting quantile \(L_{b_n}^{-1}(\alpha)\) is therefore governed by a covering
radius at the subsample scale \(b_n\), rather than by the covering radius of the
full sample \(S_n\). The factor \(2\) in
$\delta_{b_n,\alpha}^{\rm raw}=2\,L_{b_n}^{-1}(\alpha)
$
is part of the subsampling argument of~\cite{fasy} and gives the asymptotic
confidence statement above; it also contributes to the finite-sample
conservativeness of the radius. In the notation used here, the bias comes from
using the subsample-to-sample covering scale \(d_H^E(S_{b_n,n}^j,S_n)\) to
control the sample-to-support covering scale \(d_H^E(S_n,X)\).

This motivates a rate correction by comparing two covering problems that occur
at different sample sizes.  The empirical quantile
\(L_{b_n}^{-1}(\alpha)\) estimates the radius needed for a set of
\(b_n\) sample points to cover the observed cloud \(S_n\).  What is needed for
constructing the estimator, however, is a radius at the scale of the full sample:
\(S_n\) has to cover the underlying support \(X\).  For samples on a
\(d\)-dimensional support, random covering radii are expected to scale as
$\left(\frac{\log m}{m}\right)^{1/d},
$
up to distribution- and geometry-dependent constants; see, for example,
\cite{janson1986randomCoverings,penrose2021randomEuclideanCoverage}.  Thus the
subsample-to-sample radius at sample size \(b_n\) can be converted into a
sample-to-support radius at sample size \(n\) by replacing the covering-rate
factor at \(b_n\) with the corresponding factor at \(n\).  This is the heuristic
behind the corrected radii used below.

Motivated by this comparison, our experiments also consider two corrected
Euclidean radii. The first uses the full logarithmic covering-rate correction:
\[
\delta_{b_n,\alpha}^{\rm log}
\coloneqq
\delta_{b_n,\alpha}^{\rm raw}
\left(
\frac{\log n/n}{\log b_n/b_n}
\right)^{1/d}.
\]
The second keeps only the power-law part of the same correction:
\[
\delta_{b_n,\alpha}^{\rm pow}
\coloneqq
\delta_{b_n,\alpha}^{\rm raw}
\left(
\frac{b_n}{n}
\right)^{1/d}.
\]
Here \(d\) is the intrinsic dimension used for the rate correction. The
logarithmic correction follows the usual covering-radius scale, while the
power-law correction ignores the logarithmic factor and gives a milder
calibration.

If a reach bound is available, the corresponding intrinsic corrected radii are
defined by the same reach conversion used above:
\[
\delta_{b_n,\alpha}^{X,{\rm log}}
\coloneqq
\psi_{\rch(X)}(\delta_{b_n,\alpha}^{\rm log}),
\qquad
\delta_{b_n,\alpha}^{X,{\rm pow}}
\coloneqq
\psi_{\rch(X)}(\delta_{b_n,\alpha}^{\rm pow}),
\]
whenever the arguments lie in the range where the reach comparison applies. If
an estimated intrinsic distance matrix is used instead, the same two rate
corrections can be applied directly to the intrinsic subsampling quantile
computed from that matrix.

\begin{oss}
The corrected radii are empirical calibrations of the raw subsampling radius,
not theorem-level certificates in this paper.  They are validated numerically in
\Cref{subsec:subsampling_validation_appendix}, and a theoretical analysis of the
correction is left for future work.
\end{oss}

\section{Numerical Experiments}\label{expes}

The experiments illustrate the workflow developed in \Cref{conf}. The
statistically controlled object is the PL--Reeb estimator
\[
\mathcal R_{\delta}^{S_n,\rho}
=
\mathcal F_{(|\Gamma_{\delta}^{S_n,\rho}|,\hat f_n^\rho)},
\]
whereas geometric interpretation is carried out on its Mapper coarsening
\[
\widetilde{\mathcal M}_{\delta,\mathcal U}^{S_n,\rho}
=
\mathcal M_{\mathcal U}(\mathcal R_{\delta}^{S_n,\rho}).
\]
Extended-persistence signatures and confidence radii are computed on the PL
object. Features selected at the PL level are then pushed forward to the Mapper
graph through the multivalued map \(\Xi_I\), which provides a coarser and more
readable representation of their support. Thus, unless explicitly stated
otherwise, the points in the extended diagrams below are features of the
PL--Reeb estimator, while the adjacent Mapper drawings display their coarsened
supports.

All experiments use the height coordinate as filter. The cover \(\mathcal U\)
is a regular interval cover of the observed filter range with overlap fraction
\(0.45\). If \(\chi\) denotes the interval length, adjacent interval centers
are spaced by \(0.55\chi\). Since every active interval has diameter at most
\(\chi\), \Cref{prop:mapper_coarsening_error} gives a Mapper confidence radius
obtained by adding at most \(\chi\) to the PL confidence radius.
Implementation details, runtime benchmarks, validation of the rate-corrected
subsampling radii, and visualization conventions are collected in
\Cref{sec:implementation_visualization_notes}.

As mentioned in the Introduction, the experiments also highlight a practical
difference from the confidence framework of~\cite{JMLR:v19:17-291}. That work
likewise derives confidence regions by combining probabilistic control of the
sampling error with a deterministic approximation bound. At the sample sizes
considered in its experiments, however, the resulting regions are reported to
be too conservative for interpretation, and the numerical analysis is instead
based on a bottleneck bootstrap whose validity is left open. Instead, the sharper
PL--Reeb bounds developed here, see \Cref{sec:jmlr_comparison} for a
quantitative comparison, remain informative at the same finite-sample
scales. Once a valid bound on the sampling radius is available, its propagation
to an interleaving confidence region and to persistence confidence bands is
fully theorem-supported. Thus, the only additional statistical issue in the
numerical pipeline is the data-driven calibration of the sampling radius. For
this step, we use practically validated procedures proposed by the authors of
the corresponding probabilistic frameworks, whose precise theoretical status
for the implementations adopted here is clarified in the following remark.

\begin{oss}[Confidence levels and significance]
Throughout the experiments, we set \(1-\alpha=0.85\). We call a persistence
feature \emph{significant} when it lies outside the band produced by the
corresponding calibration procedure.
For the \((a,b)\)-standard method used below, a formal coverage guarantee is
available when the local-mass constant is estimated on a sample independent of
the one used to construct the topological estimator, as in the split-sample
procedure of~\cite{fasy}. Splitting the data, however, substantially reduces
the effective sample size and weakens the resulting estimates. We therefore
use the full sample both to estimate the local-mass constant and to construct
the PL--Reeb estimator. This unsplit implementation is reported
in~\cite{fasy} to perform well in practice, but its coverage guarantee has not
been formally established.
For the rate-corrected subsampling method, the correction of the raw
subsampling radius is likewise calibrated empirically, as discussed in
\Cref{subsec:subsampling_validation_appendix}. In both cases, the deterministic
propagation from the resulting sampling radius to interleaving and persistence
bounds remains theorem-supported. Accordingly, for these numerical
implementations, the term \emph{significant} means that a feature lies outside
the corresponding calibrated band; it should not be read as asserting that the
precise finite-sample coverage of the implemented calibration has been proved
to equal \(1-\alpha\).
\end{oss}

\subsection{Data sets and deterministic calibration}

We use two data sets.  The first is a sample from a torus of major radius \(15\) and minor radius \(4\), whose reach is therefore \(4\).  The second is the ant point cloud from~\cite{10.1145/1531326.1531379}.  In the statistical experiments the torus sample has \(5000\) points, while the deterministic calibration uses \(3000\) torus points; the ant sample has \(6370\) points.  The random seed is \(7\).  Intrinsic computations use an Isomap-type approximation of the geodesic distance with \(10\) nearest neighbors \cite{isomap,scikit-learn}: the weighted graph distance approximates \(d_X\) by replacing paths in \(X\)
with chains of nearby sample points and summing the ambient lengths of the
resulting short segments.

\begin{figure}
    \centering
    \begin{subfigure}[b]{0.48\textwidth}
        \centering
        \includegraphics[width=\textwidth]{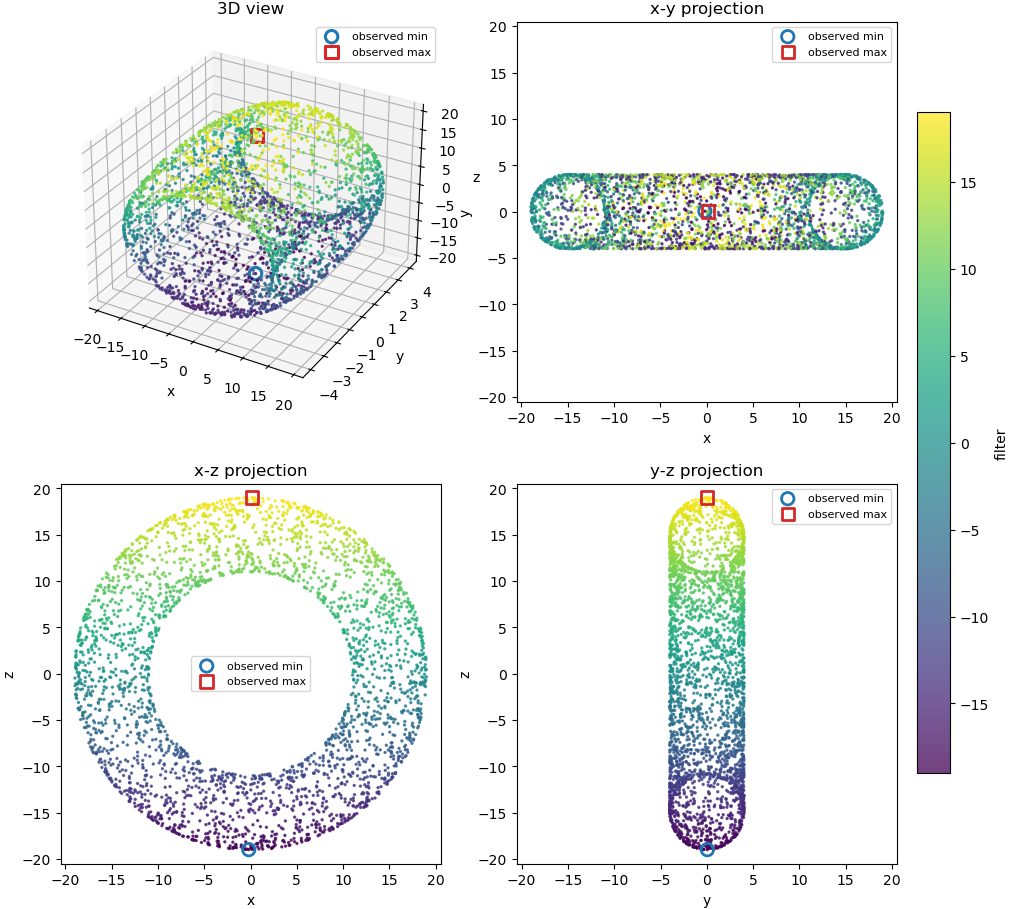}
        \caption{Torus.}
    \end{subfigure}
    \hfill
    \begin{subfigure}[b]{0.48\textwidth}
        \centering
        \includegraphics[width=\textwidth]{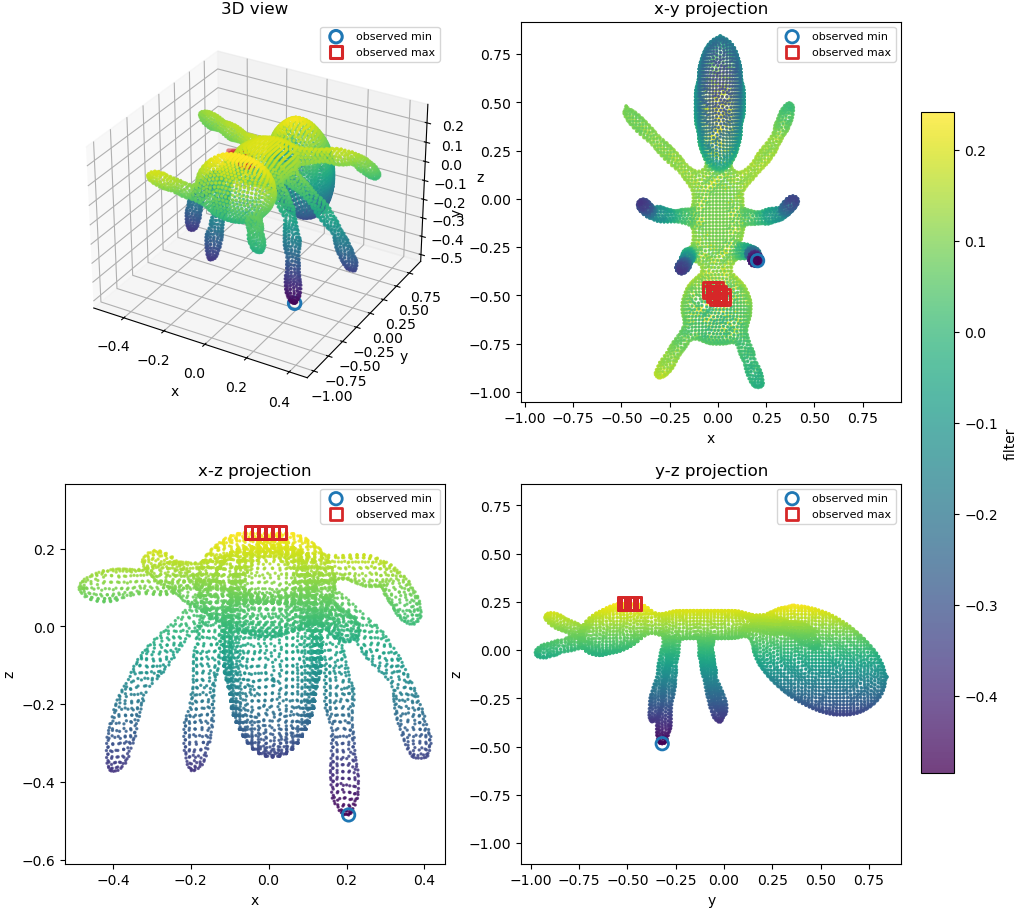}
        \caption{Ant.}
    \end{subfigure}
    \caption{Point clouds used in the experiments, coloured by the height filter.  The circled points are the observed global minimum and maximum of the filter.}
    \label{fig:exp_datasets}
\end{figure}

Before selecting scales statistically, we first inspect deterministic, hand-tuned scales.  These examples display the objects used later for inference and visualization.  For the torus we set \(\delta=1.2\) for both the Euclidean and intrinsic proximity graphs and take \(\chi=3.2\).  For the ant we use \(\delta_E=0.02902916\), \(\delta_X=0.04354374\), and \(\chi=0.10885935\).  The cover overlap fraction is again \(0.45\).  These hand-selected scales are not confidence radii; their role is to show the PL graph, the Mapper coarsening, and the PL-to-Mapper feature pushforward.

\begin{figure}
    \centering
    \begin{subfigure}[b]{0.30\textwidth}
        \centering
        \includegraphics[width=\textwidth]{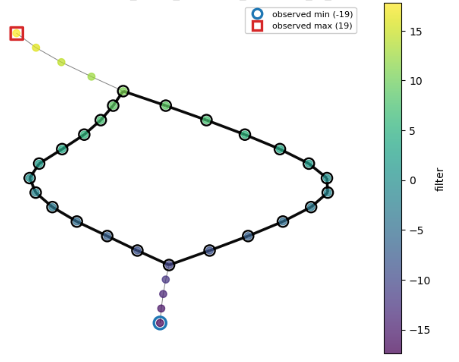}
        \caption{Mapper graph.}
    \end{subfigure}
    \hfill
    \begin{subfigure}[b]{0.30\textwidth}
        \centering
        \includegraphics[width=\textwidth]{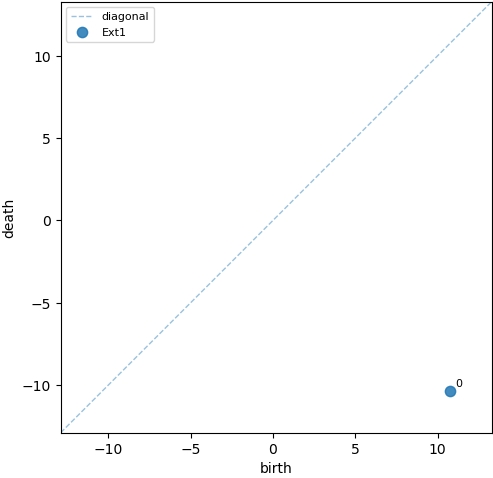}
        \caption{Mapper signature.}
    \end{subfigure}
    \hfill
    \begin{subfigure}[b]{0.30\textwidth}
        \centering
        \includegraphics[width=\textwidth]{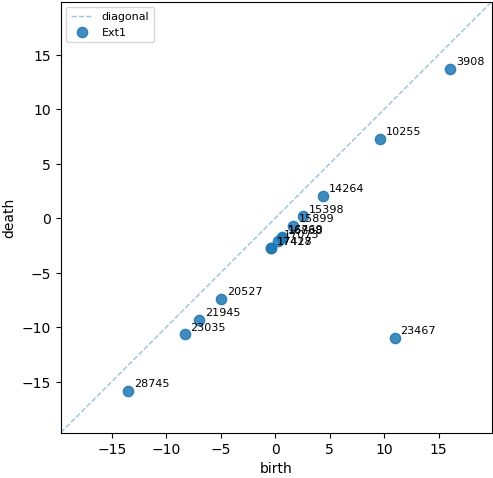}
        \caption{PL signature.}
    \end{subfigure}
    \caption{Hand-tuned extrinsic torus.  Both signatures display the dominant extended one-dimensional class in \(\operatorname{Ext}_1\), corresponding to the circular Reeb structure of the height function.  The Mapper graph highlights a representative cycle for this feature.}
    \label{fig:manual_torus_pipeline}
\end{figure}

\begin{figure}
    \centering
    \includegraphics[width=0.95\textwidth]{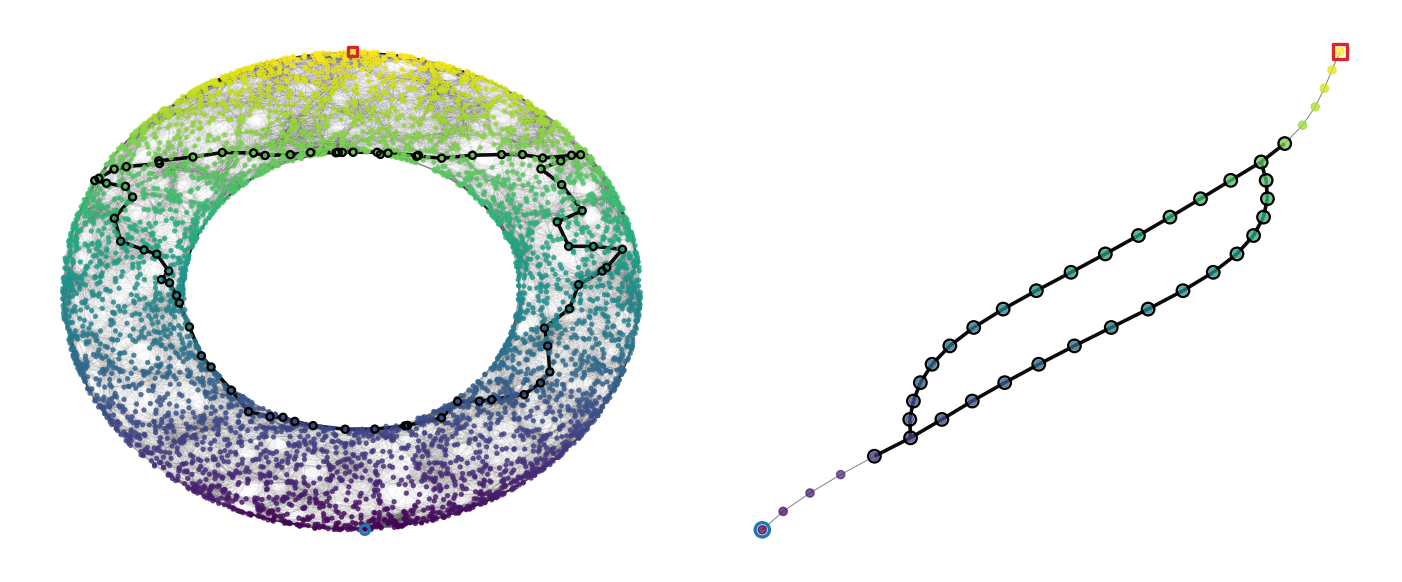}
    \caption{Hand-tuned extrinsic torus.  Left: a representative of the dominant \(\operatorname{Ext}_1\) PL feature drawn on the PL graph.  Right: the pushforward of the same feature to the Mapper graph via the multivalued map \(\Xi\).  The PL representative is difficult to parse geometrically, whereas the Mapper support makes the circular Reeb structure visually clear.  Some boundary effects in the displayed cycle may come from the multivalued nature of \(\Xi\), or from the gluing induced by the Mapper coarsening.}
    \label{fig:manual_torus_pl_to_mapper}
\end{figure}

\begin{figure}
    \centering
    \begin{subfigure}[b]{0.45\textwidth}
        \centering
        \includegraphics[width=\textwidth]{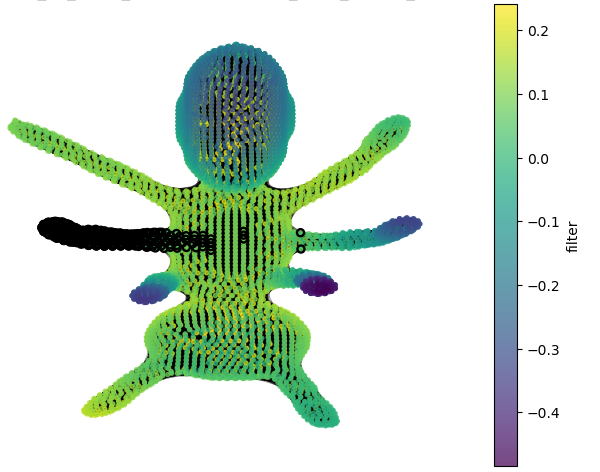}
        \caption{Intrinsic PL representative.}
    \end{subfigure}
    \hfill
    \begin{subfigure}[b]{0.45\textwidth}
        \centering
        \includegraphics[width=\textwidth]{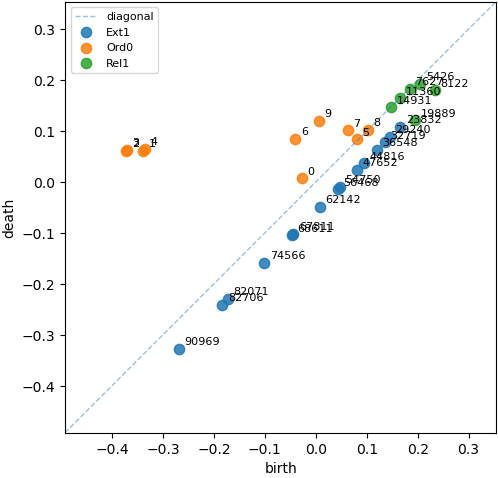}
        \caption{Intrinsic PL signature.}
    \end{subfigure}
    \caption{Hand-tuned intrinsic ant.  The highlighted support in the PL graph represents an ordinary zero-dimensional feature.  The PL signature displays the most persistent records of each type and already shows the richer branching structure of the ant.}
    \label{fig:manual_ant_pipeline}
\end{figure}

\begin{figure}
    \centering
    \begin{subfigure}[b]{\textwidth}
        \centering
        \includegraphics[width=0.82\textwidth]{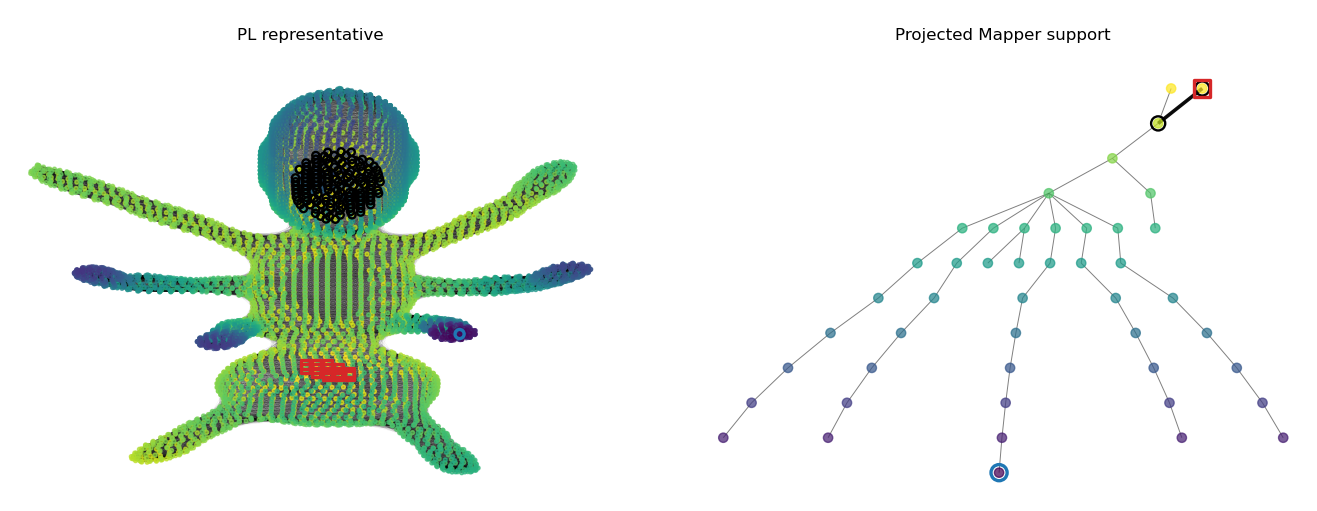}
        \caption{Euclidean metric.}
    \end{subfigure}

    \vspace{0.35em}

    \begin{subfigure}[b]{\textwidth}
        \centering
        \includegraphics[width=0.82\textwidth]{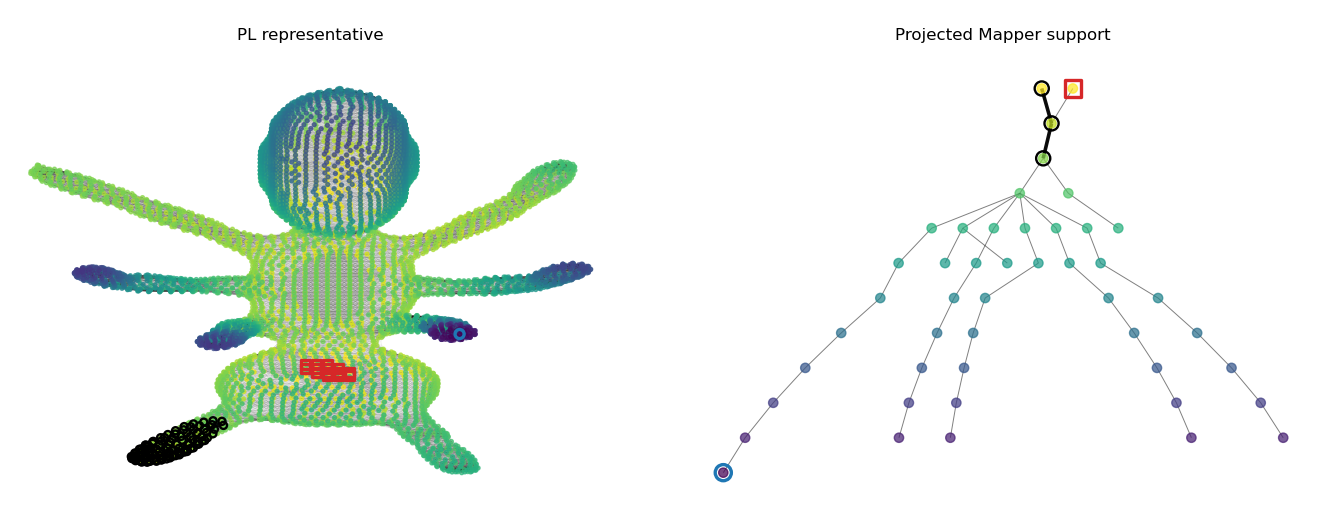}
        \caption{Intrinsic metric.}
    \end{subfigure}
    \caption{Hand-tuned ant, relative one-dimensional features on the PL graph and their Mapper pushforwards.  The intrinsic metric separates the geometry of the ant more faithfully: one selected feature is associated with the upward-pointing antenna, while the other corresponds to a local maximum in the abdomen.  After Mapper coarsening, the abdominal local maximum is merged with the component containing the global maximum at the head.}
    \label{fig:ant_pl_graph_metric_comparison}
\end{figure}

\Cref{fig:manual_torus_pipeline,fig:manual_torus_pl_to_mapper,fig:ant_pl_graph_metric_comparison} show the two roles of the pipeline.  The PL graph is the object on which approximation and statistical statements are made, while the Mapper graph is a simplified representation used to interpret selected PL features.  This distinction is especially visible for the torus cycle, whose PL representative is hard to read but whose Mapper support is clear.  The ant examples show the effect of the metric: intrinsic distances reduce ambient shortcuts and reveal features that the Euclidean graph can merge.

In the statistical experiments, persistence diagrams are drawn with confidence bands for the bottleneck distance; by \Cref{sec:dDelta_le_2dI}, the constants differ between zero- and one-dimensional features.  Persistence pairs outside the relevant band are interpreted as significant features of the unknown Reeb graph.  The stronger $d_I$-confidence statement also constrains how features may be rearranged: the interleaving maps compare the graphs through their smoothings, so features that remain distinct at the confidence scale cannot be identified by the interleaving maps.  Finally, features are interpreted as deviations from the trivial Reeb graph given by the interval \([\min f,\max f]\).  Components containing the global minimum or maximum are therefore not highlighted as features of interest: their existence is forced by the interval baseline.

\subsection{Using \texorpdfstring{\((a,b)\)}{(a,b)}-standard assumptions}

We now select \(\delta\) from the \((a,b)\)-standard tail bound.  The intrinsic dimension is fixed to \(b=2\).  Given the pilot radius
\[
r_n=C\left(\frac{\log n}{n}\right)^{1/(b+2)},
\]
with \(C=2\), we estimate the local mass constant by
\[
\widehat a_E
=
2^b\min_i\frac{P_n(B_E(X_i,r_n/2))}{r_n^b},
\qquad
\widehat a_X
=
2^b\min_i\frac{P_n(B_{d_X}(X_i,r_n/2))}{r_n^b}.
\]
The Euclidean estimate is the plug-in estimator used in the framework of~\cite{fasy}; the intrinsic version is the analogous numerical proxy computed from the Isomap distance matrix.  The dimension \(b\) is treated as known; if it were not known, one could use local PCA or related intrinsic-dimension estimators \cite{fan2010intrinsic}.  The selected scale \(\delta_\alpha\) is the smallest \(\delta\) such that
\[
\alpha_\delta
=\min\left\{1,\frac{4^b}{\widehat a\,\delta^b}
\exp\left(-\widehat a\left(\frac{\delta}{2}\right)^b n\right)\right\}
\le \alpha.
\]
The cover length is set to \(\chi=2\delta\), with overlap fraction \(0.45\).

\begin{table}[H]
\centering
\begin{tabular}{|c|c|c|c|c|c|c|}
\hline
Dataset & Metric & \(\widehat a\) & \(\delta\) & \(\chi\) & PL radius & Mapper radius\\
\hline
Torus & Euclidean & 0.004846 & 1.25528 & 2.51056 & 2.55371 & 5.06428\\
Torus & Intrinsic & 0.004846 & 1.25528 & 2.51056 & 1.25528 & 3.76584\\
Ant & Euclidean & 0.537648 & 0.106789 & 0.213578 & 0.213578 & 0.427156\\
Ant & Intrinsic & 0.508014 & 0.109860 & 0.219719 & 0.109860 & 0.329579\\
\hline
\end{tabular}
\caption{Parameters selected by the \((a,b)\)-standard plug-in procedure at confidence level \(0.85\).  The Mapper radius is the PL confidence radius plus the cover length \(\chi\).  For the torus, the Euclidean and intrinsic values of \(\widehat a\), \(\delta\), and \(\chi\) coincide because the pilot balls are very small: at this scale the Euclidean and intrinsic neighbourhood counts agree for the local-mass estimator.}
\label{tab:ab_parameters}
\end{table}

\begin{figure}
    \centering
    \begin{subfigure}[b]{0.24\textwidth}
        \centering
        \includegraphics[width=\textwidth]{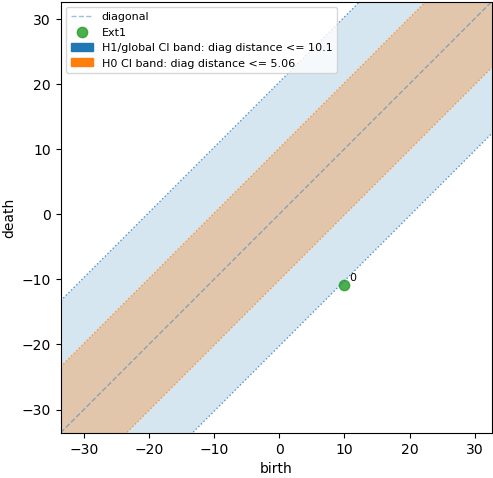}
        \caption{Ex. Mapper.}
    \end{subfigure}
    \begin{subfigure}[b]{0.24\textwidth}
        \centering
        \includegraphics[width=\textwidth]{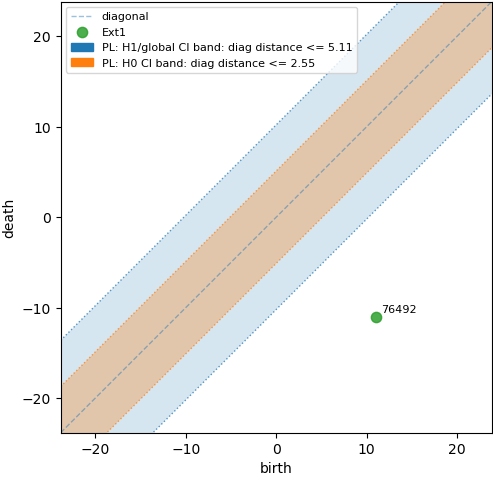}
        \caption{Ex. PL.}
    \end{subfigure}
    \begin{subfigure}[b]{0.24\textwidth}
        \centering
        \includegraphics[width=\textwidth]{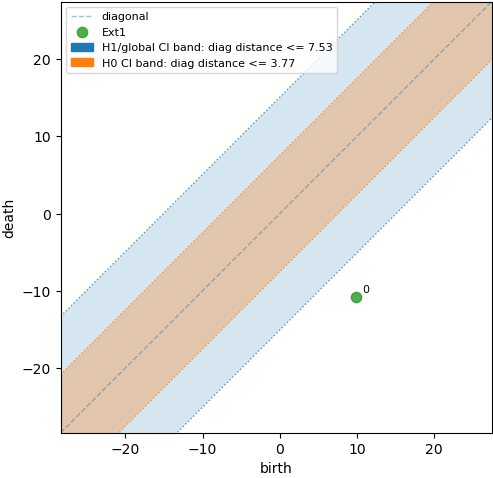}
        \caption{Int. Mapper.}
    \end{subfigure}
    \begin{subfigure}[b]{0.24\textwidth}
        \centering
        \includegraphics[width=\textwidth]{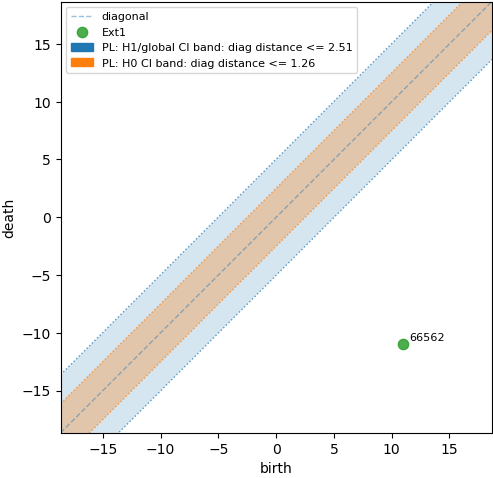}
        \caption{Int. PL.}
    \end{subfigure}
    \caption{Torus, \((a,b)\)-standard scale selection.  The panels compare Mapper and PL confidence bands for the Euclidean (Ex.) and intrinsic (Int.) metrics.  The dominant \(\operatorname{Ext}_1\) feature is separated from the diagonal in all settings, but the separation is much clearer in the PL diagrams than in the Mapper diagrams, and clearer for the intrinsic metric than for the Euclidean metric.}
    \label{fig:torus_ab_pipeline}
\end{figure}

For the torus, all four signatures in \Cref{fig:torus_ab_pipeline} identify the main loop in \(\operatorname{Ext}_1\), which is the expected topological feature of the height-filtered torus.  The equality of the Euclidean and intrinsic radii in \Cref{tab:ab_parameters} should not be over-interpreted: it is caused by the fixed-radius local-mass counts being identical at the small pilot scale.  The graph constructions remain different, and the intrinsic PL confidence band is still sharper.  This again illustrates why the statistical decision is made at the PL level, while the Mapper graph is used for visualization.

\begin{figure}
    \centering
    \begin{subfigure}[b]{0.24\textwidth}
        \centering
        \includegraphics[width=\textwidth]{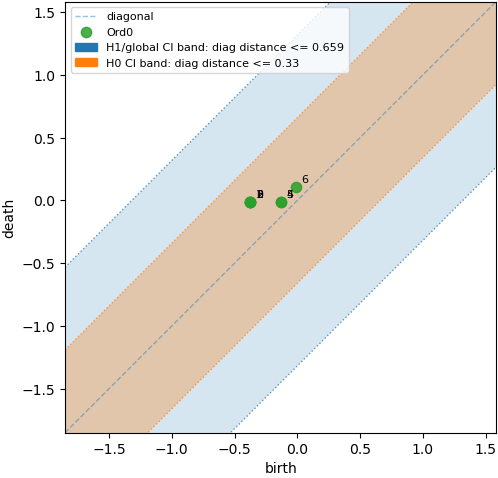}
        \caption{Int. Mapper.}
    \end{subfigure}
    \begin{subfigure}[b]{0.24\textwidth}
        \centering
        \includegraphics[width=\textwidth]{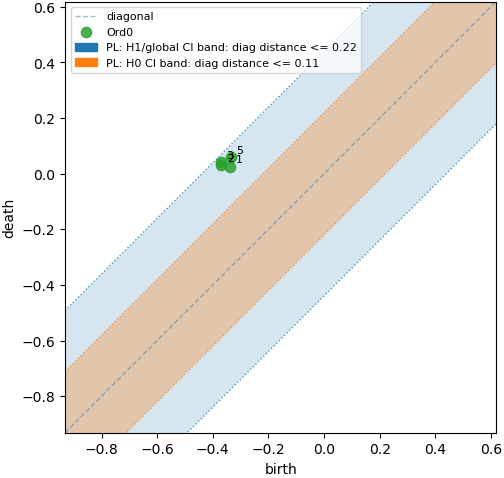}
        \caption{Int. PL.}
    \end{subfigure}
    \begin{subfigure}[b]{0.24\textwidth}
        \centering
        \includegraphics[width=\textwidth]{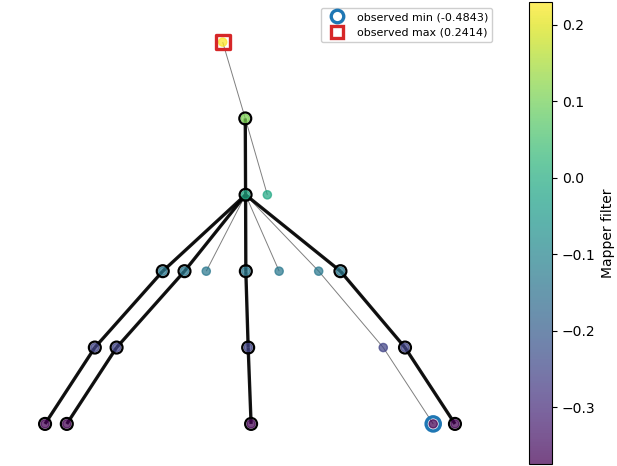}
        \caption{Int. pushforward.}
    \end{subfigure}
    \begin{subfigure}[b]{0.24\textwidth}
        \centering
        \includegraphics[width=\textwidth]{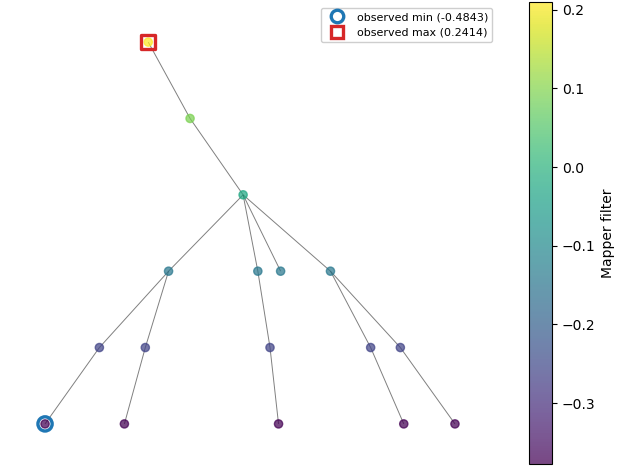}
        \caption{Ex. pushforward.}
    \end{subfigure}
    \caption{Ant, \((a,b)\)-standard scale selection.  Compared with \Cref{fig:torus_ab_pipeline}, significant features are harder to detect.  The Euclidean construction does not produce non-baseline significant features, and the intrinsic Mapper signature remains too coarse: its persistence pairs lie inside the confidence bands.  The intrinsic PL graph, however, detects several significant branching features in \(\operatorname{Ext}_0\), which are pushed forward to the Mapper graph in the intrinsic panel.  As expected from \Cref{fig:ant_pl_graph_metric_comparison}, these features are associated with the four legs with lower local minima and with the tail of the ant.}
    \label{fig:ant_ab_pipeline}
\end{figure}

The ant data set is more difficult.  The Euclidean construction gives a Mapper graph close to the hand-tuned picture, but the confidence radius is too large for non-baseline features to be significant.  The intrinsic Mapper confidence set is still too coarse, whereas the intrinsic PL graph is fine enough for the branching features visible in \Cref{fig:ant_ab_pipeline} to be significant.  This is the clearest instance in the experiments where the intrinsic metric and the PL-level confidence statement are both needed.

\subsection{Rate-corrected subsampling}

We next select the scale by subsampling.  For a subsample size \(b_n\), the raw Method-I radius of~\cite{fasy} is
\[
\delta_{b_n,\alpha}^{\rm raw}=2L_{b_n}^{-1}(\alpha),
\qquad
L_{b_n}(t)=\frac1N\sum_{j=1}^N\mathbbm 1_{\{d_H(S_{b_n,n}^j,S_n)>t\}},
\]
where the experiments use \(N=1000\) resamples and
\[
b_n=\left\lfloor \frac{n}{(\log n)^{1.001}}\right\rfloor .
\]
Thus \(b_n=585\) for the torus and \(b_n=725\) for the ant.  As discussed in \Cref{sub}, the raw radius is conservative in finite
samples. We therefore consider two rate corrections, using a slightly
simplified notation with respect to \Cref{sub}:
\[
\delta_{\log}
=
\delta_{b_n,\alpha}^{\rm raw}
\left(\frac{\log(n)/n}{\log(b_n)/b_n}\right)^{1/d},
\qquad
\delta_{\rm pow}
=
\delta_{b_n,\alpha}^{\rm raw}
\left(\frac{b_n}{n}\right)^{1/d}.
\]
The logarithmic correction \(\delta_{\log}\) is more conservative, whereas the no-log correction \(\delta_{\rm pow}\) is more aggressive.  In the two-dimensional supports used for the displayed experiments, the validation results of \Cref{subsec:subsampling_validation_appendix} show that \(\delta_{\rm pow}\) performs very well.  We therefore use \(\delta_{\rm pow}\) for the main figures, and we also check \(\delta_{\log}\) to rule out artifacts caused by an overly aggressive correction.  The resulting significant features are comparable.

\begin{table}[H]
\centering
\begin{tabular}{|c|c|c|c|c|c|c|}
\hline
Dataset & Metric & \(\delta_{\rm raw}\) & \(\delta_{\log}\) & \(\delta_{\rm pow}\) & PL radius & Mapper radius\\
\hline
Torus & Euclidean & 6.77331 & 2.67866 & 2.31683 & 4.94205 & 9.57570\\
Torus & Intrinsic & 8.07747 & 2.73142 & 2.35050 & 2.35050 & 7.05150\\
Ant & Euclidean & 0.186042 & 0.072382 & 0.062764 & 0.125528 & 0.251056\\
Ant & Intrinsic & 0.195055 & 0.075889 & 0.065805 & 0.065805 & 0.197414\\
\hline
\end{tabular}
\caption{Subsampling radii.  The displayed experiments use \(\delta_{\rm pow}\), the no-log rate correction.  The raw radius and the logarithmic correction are reported for comparison.}
\label{tab:sub_parameters}
\end{table}

\begin{figure}
    \centering
    \begin{subfigure}[b]{0.24\textwidth}
        \centering
        \includegraphics[width=\textwidth]{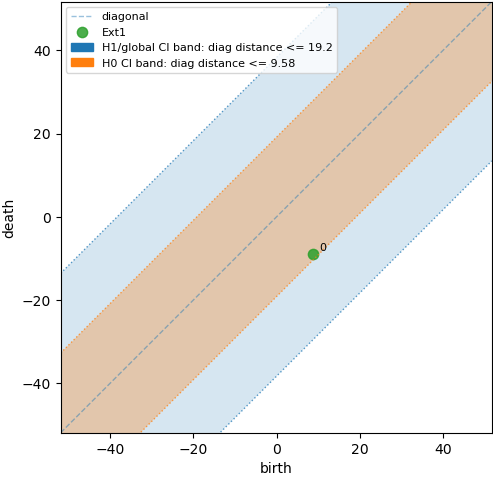}
        \caption{Ex. Mapper.}
    \end{subfigure}
    \begin{subfigure}[b]{0.24\textwidth}
        \centering
        \includegraphics[width=\textwidth]{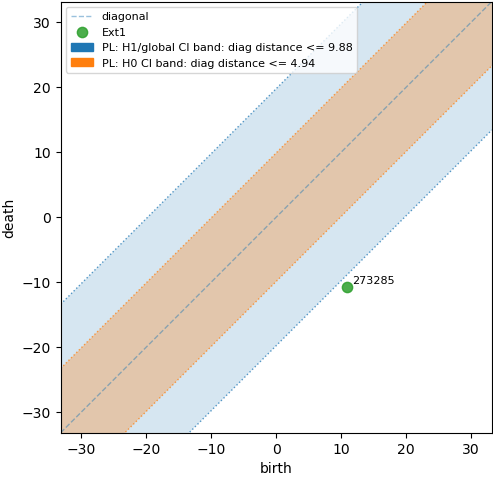}
        \caption{Ex. PL.}
    \end{subfigure}
    \begin{subfigure}[b]{0.24\textwidth}
        \centering
        \includegraphics[width=\textwidth]{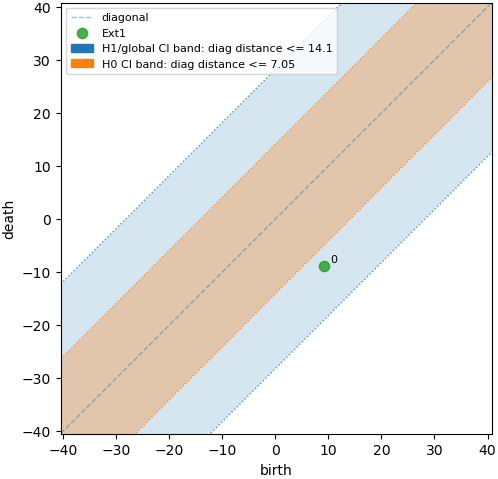}
        \caption{Int. Mapper.}
    \end{subfigure}
    \begin{subfigure}[b]{0.24\textwidth}
        \centering
        \includegraphics[width=\textwidth]{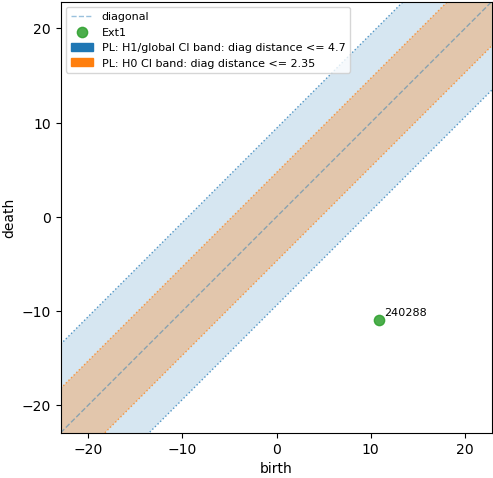}
        \caption{Int. PL.}
    \end{subfigure}
    \caption{Torus, subsampling with the no-log correction.  The dominant \(\operatorname{Ext}_1\) feature is separated from the diagonal in the PL diagrams and in the intrinsic Mapper diagram.  The Euclidean Mapper confidence band is too wide for it to be significant.}
    \label{fig:torus_sub_pipeline}
\end{figure}

\begin{figure}
    \centering
    \begin{subfigure}[b]{0.24\textwidth}
        \centering
        \includegraphics[width=\textwidth]{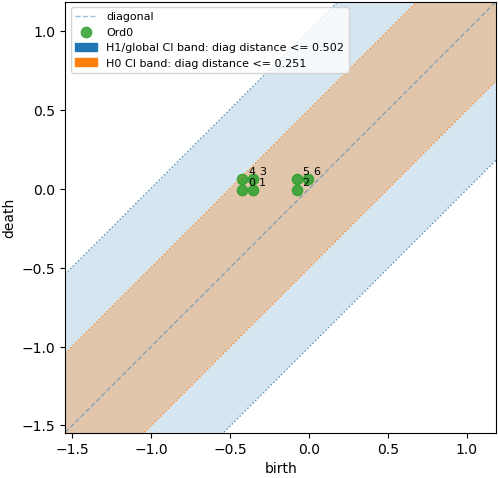}
        \caption{Ex. PL.}
    \end{subfigure}
    \begin{subfigure}[b]{0.24\textwidth}
        \centering
        \includegraphics[width=\textwidth]{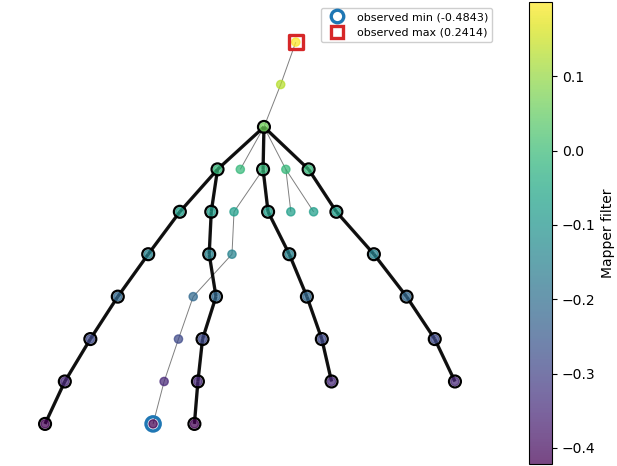}
        \caption{Ex. pushforward.}
    \end{subfigure}
    \begin{subfigure}[b]{0.24\textwidth}
        \centering
        \includegraphics[width=\textwidth]{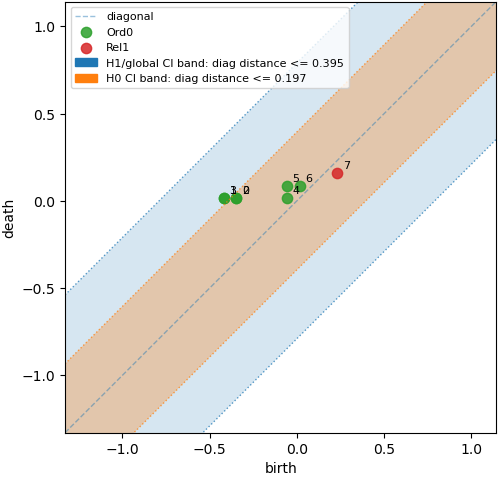}
        \caption{Int. PL.}
    \end{subfigure}
    \begin{subfigure}[b]{0.24\textwidth}
        \centering
        \includegraphics[width=\textwidth]{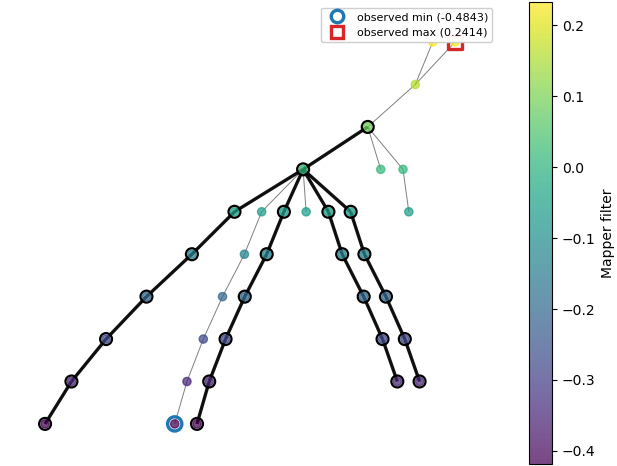}
        \caption{Int. pushforward.}
    \end{subfigure}
    \caption{Ant, subsampling with the no-log rate correction.  Both PL graphs capture several significant branching features in \(\operatorname{Ext}_0\), which are pushed forward to the Mapper graphs.  As before, those features are associated with the four legs with lower local minima, and the fifth one is associated with the tail of the ant.  No significant features are captured by the Mapper-induced confidence sets.}
    \label{fig:ant_sub_pipeline}
\end{figure}

For the torus, \Cref{fig:torus_sub_pipeline} gives the same qualitative conclusion as the \((a,b)\)-standard experiment.  The main \(\operatorname{Ext}_1\) loop is significant in the PL diagrams, and the intrinsic metric gives a stronger separation from the confidence band.  For the ant, \Cref{fig:ant_sub_pipeline} shows that the no-log subsampling correction is sufficiently selective for the relevant branching features to be significant on the PL graph.  The Mapper graph remains the visualization device: it shows where the significant PL features live, but its own confidence radius is too coarse for the same features to be significant directly.

\section{Conclusions and future development}

This paper develops confidence regions for Reeb graphs from finite samples using
the interleaving distance.  The main estimator is the Reeb cosheaf of a PL
filtered proximity graph built on the data.  This construction is finite and
computable, and the stability results proved above compare it with the target
Reeb cosheaf in both intrinsic and extrinsic metrics.  Mapper is then used as a
controlled coarsening of the PL estimator: it gives a readable graph for
visualization, while the coarsening error is bounded explicitly by the cover
resolution.  In this way the pipeline separates two tasks that are often
combined in applications: PL--Reeb graphs provide the object used for
approximation and statistical inference, while Mapper graphs provide the object
used for geometric interpretation.

The paper also relates this metric viewpoint to persistence-based summaries.  We
show that the extended-persistence pseudometric is controlled by the interleaving
distance, with sharper constants for the zero-dimensional components.  Therefore
interleaving confidence regions induce confidence statements for extended
persistence diagrams, while still retaining an object-level interpretation for
the underlying Reeb-type estimator.  The numerical experiments illustrate this
principle: significant features are detected from confidence bands in the PL
extended diagram and are then pushed forward to Mapper graphs for visualization.

Several directions remain open.  First, the rate-corrected subsampling radii used
in the experiments should be put on a theoretical footing.  The correction is
motivated by the distinction between two covering problems: the subsampling
quantile measures how well a subsample of size \(b_n\) covers the observed cloud
\(S_n\), whereas the estimator requires the full sample \(S_n\) to cover the
underlying support \(X\).  The experiments suggest that correcting for this
change of scale is effective, but the present paper validates the procedure only
empirically.  The runtime benchmarks in \Cref{subsec:runtime_benchmarks_appendix}
show that such improvements in the estimation of \(\delta\) are also
computationally valuable: the selected radius controls the edge density of the PL
graph and hence the cost of the persistence computation.  A full proof would
therefore strengthen both the statistical calibration of the confidence regions
and the practical efficiency of the pipeline, by justifying smaller but still
reliable graph-construction radii.  Second, one should develop theoretical
results for bootstrap estimates of the Hausdorff sampling scale.  Bootstrap
Hausdorff radii are already discussed in the Mapper-confidence literature as a
practical exploratory tool \cite{JMLR:v19:17-291}, but without a corresponding
validity theorem.  Establishing such results, in parallel with the theory needed
for the rate correction above, would make the data-driven calibration of the PL
estimator more robust.  Third, the pipeline would benefit from improved
estimation procedures for the other geometric quantities entering the bounds,
such as reach, moduli of continuity, and intrinsic covering constants.  Finally,
as recalled in the introduction, Mapper has already been useful in a range of
applications.  This suggests that the PL-for-inference and Mapper-for-visualization
pipeline developed here could provide a fertile framework for more complex data
sets, especially in settings where one wants both statistically controlled
features and an interpretable graph-level summary.

\section*{Acknowledgments}

This work originated in part from the Master’s thesis of A.C., supervised by Alessandra Micheletti. The authors thank Alessandra Micheletti for her guidance during the development of the thesis and for the discussions that contributed to this work. M.P. acknowledges support from the Fondo Istituzionale per la Ricerca of Università della Svizzera italiana through the project \emph{Using Topological Data Analysis to Understand Microglia Shape Variability in Space and Time}.

\section*{Use of Large Language Models}

Large language models were used occasionally to improve the clarity and style of the manuscript and to assist in checking the consistency of selected arguments. All mathematical content and final judgments remain the responsibility of the authors.

\appendix

\section{Comparison with~\cite{JMLR:v19:17-291} and
\cite{bjerkevik2025reeb}}
\label{sec:jmlr_comparison}

We compare our approximation--coarsening pipeline with the deterministic Mapper
bounds of~\cite[Theorem~7 and Remark~8]{JMLR:v19:17-291}, and our extrinsic
PL--Reeb estimator with the sample-thickening construction of
\cite[Corollary~5.6]{bjerkevik2025reeb}. The two comparisons emphasize different
features of our approach. The first shows how separating the PL approximation
from the Mapper coarsening improves the dependence on the cover resolution. The
second compares two estimators that use different amounts of information about
the filter away from the sample.

\subsection{Comparison with~\cite{JMLR:v19:17-291}}

Let \(X\subset\mathbb R^m\) be compact with positive reach
\(\tau=\rch(X)\), let \(S_n\subseteq X\) be finite, and assume that
\[
d_H^E(S_n,X)\le\delta<\tau.
\]
Set
\[
\delta_X
\coloneqq
\psi_\tau(\delta)
=
2\tau\arcsin\!\left(\frac{\delta}{2\tau}\right).
\]
By \Cref{thm:metric_distortion_reach},
\(d_H^X(S_n,X)\le\delta_X\).

Assume that \(f\colon X\to\mathbb R\) is \(L\)-Lipschitz with respect to the
Euclidean metric, and set
\[
R\coloneqq\mathfrak R(X,f).
\]
Since \(\|x-x'\|\le d_X(x,x')\), it is also
\(L\)-Lipschitz with respect to \(d_X\). Hence
\Cref{thm:intrinsic_reeb_stability,prop:mapper_coarsening_error} give
\begin{equation}\label{eq:our_comparison_intrinsic}
d_I\!\left(
\mathcal F_R,
\widetilde{\mathcal M}^{S_n,d_X}_{\delta_X,\mathcal U_X}
\right)
\le
L\delta_X+
\res_{\mathcal R_{\delta_X}^{S_n,d_X}}(\mathcal U_X).
\end{equation}
In particular, at the PL--Reeb level,
\begin{equation}\label{eq:our_comparison_intrinsic_PL}
d_I\!\left(
\mathcal F_R,
\mathcal R_{\delta_X}^{S_n,d_X}
\right)
\le
L\delta_X.
\end{equation}

We now recall only the part of
\cite[Theorem~7]{JMLR:v19:17-291} needed for the comparison. Its Mapper
estimator depends on a Euclidean neighborhood-graph scale
\(\delta_{\mathrm J}\), a regular interval cover of length \(r\), and a gain
\(g\in(0,1/2)\). We denote it by
\(M^{\mathrm J}_{r,g,\delta_{\mathrm J}}\). Under the geometric and
Morse-type assumptions of that theorem, and provided that
\(4d_H^E(S_n,X)\le\delta_{\mathrm J}\) and
\[
\max\left\{
|f(x)-f(x')|:
x,x'\in S_n,\ \|x-x'\|\le\delta_{\mathrm J}
\right\}
<gr,
\]
one has
\begin{equation}\label{eq:jmlr_theorem7_bound}
d_\Delta\!\left(
R,
M^{\mathrm J}_{r,g,\delta_{\mathrm J}}
\right)
\le
r+2\omega_f^E(\delta_{\mathrm J}),
\end{equation}
where \(\omega_f^E\) is the Euclidean modulus of continuity of \(f\).

At the common Euclidean sampling scale \(\delta\), take
\[
\delta_{\mathrm J}=4\delta.
\]
This satisfies the sampling condition
\(4d_H^E(S_n,X)\le\delta_{\mathrm J}\). The additional geometric scale condition
in~\cite[Theorem~7]{JMLR:v19:17-291} then becomes
\[
4\delta
\le
\frac14\min\{\rch(X),\rho(X)\}.
\]
Since \(f\) is Euclidean \(L\)-Lipschitz, the filter-admissibility condition is
satisfied, for every \(\alpha>0\), by choosing
\[
r=\frac{4L}{g}\delta+\alpha.
\]
Equation~\eqref{eq:jmlr_theorem7_bound} therefore yields
\begin{equation}\label{eq:jmlr_normalized_standard}
d_\Delta\!\left(
R,
M^{\mathrm J}_{r,g,4\delta}
\right)
\le
\left(\frac4g+8\right)L\delta+\alpha.
\end{equation}

For the edge-based MultiNerve construction,
\cite[Remark~8]{JMLR:v19:17-291} replaces the admissibility threshold \(gr\)
by \(r\), and the term \(r\) in
\eqref{eq:jmlr_theorem7_bound} by \(r/2\). Thus the choice
\(r=4L\delta+\alpha\) gives
\begin{equation}\label{eq:jmlr_normalized_multinerve}
d_\Delta\!\left(
R,
M^{\mathrm J,\mathrm{eMN}}_{r,g,4\delta}
\right)
\le
10L\delta+\frac{\alpha}{2}.
\end{equation}

We first compare the PL approximations. By
\Cref{thm:dDelta_le_2dI}, \eqref{eq:our_comparison_intrinsic_PL} gives
\[
d_\Delta\!\left(
\mathcal F_R,
\mathcal R_{\delta_X}^{S_n,d_X}
\right)
\le
2L\delta_X,
\]
with the sharper bound \(L\delta_X\) for the \(H_0\)-related components
\(\Ord_0\), \(\Ext_0\), and \(\Rel_1\). Since
\[
\delta_X
=
\delta+
O\!\left(\frac{\delta^3}{\tau^2}\right),
\]
the full-signature PL--Reeb bound is \(2L\delta\) to first order, whereas the
bounds in \eqref{eq:jmlr_normalized_standard} and
\eqref{eq:jmlr_normalized_multinerve} are, respectively,
\[
\left(\frac4g+8\right)L\delta+\alpha
\qquad\text{and}\qquad
10L\delta+\frac{\alpha}{2}.
\]

We now fix a common deterministic error level and compare the cover resolutions
required by the two approaches to attain it. Combining
\eqref{eq:our_comparison_intrinsic} with
\Cref{thm:dDelta_le_2dI}, our Mapper coarsening satisfies
\[
d_\Delta
\le
2\left(
L\delta_X+
\res_{\mathcal R_{\delta_X}^{S_n,d_X}}(\mathcal U_X)
\right).
\]
It therefore achieves the same deterministic full-signature bound as the Mapper
estimate \eqref{eq:jmlr_normalized_standard} whenever
\begin{equation}\label{eq:comparison_resolution_standard}
\res_{\mathcal R_{\delta_X}^{S_n,d_X}}(\mathcal U_X)
\le
\left(\frac2g+4\right)L\delta
-
L\delta_X
+
\frac{\alpha}{2}.
\end{equation}
Since \(\delta_X=\delta+O(\delta^3/\tau^2)\), the right-hand side is
\[
\left(\frac2g+3\right)L\delta
+
\frac{\alpha}{2}
+
O\!\left(\frac{L\delta^3}{\tau^2}\right).
\]
The corresponding cover length in~\cite{JMLR:v19:17-291} is
\[
r=\frac{4L}{g}\delta+\alpha.
\]

Thus the two methods allow cover scales of the same order. At the limiting
value \(g\uparrow1/2\) and for \(\alpha\downarrow0\), the admissible resolution
of our cover approaches \(7L\delta\), whereas the cover length required
in~\cite{JMLR:v19:17-291} approaches \(8L\delta\), up to the higher-order reach
correction in \eqref{eq:comparison_resolution_standard}. For the value \(g=0.4\)
used in the experiments of~\cite{JMLR:v19:17-291}, the admissible resolution of
our cover is asymptotically \(8L\delta\), whereas the corresponding cover length
in~\cite{JMLR:v19:17-291} is \(10L\delta\).

Similarly, our Mapper coarsening achieves the same deterministic
full-signature bound as the edge-based MultiNerve estimate
\eqref{eq:jmlr_normalized_multinerve} whenever
\begin{equation}\label{eq:comparison_resolution_multinerve}
\res_{\mathcal R_{\delta_X}^{S_n,d_X}}(\mathcal U_X)
\le
5L\delta
-
L\delta_X
+
\frac{\alpha}{4}.
\end{equation}
To first order, the admissible resolution is therefore
\[
4L\delta+\frac{\alpha}{4},
\]
whereas the corresponding MultiNerve cover length is
\[
r=4L\delta+\alpha.
\]
The two cover scales consequently agree to first order as
\(\delta/\tau\to0\) and \(\alpha\downarrow0\).

At the level of the final Mapper graphs, the two approaches therefore yield
bounds of the same order and comparable admissible cover resolutions, up to the
smaller differences induced by their respective constructions and experimental
setups. The relevant improvement in our framework lies instead at the PL--Reeb
level, where the sharper approximation bound is used for statistical
certification before any Mapper coarsening is introduced.

An analogous comparison may be carried out using the extrinsic PL--Reeb
estimator. In that case, one replaces the intrinsic approximation term
\(L\delta_X\) above by
\[
L\eta_\tau(\delta)
=
2\tau L\arcsin\!\left(\frac{\delta}{\tau}\right).
\]

\subsection{Comparison with~\cite{bjerkevik2025reeb}}

We finally compare our extrinsic PL--Reeb estimator with the sample-thickening
construction of~\cite{bjerkevik2025reeb}. Let \(X\subset\mathbb R^m\) and
\(S_n\subseteq X\) satisfy the assumptions of
\cite[Corollary~5.6]{bjerkevik2025reeb}, with its scale parameter equal to
\(2\delta\), and assume that \(d_H^E(S_n,X)\le\delta\).

Let
\[
\widetilde f\colon
\left\{
z\in\mathbb R^m:
\dist(z,S_n)\le2\delta
\right\}
\longrightarrow\mathbb R
\]
be Euclidean \(L\)-Lipschitz, and set \(f=\widetilde f|_X\). Write
\[
R\coloneqq\mathfrak R(X,f),
\qquad
\widetilde R
\coloneqq
\mathfrak R\!\left(
\{z\in\mathbb R^m:\dist(z,S_n)\le2\delta\},
\widetilde f
\right).
\]
Then \cite[Corollary~5.6]{bjerkevik2025reeb} gives
\begin{equation}\label{eq:bjerkevik_bound}
d_I\!\left(
\mathcal F_{\widetilde R},
\mathcal F_R
\right)
<
2L\delta.
\end{equation}

On the other hand, writing \(\tau=\rch(X)\), our extrinsic PL--Reeb estimate is
\begin{equation}\label{eq:bjerkevik_our_bound}
d_I\!\left(
\mathcal F_R,
\mathcal R_{\delta}^{S_n,\|\cdot\|}
\right)
\le
L\eta_\tau(\delta)
=
2\tau L\arcsin\!\left(\frac{\delta}{\tau}\right).
\end{equation}

The bound in \eqref{eq:bjerkevik_bound} is slightly sharper, but it relies on
additional functional information: the filter is assumed to be defined and
\(L\)-Lipschitz on the whole Euclidean offset of the sample. Its variation along
the relevant deformation can therefore be controlled directly by the ambient
displacement. Our estimator instead uses only the sampled values \(f|_{S_n}\).
An edge of \(\Gamma_{\delta}^{S_n,\|\cdot\|}\), whose endpoints may be at
Euclidean distance \(2\delta\), must be represented by a path in \(X\), and the
positive-reach distortion estimate gives the larger scale
\[
\eta_\tau(\delta)
=
2\tau\arcsin\!\left(\frac{\delta}{\tau}\right).
\]
Nevertheless,
\[
L\eta_\tau(\delta)
=
2L\delta+
O\!\left(\frac{L\delta^3}{\tau^2}\right),
\]
so the two bounds agree to first order as \(\delta/\tau\to0\).

\section{Additional implementation, visualization, validation, and runtime notes}\label{sec:implementation_visualization_notes}

This appendix records the implementation and display conventions needed to reproduce the numerical experiments in \Cref{expes}.  The numerical pipeline first constructs the PL graph, then computes persistence and confidence information on that graph, and only afterwards forms Mapper coarsenings for visualization.  We also report runtime benchmarks for the PL--Reeb/Mapper pipeline and validate the rate-corrected subsampling radii used in the experiments.

\subsection{Sparse PL graph construction and Mapper coarsening}

The first step is to build the proximity graph \(\Gamma_\delta^{S_n,\rho}\).  In the experiments the filter is the height function, and therefore an edge \((x_i,x_j)\) can occur only if the filter values of its endpoints are sufficiently close.  The implementation exploits this necessary condition before performing any metric-threshold test.  It sorts the sample by filter value and restricts distance computations to the candidate window
\[
\mathcal C_W=
\bigl\{(i,j): i<j,
 |f(x_i)-f(x_j)|\le W\bigr\},
\]
where \(W\) is chosen from the modulus-of-continuity bound used in the corresponding theorem.  Candidate pairs are then tested against the relevant distance threshold.  For the Euclidean estimator this is done by chunked radius queries, or equivalently by a single radius-neighbour query in the smaller runs.  For the intrinsic estimator, the implementation first builds a \(10\)-nearest-neighbour graph, computes the associated shortest-path distance matrix, and thresholds this matrix at the selected intrinsic radius.  The intrinsic construction is therefore more expensive, but it reduces the ambient shortcuts that can appear in the Euclidean graph.

The PL graph is stored initially as edge arrays, rather than immediately as a NetworkX object, to avoid unnecessary overhead.  Conversion to a NetworkX graph is delayed until it is needed for extended-persistence diagnostics, representative extraction, or plotting; NetworkX is used for these graph-level operations \cite{networkx}.  This is the first point at which the implementation relies on NetworkX.  In the benchmarked regimes reported below, the edge-array construction itself is inexpensive.  The dominant cost is the extended-persistence computation on the PL graph, while conversion to a NetworkX graph and Mapper coarsening are visible but secondary costs.

Mapper is computed from the PL graph, not directly from the raw point cloud.  Each graph edge is regarded as a PL interval with linearly interpolated filter value.  If \(U\) is a cover interval, the preimage of \(U\) in such an edge may be the whole open edge, a proper open subinterval, or the empty set.  These open-edge pieces are included when forming connected components of the pullback cover.  A vertex-only treatment would miss components supported in edge interiors and would give the wrong overlap graph near cover boundaries.

Persistence is computed once on the PL object.  For visualization, representatives are then restricted to a controlled family: at most fifteen records per persistence type, ordered by persistence, after optionally discarding the baseline \(\operatorname{Ext}_0\) class.  The PL-to-Mapper visualization pushes each chosen PL representative forward to the Mapper graph through the multivalued map \(\Xi_I\).

\subsection{Runtime benchmarks}\label{subsec:runtime_benchmarks_appendix}

We benchmarked the numerical pipeline on the same ant and torus examples used in \Cref{expes}.  The benchmark grid uses ten independent repetitions for each data set, sample size, metric, and scale-selection rule.  For the ant we use sample sizes \(n\in\{1500,2500,5000\}\); for the torus we use \(n\in\{1500,2500,5000,7500\}\).  Each run is evaluated with the Euclidean metric and with the intrinsic Isomap metric based on ten nearest neighbours.  We compare two scale choices.  The deterministic scale uses the hand-tuned values from \Cref{fig:manual_torus_pipeline,fig:manual_ant_pipeline}: for the torus, \(\delta_E=0.9\), \(\delta_X=1.2\), and \(\chi=3.2\); for the ant, \(\delta_E=0.04\operatorname{range}(f)\), \(\delta_X=0.06\operatorname{range}(f)\), and \(\chi=0.15\operatorname{range}(f)\).  The statistical scale is obtained by applying the \((a,b)\)-standard calibration to the sample generated in the current repetition: the local-mass profile is estimated from that sample, the corresponding constant \(a\) is selected with the same rule as in the numerical experiments, and \(\delta\) is then obtained from the \((a,b)\)-standard tail bound with \(b=2\), confidence level \(1-\alpha=0.85\), pilot constant \(C=2\), and cover length \(\chi=2\delta\).  Thus the statistical radius is recomputed by the benchmark code in each repetition, although in the torus runs the recorded values are identical across repetitions and across the two metrics because the small pilot balls give the same local-mass counts.  In all cases the overlap fraction is \(0.45\).

The timing measurements separate the intrinsic-distance preprocessing, the windowed PL edge-array construction, the conversion of the PL graph to a NetworkX graph, the extended-persistence computation on the PL graph, the Mapper coarsening, and the Mapper persistence computation.  \Cref{fig:runtime_total} shows the resulting total runtimes, while \Cref{tab:runtime_largest} reports a component-level breakdown for the largest sample size of each data set.

\begin{figure}[H]
    \centering
    \begin{subfigure}[b]{0.48\textwidth}
        \centering
        \includegraphics[width=\textwidth]{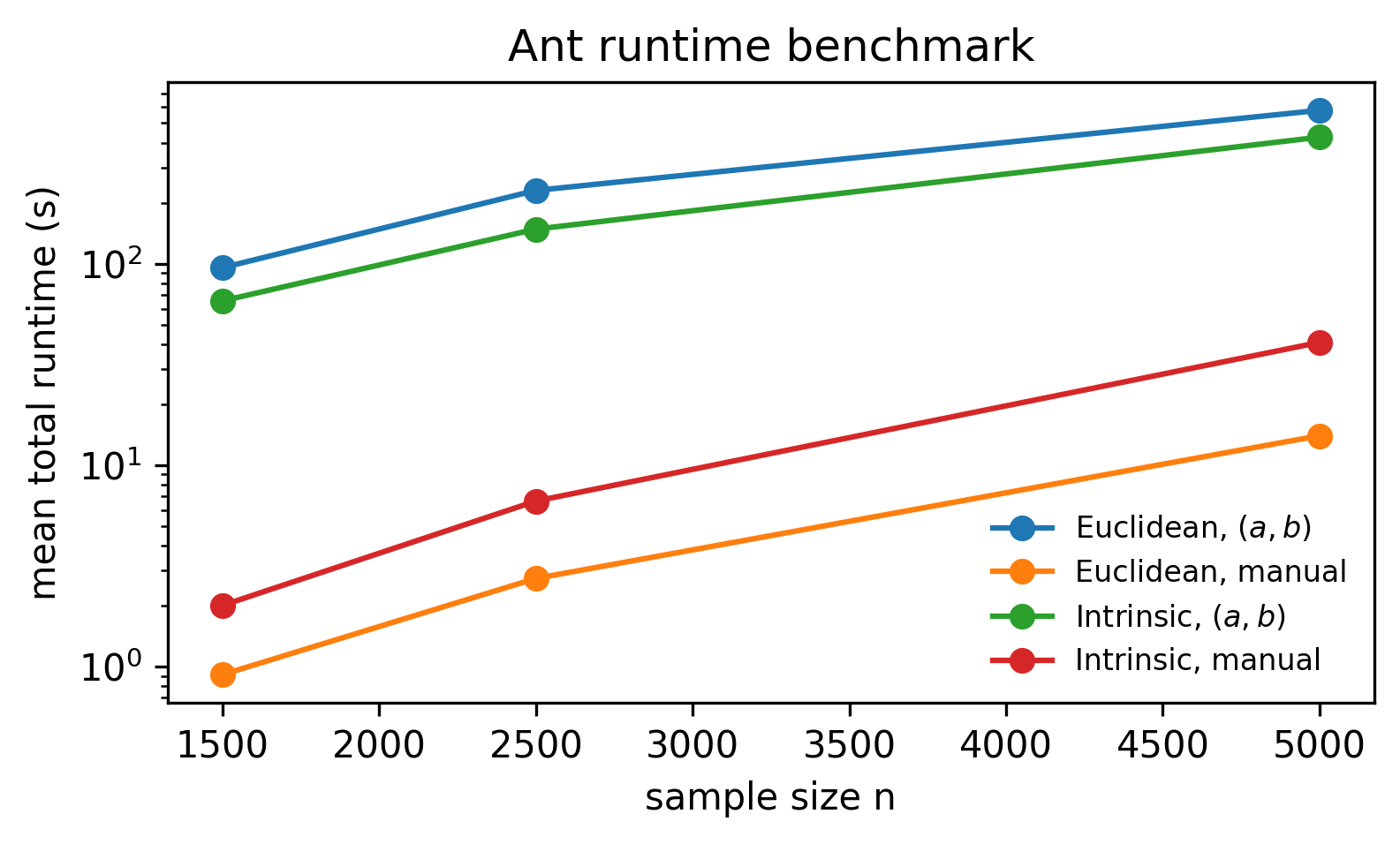}
        \caption{Ant.}
    \end{subfigure}
    \hfill
    \begin{subfigure}[b]{0.48\textwidth}
        \centering
        \includegraphics[width=\textwidth]{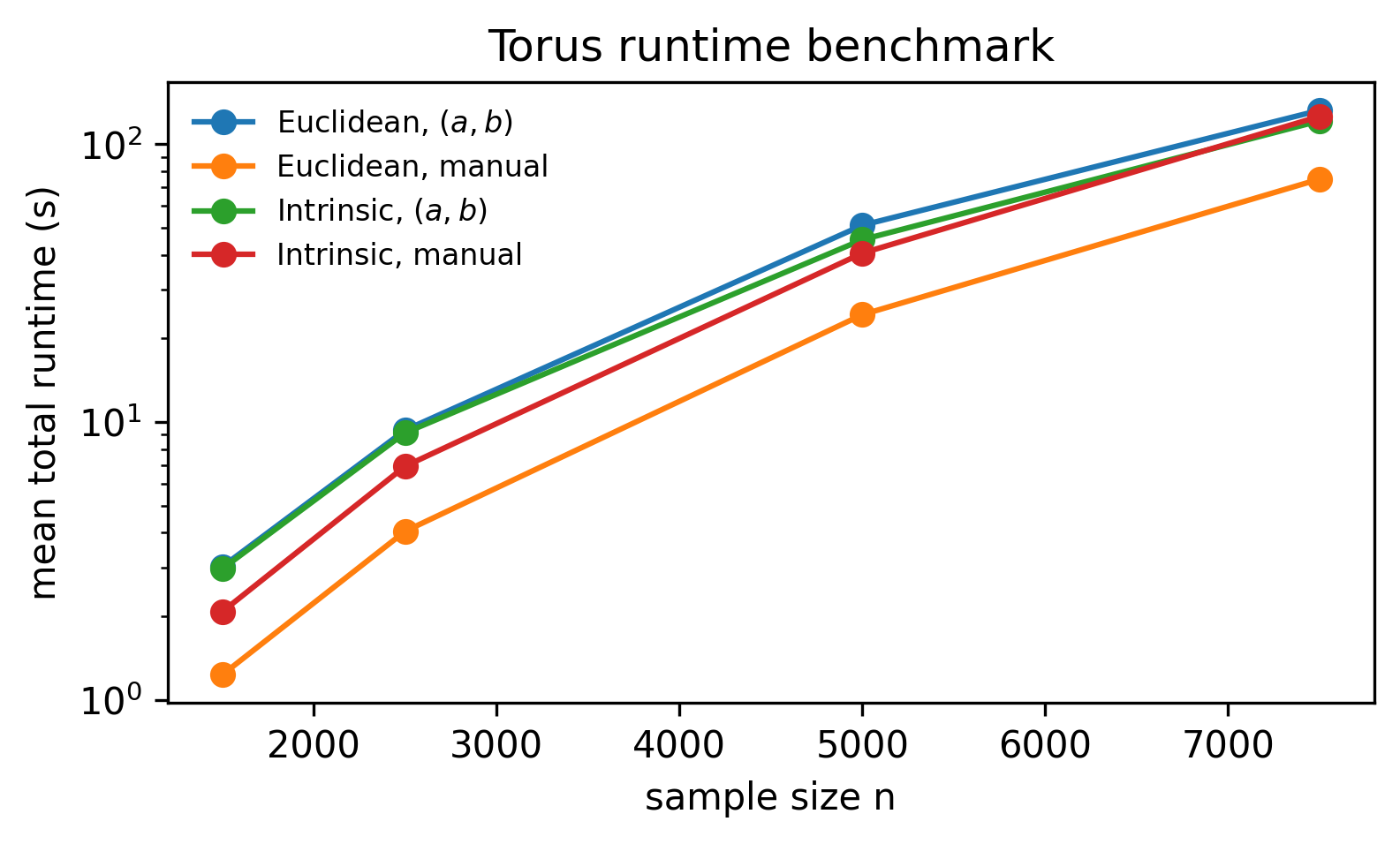}
        \caption{Torus.}
    \end{subfigure}
    \caption{Mean total runtime over ten repetitions.  The vertical axis is logarithmic because the deterministic and \((a,b)\)-standard scales can lead to very different graph densities, especially for the ant.}
    \label{fig:runtime_total}
\end{figure}

\begin{table}[H]
\centering
\begin{adjustbox}{max width=\textwidth}
\begin{tabular}{|c|c|c|c|c|c|c|c|c|c|c|}
\hline
Data & \(n\) & Metric & Scale & Edges \((10^3)\) & Edge arrays & Intrinsic & NetworkX & PL pers. & Mapper & Total\\
\hline
Ant & 5000 & Euclidean & manual & 35 & 0.02 & 0.00 & 0.15 & 13.5 & 0.34 & 14.0\\
Ant & 5000 & Euclidean & \((a,b)\) & 1235 & 0.39 & 0.00 & 4.27 & 563.7 & 10.12 & 578.7\\
Ant & 5000 & intrinsic & manual & 126 & 0.08 & 3.54 & 0.46 & 35.6 & 1.05 & 40.8\\
Ant & 5000 & intrinsic & \((a,b)\) & 965 & 0.30 & 3.54 & 3.40 & 410.3 & 7.95 & 425.7\\
Torus & 7500 & Euclidean & manual & 122 & 0.06 & 0.00 & 0.19 & 73.6 & 1.00 & 74.9\\
Torus & 7500 & Euclidean & \((a,b)\) & 206 & 0.08 & 0.00 & 0.38 & 130.1 & 1.90 & 132.5\\
Torus & 7500 & intrinsic & manual & 190 & 0.11 & 8.13 & 0.35 & 115.8 & 1.57 & 126.0\\
Torus & 7500 & intrinsic & \((a,b)\) & 180 & 0.11 & 8.14 & 0.34 & 111.0 & 1.61 & 121.3\\
\hline
\end{tabular}
\end{adjustbox}
\caption{Runtime breakdown at the largest sample size of each data set.  All times are means over ten repetitions and are reported in seconds.  ``Edge arrays'' is the windowed construction of the sparse PL proximity graph before NetworkX conversion.  ``NetworkX'' is the conversion of the edge-array representation to a NetworkX graph object for persistence and visualization routines.  ``Mapper'' is the Mapper coarsening time; Mapper persistence is omitted from the table because it is below \(0.003\) seconds in all displayed rows.}
\label{tab:runtime_largest}
\end{table}

The benchmark confirms the computational role played by the PL object.  The windowed edge-array construction is cheap relative to persistence: even in the largest ant \((a,b)\)-standard runs it takes less than half a second on average, whereas PL persistence takes several minutes.  Conversion to a NetworkX graph and Mapper coarsening are measurable but remain secondary.  The intrinsic pipeline adds the Isomap preprocessing step, whose cost increases with \(n\), but this is not the dominant term in the reported range.  Moreover, the intrinsic graph can have fewer accepted PL edges than the Euclidean graph at the same statistical scale, partially offsetting the cost of estimating geodesic distances.  The much larger difference between the deterministic and \((a,b)\)-standard ant timings also shows that the selected radius \(\delta\), through the edge density it induces, is the main practical determinant of runtime.  Overall, these timings support the methodological separation used throughout the paper: PL--Reeb graphs are the objects on which inference is performed, while Mapper coarsenings provide a substantially cheaper and more readable visualization layer.

\subsection{Residualized representatives for visualization}

The global component joining the observed minimum to the observed maximum is not highlighted among significant features, because it represents the interval baseline \([\min f,\max f]\) rather than a deviation from it.  In addition, ordinary zero-dimensional and relative one-dimensional representatives may be nested.  Drawing all nested supports literally would produce overlaid branches and unreadable figures.  The displayed supports therefore use a residualization convention, which affects only the visualization.

Fix a persistence type and let \(\sigma_1,\ldots,\sigma_m\) be the selected representatives of that type.  The representatives are ordered consistently with the elder-rule convention used to draw them: for ordinary zero-dimensional classes we use decreasing birth value, so later-born branches are processed first; for relative one-dimensional classes we use the dual order.  Let \(E(\sigma_i)\) be the set of graph edges in the support of \(\sigma_i\).  The support actually drawn for \(\sigma_i\) is
\[
E^{\rm disp}(\sigma_i)
=
E(\sigma_i)\setminus \bigcup_{j<i}E(\sigma_j).
\]
Thus each representative is shown after removing the edges already assigned to earlier representatives.  If \(E^{\rm disp}(\sigma_i)\) is empty, the feature is still present in the persistence diagram and in the statistical decision, but it is not drawn as a separate visible branch.  The persistence pairs, confidence bands, and significance decisions are always computed from the original PL graph, before residualization.

\subsection{Validation of rate-corrected subsampling radii}\label{subsec:subsampling_validation_appendix}

The validation experiment tests the raw Fasy-type subsampling radius and its two rate-corrected variants on supports for which the true covering radius can be approximated by a dense probe mesh.  The supports are a circle, a sphere, and a torus, with intrinsic dimensions \(1\), \(2\), and \(2\), respectively.  For each support we consider sample sizes \(n\in\{600,1200,2400\}\), three sampling densities (uniform, mild non-uniform, and strong non-uniform), and three logarithmic powers in the subsample-size rule.  The tables below report the default rule
\[
b_n=\left\lfloor \frac{n}{(\log n)^{1.001}}\right\rfloor,
\]
using \(500\) independent trials and \(1000\) subsamples per trial.  We report both the smallest and largest sample sizes, \(n=600\) and \(n=2400\), to show the finite-sample behaviour at the two ends of the validation grid.

For each trial, the implementation computes the mesh-based approximation of \(d_H(S_n,X)\), the raw subsampling radius \(\delta_{\rm raw}\), the logarithmic correction \(\delta_{\log}\), and the no-log correction \(\delta_{\rm pow}\).  A trial is counted as covered when the corresponding radius is at least the mesh-based covering radius.  The coverage column reports the empirical coverage together with the Wilson 95\% confidence interval.  The median ratio \(\delta/d_H(S_n,X)\) measures conservativeness; values closer to one indicate less overestimation of the covering radius.

\begin{table}[H]
\centering
\begin{adjustbox}{max width=\textwidth}
\begin{tabular}{|c|c|c|c|c|c|}
\hline
Support & Method & Coverage, \(n=600\) & Median ratio, \(n=600\) & Coverage, \(n=2400\) & Median ratio, \(n=2400\)\\
\hline
Circle & raw & 1.000 [0.992,1.000] & 11.67 & 1.000 [0.992,1.000] & 14.19\\
Circle & log & 1.000 [0.992,1.000] & 2.55 & 1.000 [0.992,1.000] & 2.47\\
Circle & no-log & 0.996 [0.986,0.999] & 1.81 & 0.998 [0.989,1.000] & 1.82\\
Sphere & raw & 1.000 [0.992,1.000] & 4.59 & 1.000 [0.992,1.000] & 5.15\\
Sphere & log & 1.000 [0.992,1.000] & 2.14 & 1.000 [0.992,1.000] & 2.15\\
Sphere & no-log & 1.000 [0.992,1.000] & 1.81 & 1.000 [0.992,1.000] & 1.84\\
Torus & raw & 1.000 [0.992,1.000] & 4.35 & 1.000 [0.992,1.000] & 5.11\\
Torus & log & 1.000 [0.992,1.000] & 2.03 & 1.000 [0.992,1.000] & 2.13\\
Torus & no-log & 1.000 [0.992,1.000] & 1.71 & 1.000 [0.992,1.000] & 1.83\\
\hline
\end{tabular}
\end{adjustbox}
\caption{Validation summary for uniform sampling with the default subsampling rule.  For \(n=600\), the subsample size is \(b_n=93\); for \(n=2400\), it is \(b_n=307\).}
\label{tab:subsampling_validation_uniform}
\end{table}

\begin{table}[H]
\centering
\begin{adjustbox}{max width=\textwidth}
\begin{tabular}{|c|c|c|c|c|c|}
\hline
Support & Method & Coverage, \(n=600\) & Median ratio, \(n=600\) & Coverage, \(n=2400\) & Median ratio, \(n=2400\)\\
\hline
Circle & raw & 1.000 [0.992,1.000] & 11.03 & 1.000 [0.992,1.000] & 13.62\\
Circle & log & 1.000 [0.992,1.000] & 2.41 & 1.000 [0.992,1.000] & 2.37\\
Circle & no-log & 0.992 [0.980,0.997] & 1.71 & 0.994 [0.983,0.998] & 1.74\\
Sphere & raw & 1.000 [0.992,1.000] & 4.47 & 1.000 [0.992,1.000] & 5.08\\
Sphere & log & 1.000 [0.992,1.000] & 2.09 & 1.000 [0.992,1.000] & 2.12\\
Sphere & no-log & 1.000 [0.992,1.000] & 1.76 & 1.000 [0.992,1.000] & 1.82\\
Torus & raw & 1.000 [0.992,1.000] & 4.11 & 1.000 [0.992,1.000] & 4.84\\
Torus & log & 1.000 [0.992,1.000] & 1.92 & 1.000 [0.992,1.000] & 2.02\\
Torus & no-log & 1.000 [0.992,1.000] & 1.62 & 1.000 [0.992,1.000] & 1.73\\
\hline
\end{tabular}
\end{adjustbox}
\caption{Validation summary for mild non-uniform sampling with the default subsampling rule.}
\label{tab:subsampling_validation_mild_nonuniform}
\end{table}

\begin{table}[H]
\centering
\begin{adjustbox}{max width=\textwidth}
\begin{tabular}{|c|c|c|c|c|c|}
\hline
Support & Method & Coverage, \(n=600\) & Median ratio, \(n=600\) & Coverage, \(n=2400\) & Median ratio, \(n=2400\)\\
\hline
Circle & raw & 1.000 [0.992,1.000] & 9.86 & 1.000 [0.992,1.000] & 13.20\\
Circle & log & 0.996 [0.986,0.999] & 2.16 & 1.000 [0.992,1.000] & 2.30\\
Circle & no-log & 0.932 [0.906,0.951] & 1.53 & 0.988 [0.974,0.994] & 1.69\\
Sphere & raw & 1.000 [0.992,1.000] & 4.22 & 1.000 [0.992,1.000] & 4.83\\
Sphere & log & 1.000 [0.992,1.000] & 1.97 & 1.000 [0.992,1.000] & 2.01\\
Sphere & no-log & 1.000 [0.992,1.000] & 1.66 & 1.000 [0.992,1.000] & 1.73\\
Torus & raw & 1.000 [0.992,1.000] & 3.68 & 1.000 [0.992,1.000] & 4.38\\
Torus & log & 1.000 [0.992,1.000] & 1.72 & 1.000 [0.992,1.000] & 1.83\\
Torus & no-log & 1.000 [0.992,1.000] & 1.45 & 1.000 [0.992,1.000] & 1.57\\
\hline
\end{tabular}
\end{adjustbox}
\caption{Validation summary for strong non-uniform sampling with the default subsampling rule.}
\label{tab:subsampling_validation_strong_nonuniform}
\end{table}

The raw radius is uniformly conservative.  The logarithmic correction substantially reduces the median ratio while keeping empirical coverage at or above the nominal level in these runs.  The no-log correction is more aggressive and often remains accurate, especially on the sphere and torus.  The circle is the most delicate support in this validation grid: because it is one-dimensional, the logarithmic factor is comparatively more relevant, and the strongly non-uniform circle at the smallest sample size is the main finite-sample failure mode of the no-log correction.  On the two-dimensional supports that drive the displayed torus and ant experiments, however, \(\delta_{\rm pow}\) performs very well.  We use \(\delta_{\rm pow}\) in the main figures for this reason, and we use \(\delta_{\log}\) as a conservative robustness check against possible artifacts of an overly aggressive correction.  In the displayed experiments this check does not reveal contradictory structure: the qualitative summaries obtained from the two corrected radii are comparable.

\section{Genericity for the newborn-cycle argument}
\label{subsec:h1_genericity_newborn}

\paragraph{Combinatorial genericity.}

We now isolate a simple combinatorial condition implying the
properties needed for \Cref{subsec:h1_component_newborn_corrected}, and then show that the corresponding class is dense.

For a vertex \(v\) of a Reeb graph \(R=(G,h_R)\), we denote by \(u(v)\) its up-degree
and by \(\ell(v)\) its down-degree.

We say that \(R\) is \emph{combinatorially generic} if:
\begin{enumerate}
    \item the values of \(h_R\) at the vertices of \(R\) are pairwise distinct;
    \item for every vertex \(v\) of \(R\),
    \[
    \min\{u(v),\ell(v)\}\le 1
    \qquad\text{and}\qquad
    \max\{u(v),\ell(v)\}\le 2.
    \]
\end{enumerate}

The next proposition shows that this condition implies exactly the distinctness
properties needed later.

\begin{prop}\label{prop:genericity_from_degree_pattern}
Let \(R=(G,h_R)\) be a finite Reeb graph. Assume that the values of \(h_R\) at the vertices
of \(R\) are pairwise distinct and that, for every vertex \(v\),
\[
\min\{u(v),\ell(v)\}\le 1,
\qquad
\max\{u(v),\ell(v)\}\le 2.
\]
Then:
\begin{enumerate}
    \item the ordinary \(H_1\)-birth times in the sublevel filtration of \(h_R\) are pairwise distinct;
    \item the ordinary \(H_1\)-birth times in the sublevel filtration of \(-h_R\) are pairwise distinct;
    \item no vertex can simultaneously support an ordinary \(H_1\)-birth for the
    sublevel filtration of \(h_R\) and an ordinary \(H_1\)-birth for the sublevel filtration of
    \(-h_R\).
\end{enumerate}
\end{prop}

\begin{proof}
We first consider the filtration of \(h_R\). View \(R\) as a simplicial complex. Since \(h_R\)
is strictly monotone on every open edge and the vertex values are pairwise distinct,
the sublevel filtration of \(R\) is, up to levelwise homotopy equivalence, identified
with the lower-star filtration. Thus, when a vertex \(v\) with value \(t=h_R(v)\) enters
the filtration, the corresponding change is obtained by adjoining the lower star of
\(v\), namely the vertex \(v\) together with the edges for which \(v\) is the upper
endpoint. The number of such edges is exactly \(\ell(v)\).

Fix \(t\in\mathbb R\). Since the vertex values are pairwise distinct, there is at most one
vertex \(v\) with \(h_R(v)=t\). If there is no such vertex, then no topological change
occurs at level \(t\).

Assume now that \(h_R(v)=t\). Since \(\ell(v)\le 2\), there are three cases.

If \(\ell(v)=0\), then only the vertex \(v\) is added, so no \(H_1\)-class is created.

If \(\ell(v)=1\), then the lower star of \(v\) consists of the vertex \(v\) and one edge
having \(v\) as upper endpoint. Adjoining this lower star cannot create a cycle, so no
\(H_1\)-class is created.

If \(\ell(v)=2\), let \(e_1,e_2\) be the two edges for which \(v\) is the upper endpoint,
and let \(x_1,x_2\) be their other endpoints. Starting from the previous stage of the
lower-star filtration, first adjoin \(v\), then \(e_1\), and finally \(e_2\). After adjoining
\(v\), no \(H_1\)-class is created. After adjoining \(e_1\), the component containing \(v\)
is attached to the component of \(x_1\), so still no \(H_1\)-class is created. Finally,
adjoining \(e_2\), either \(x_2\) lies in a different connected component of the current
graph, in which case \(e_2\) merges components and creates no cycle, or \(x_2\) lies in
the same connected component as \(v\), in which case \(e_2\) creates exactly one
cycle.

Thus, at each level \(t\), the filtration of \(h_R\) creates at most one new ordinary
\(H_1\)-class. Since distinct vertices have distinct values, the ordinary \(H_1\)-birth
times in the sublevel filtration of \(h_R\) are pairwise distinct.

Applying the same argument to \(-h_R\), and observing that the down-degree for \(-h_R\)
is exactly the up-degree for \(h_R\), the inequality \(u(v)\le 2\) implies that the
ordinary \(H_1\)-birth times in the sublevel filtration of \(-h_R\) are pairwise distinct.

Finally, if a vertex \(v\) supported an ordinary \(H_1\)-birth both for the filtration of
\(h_R\) and for that of \(-h_R\), then necessarily
$\ell(v)=2
$ and
$u(v)=2,
$
contradicting \(\min\{u(v),\ell(v)\}\le 1\).
\end{proof}

To prove density, we use the functional distortion distance. Recall that for a finite
Reeb graph \(R=(G,h_R)\), the associated path-height metric is defined by
\[
d_{h_R}(x,x')
\coloneqq
\inf_{\pi:x\leadsto x'}
\bigl(\max h_R(\pi)-\min h_R(\pi)\bigr),
\]
where the infimum ranges over all continuous paths \(\pi\) in \(R\) joining \(x\) to \(x'\). Note that this is a metric as, if \(d_{h_R}(x,x')=0\), then \(x\) and \(x'\) lie in the same connected
component of a level set of \(h_R\), and therefore \(x=x'\) by definition of the Reeb
graph.

If \(R'=(G',h_{R'})\) is another Reeb graph and
\(\varphi\colon R\to R'\), \(\psi\colon R'\to R\) are continuous maps, let
\[
C(\varphi,\psi)\coloneqq
\{(x,\varphi(x))\mid x\in R\}\cup \{(\psi(y),y)\mid y\in R'\}
\subseteq R\times R'
\]
be the associated correspondence, and define
\[
D(\varphi,\psi)\coloneqq
\frac12
\sup_{(x,y),(x',y')\in C(\varphi,\psi)}
\bigl|d_{h_R}(x,x')-d_{h_{R'}}(y,y')\bigr|.
\]
The functional distortion distance is
\[
d_{FD}(R,R')
\coloneqq
\inf_{\varphi,\psi}
\max\Bigl\{
D(\varphi,\psi),\,
\|h_R-h_{R'}\circ \varphi\|_\infty,\,
\|h_R\circ \psi-h_{R'}\|_\infty
\Bigr\},
\]
where the infimum ranges over all continuous maps \(\varphi\colon R\to R'\) and
\(\psi\colon R'\to R\).

We shall use the following comparison estimates with the functional distortion
distance. The interleaving distance is bounded by \(d_{FD}\), and the bottleneck
distance between the extended-persistence signatures is bounded, componentwise,
by a constant multiple of \(d_{FD}\); in particular the estimates of
\cite[Theorem 14]{bauer2015strong} and \cite[Theorem 15]{bauer2014measuring}
give the uniform bound
\[
d_I,d_\Delta\le 3d_{FD}.
\]
The constant \(3\) is only used in the density argument below.

We now describe the local move used to reduce the combinatorial defect.

\begin{lem}[Local branch-sliding move]\label{lem:local_degree_reduction_fd}
Let \(R=(G,h_R)\) be a Reeb graph
and let \(v\in R\) be a vertex with value \(a\coloneqq h_R(v)\). Assume that either
\[
\ell(v)\ge 3
\qquad\text{or}\qquad
(\ell(v),u(v))=(2,2).
\]
Fix \(\eta>0\) so small that, on every edge incident to \(v\), the interval
\[
\left(a-\frac{\eta}{2},a\right)
\]
is contained in the image of that edge under \(h_R\). Then there exists a Reeb graph \(R'=(G',h_{R'})\), obtained from
\(R\) by a local modification supported in an arbitrarily small neighborhood of
\(v\), such that:
\begin{enumerate}
    \item one downward branch is detached from \(v\) and reattached at a new vertex
    \(w\) satisfying
    \[
    a-\frac{\eta}{2}<h_{R'}(w)<a;
    \]
    \item if \(v'\) denotes the old vertex \(v\), viewed in \(R'\), then
    \[
    \ell(v')=\ell(v)-1,
    \qquad
    u(v')=u(v);
    \]
    \item the new vertex \(w\) satisfies
    \[
    \ell(w)=2,
    \qquad
    u(w)=1;
    \]
    \item every other vertex keeps the same up-degree and down-degree;
    \item
    \[
    d_{FD}(R,R')\le \eta.
    \]
\end{enumerate}

The symmetric statement holds if one instead slides an upward branch upward after
replacing \(h_R\) by \(-h_R\).
\end{lem}

\begin{proof}
We treat the downward case; the upward case is obtained by applying the same
argument to \(-h_R\). See also \Cref{fig:branch_sliding_move}.

Let \(e_1,e_2\) be two distinct downward edges incident to \(v\). Since \(h_R\) is
continuous and
$\left(a-\frac{\eta}{2},a\right)
$
is contained in the image of every downward edge incident to \(v\), and since the graph is
finite, we may choose a point \(w\) in the interior of \(e_1\) such that
\[
a-\frac{\eta}{2}<h_R(w)<a,
\]
and, writing
\[
\kappa\coloneqq a-h_R(w),
\]
the interval
$[h_R(w)-\kappa,h_R(w)]
$
contains no critical value of \(h_R\) on the portions of \(e_1\) and \(e_2\) below \(w\)
and below \(v\), respectively. By construction,
$0<\kappa<\frac{\eta}{2}.
$

Construct a new graph \(R'\) by detaching the edge \(e_2\) from \(v\) and
reattaching it at the point \(w\). Declare \(w\) to be a new vertex. Then the old edge
\(e_1\) is split by \(w\) into two edges of \(R'\), which we denote by
\[
e_1^{\uparrow}
\qquad\text{and}\qquad
e_1^{\downarrow},
\]
where \(e_1^{\uparrow}\) joins \(w\) to the old vertex \(v\), now denoted by \(v'\), and
\(e_1^{\downarrow}\) joins \(w\) to the lower endpoint of \(e_1\). We denote by \(e_2'\)
the reattached copy of \(e_2\) in \(R'\).

We now define a function
$h_{R'}\colon R'\to\mathbb R.
$

On the complement of the reattached edge \(e_2'\), we let \(h_{R'}\) coincide with \(h_R\)
under the natural identification of the underlying geometric realizations. On the
edge \(e_2'\), we define \(h_{R'}\) so that:
\begin{enumerate}
    \item \(h_{R'}(w)=h_R(w)\);
    \item \(h_{R'}\) agrees with \(h_R\) on the subsegment of the original edge \(e_2\)
    corresponding to the points whose \(h_R\)-values are at most \(h_R(w)\). 
\end{enumerate}
Equivalently, \(h_{R'}|_{e_2'}\) is obtained by truncating the top part of the old function
\(h_R|_{e_2}\) at the new upper value \(h_R(w)\).

By construction, \(h_{R'}\) is continuous on every edge and is strictly monotone on \(e_2'\). At the new vertex \(w\), the
restrictions of \(h_{R'}\) along the two incident branches \(e_1^{\downarrow}\) and \(e_2'\)
both take the value \(h_R(w)\), and the restriction along \(e_1^{\uparrow}\) also takes the
same value there. Thus \(h_{R'}\) is continuous at \(w\). Therefore \(h_{R'}\) is continuous on
\(R'\), and it is strictly monotone on every edge of \(R'\). Hence \(R'\) is
again a Reeb graph.

At the old vertex \(v'\), one downward edge has been removed and no upward edge
has been changed. Hence
\[
\ell(v')=\ell(v)-1,
\qquad
u(v')=u(v).
\]
At the new vertex \(w\), the edge \(e_1^{\uparrow}\) is the unique upward edge, while
\(e_1^{\downarrow}\) and the reattached edge \(e_2'\) are the two downward edges. Thus
\[
u(w)=1,
\qquad
\ell(w)=2.
\]
All other vertices are unchanged.

We now define the maps entering the functional distortion distance.

Let
\[
S_1\subset e_1,\qquad S_2\subset e_2
\]
be the closed segments in the original graph \(R\) consisting of the points whose
\(h_R\)-values lie in \([a-\kappa,a]\). By construction,
\[
h_R(S_1)=h_R(S_2)=h_{R'}(e_1^{\uparrow})=[a-\kappa,a].
\]

Define
$\varphi\colon R\to R'
$
as follows.
\begin{itemize}
    \item Outside \(S_1\cup S_2\), let \(\varphi\) be the evident identification with the
    unchanged part of \(R'\).
    \item If \(x\in S_1\cup S_2\), let \(\varphi(x)\) be the unique point of \(e_1^{\uparrow}\) such that
    \[
    h_{R'}(\varphi(x))=h_R(x).
    \]
\end{itemize}
Thus \(\varphi\) identifies the two upper segments \(S_1\) and \(S_2\) with the single
segment \(e_1^{\uparrow}\), preserving the height coordinate. In particular,
$h_{R'}\circ \varphi = h_R.
$

We next define
$\psi\colon R'\to R.
$
Outside the local neighborhood where the sliding move took place, i.e. outside \(e_1^{\uparrow},e_1^{\downarrow},\) and \(e_2'\), let \(\psi\) be the
evident identification with the unchanged part of \(R\). On the local neighborhood,
define \(\psi\) as follows:
\begin{itemize}

    \item for every \(y\in e_1^{\uparrow}\), set
$    \psi(y)=v;
$   
    \item on the lower part of \(e_1^{\downarrow}\), i.e. \(y\in e_1^{\downarrow}\) with \(h_{R'}(y)\le h_R(w)-\kappa\), let
    \(\psi\) be the evident identification with the corresponding part of the original
    edge \(e_1\);

    \item next we take the subarc of \(e_1^{\downarrow}\) with \(h_{R'}(y)\in [h_R(w)-\kappa,h_R(w)]\), and stretch it to cover \([h_R(w)-\kappa,h_R(w)+\kappa]\). That is, define
    \(\psi(y)\) to be the unique point of the original edge \(e_1\) with
    \[
    h_R(\psi(y))=2h_{R'}(y)-h_R(w)+\kappa;
    \]
    \item we act analogously on \(e_2'\): on the lower part of the reattached edge \(e_2'\) with \(h_{R'}(y)\le h_R(w)-\kappa\),
    let \(\psi\) be the evident identification with the corresponding part of the
    original edge \(e_2\);
    \item on the subarc of \(e_2'\) with \(h_{R'}(y)\in [h_R(w)-\kappa,h_R(w)]\), define
    \(\psi(y)\) to be the unique point of the original edge \(e_2\) with
    \[
    h_R(\psi(y))=2h_{R'}(y)-h_R(w)+\kappa.
    \]
\end{itemize}

By construction, \(\psi\) is continuous on each of the modified edges, it matches
continuously with the identity at the level \(h_R(w)-\kappa\), and along both
\(e_1^{\downarrow}\) and \(e_2'\) one has \(\psi(y)\to v\) as \(y\to w\). Hence \(\psi\) is
continuous on \(R'\).

We now estimate the height error of \(\psi\). On \(e_1^{\uparrow}\), one has
\[
h_R(\psi(y))-h_{R'}(y)=a-h_{R'}(y)\in [0,\kappa].
\]
On the lower parts of \(e_1^{\downarrow}\) and \(e_2'\) where \(\psi\) is the identity, one has
\[
h_R(\psi(y))-h_{R'}(y)=0.
\]
On the stretched subarcs of \(e_1^{\downarrow}\) and \(e_2'\), where
\[
h_R(\psi(y))=2h_{R'}(y)-h_R(w)+\kappa,
\]
one obtains
\[
h_R(\psi(y))-h_{R'}(y)=h_{R'}(y)-h_R(w)+\kappa.
\]
Since \(h_{R'}(y)\in [h_R(w)-\kappa,h_R(w)]\), it follows that
\[
0\le h_R(\psi(y))-h_{R'}(y)\le \kappa.
\]
Therefore
\[
0\le h_R(\psi(y))-h_{R'}(y)\le \kappa
\qquad\text{for all }y\in R'.
\]
Since \(h_{R'}\circ\varphi=h_R\), we conclude that
\[
\|h_R-h_{R'}\circ\varphi\|_\infty=0,
\qquad
\|h_R\circ\psi-h_{R'}\|_\infty\le \kappa<\frac{\eta}{2}.
\]

It remains to estimate \(D(\varphi,\psi)\). Let
\[
(x,y),(x',y')\in C(\varphi,\psi).
\]
We claim that
\[
\bigl|d_{h_R}(x,x')-d_{h_{R'}}(y,y')\bigr|\le 4\kappa.
\]

To describe the modified part more precisely, let
$T_1\subset e_1, T_2\subset e_2
$
be the segments in the original graph \(R\) consisting of the points whose
\(h_R\)-values lie in \([h_R(w)-\kappa,h_R(w))\). Thus
\[
h_R(T_1)=h_R(T_2)=[h_R(w)-\kappa,h_R(w)).
\]

We first record two elementary estimates.

For every \(x\in R\), one has
\[
d_{h_R}\bigl(x,\psi(\varphi(x))\bigr)\le \kappa.
\]
Indeed, if \(x\notin S_1\cup S_2\cup T_1\cup T_2\), then \(\varphi(x)=x\), and
\(\psi(\varphi(x))=x\), because \(\varphi(x)\) lies outside \(e_1^{\uparrow}\) and outside
the stretched subarcs of \(e_1^{\downarrow}\) and \(e_2'\), where \(\psi\) is the evident
identification with \(R\).

Assume next that \(x\in T_1\cup T_2\). Then \(\varphi(x)=x\), and by construction
\(\psi(\varphi(x))\) lies on the same original edge \(e_1\) or \(e_2\) as \(x\), with
\[
0\le h_R(\psi(\varphi(x)))-h_R(x)\le \kappa.
\]
Therefore \(x\) and \(\psi(\varphi(x))\) are joined inside that same edge by a
path whose \(h_R\)-range has length at most \(\kappa\), so
\[
d_{h_R}\bigl(x,\psi(\varphi(x))\bigr)\le \kappa.
\]

Finally, if \(x\in S_1\cup S_2\), then \(\varphi(x)\in e_1^{\uparrow}\), hence
\(\psi(\varphi(x))=v\). Since \(x\) lies on one of the two original edges \(e_1,e_2\)
and has \(h_R\)-value in \([a-\kappa,a]\), the points \(x\) and \(v\) can be joined inside
that edge by a path whose \(h_R\)-range has length at most \(\kappa\). Thus again
\[
d_{h_R}\bigl(x,\psi(\varphi(x))\bigr)\le \kappa.
\]

Similarly, for every \(y\in R'\), one has
\[
d_{h_{R'}}\bigl(y,\varphi(\psi(y))\bigr)\le \kappa.
\]
Indeed, if \(y\) lies outside \(e_1^{\uparrow}, e_1^{\downarrow}\), and \(e_2'\), then \(\psi(y)=y\) and
\(\varphi(\psi(y))=y\).

If \(y\in e_1^{\uparrow}\), then \(\psi(y)=v\) and \(\varphi(v)=v'\), so \(y\) and
\(\varphi(\psi(y))=v'\) lie in \(e_1^{\uparrow}\), whose \(h_{R'}\)-range is exactly
\([a-\kappa,a]\); hence
\[
d_{h_{R'}}\bigl(y,\varphi(\psi(y))\bigr)\le \kappa.
\]

It remains to consider \(y\in e_1^{\downarrow}\cup e_2'\). If \(y\) lies on the unchanged
lower part of one of these edges, then \(\psi(y)\) is the corresponding point of the
original edge with the same height, and \(\varphi(\psi(y))=y\), so the distance is \(0\).

If instead \(y\) lies on one of the stretched subarcs, then by construction
\(\varphi(\psi(y))\) lies in the same local modified neighborhood and satisfies
\[
h_{R'}\bigl(\varphi(\psi(y))\bigr)=h_R(\psi(y)).
\]
Since
\[
0\le h_R(\psi(y))-h_{R'}(y)\le \kappa,
\]
the two points \(y\) and \(\varphi(\psi(y))\) can be joined inside the local modified
neighborhood by a path whose \(h_{R'}\)-range is contained in an interval of length at
most \(\kappa\). Therefore
\[
d_{h_{R'}}\bigl(y,\varphi(\psi(y))\bigr)\le \kappa.
\]

Next observe that \(\varphi\) does not increase path-height distance. Indeed, since
\(h_{R'}\circ\varphi=h_R\), the image under \(\varphi\) of any path in \(R\) has exactly the same
range of height values, and therefore
\[
d_{h_{R'}}\bigl(\varphi(x),\varphi(x')\bigr)\le d_{h_R}(x,x')
\qquad\text{for all }x,x'\in R.
\]
Likewise, since \(0\le h_R\circ\psi-h_{R'}\le \kappa\), the image under \(\psi\) of any path in
\(R'\) enlarges the height range by at most \(\kappa\), so
\[
d_{h_R}\bigl(\psi(y),\psi(y')\bigr)\le d_{h_{R'}}(y,y')+\kappa
\qquad\text{for all }y,y'\in R'.
\]

We now estimate the distortion. Let \((x,y)\in C(\varphi,\psi)\). Then either
\(y=\varphi(x)\) or \(x=\psi(y)\). In the first case,
\[
d_{h_R}\bigl(x,\psi(y)\bigr)=d_{h_R}\bigl(x,\psi(\varphi(x))\bigr)\le \kappa
\]
by the first estimate. In the second case, \(x=\psi(y)\), so
\[
d_{h_R}\bigl(x,\psi(y)\bigr)=0.
\]
Hence in all cases
\[
d_{h_R}\bigl(x,\psi(y)\bigr)\le \kappa.
\]
Similarly,
\[
d_{h_R}\bigl(x',\psi(y')\bigr)\le \kappa.
\]
Therefore, by the triangle inequality,
\[
d_{h_R}\bigl(\psi(y),\psi(y')\bigr)\le d_{h_R}(x,x')+2\kappa.
\]
Using the previous bound for \(\psi\), we obtain
\[
d_{h_R}(x,x')\le d_{h_R}\bigl(\psi(y),\psi(y')\bigr)+2\kappa
\le d_{h_{R'}}(y,y')+3\kappa.
\]

On the other hand, applying the nonexpansiveness of \(\varphi\) to the points
\(\psi(y)\) and \(\psi(y')\), we get
\[
d_{h_{R'}}\bigl(\varphi(\psi(y)),\varphi(\psi(y'))\bigr)\le d_{h_R}\bigl(\psi(y),\psi(y')\bigr).
\]
Since
\[
d_{h_{R'}}\bigl(y,\varphi(\psi(y))\bigr)\le \kappa,
\qquad
d_{h_{R'}}\bigl(y',\varphi(\psi(y'))\bigr)\le \kappa,
\]
the triangle inequality yields
\[
d_{h_{R'}}(y,y')\le d_{h_R}\bigl(\psi(y),\psi(y')\bigr)+2\kappa
\le d_{h_R}(x,x')+4\kappa.
\]

Thus
\[
\bigl|d_{h_R}(x,x')-d_{h_{R'}}(y,y')\bigr|\le 4\kappa.
\]
In particular,
\[
D(\varphi,\psi)\le 2\kappa<\eta.
\]

Combining this with
\[
\|h_R-h_{R'}\circ\varphi\|_\infty=0
\qquad\text{and}\qquad
\|h_R\circ\psi-h_{R'}\|_\infty\le \kappa<\frac{\eta}{2},
\]
we conclude that
\[
d_{FD}(R,R')\le \eta.
\]
\end{proof}

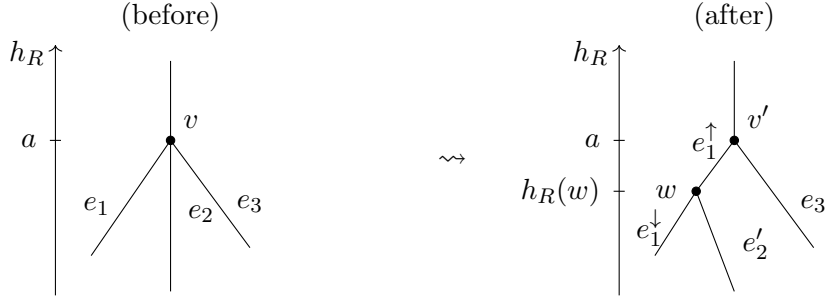
\begin{figure}
\centering
\begin{tikzpicture}[scale=1.05, every node/.style={font=\small}]
    \begin{scope}[xshift=0cm]
        \node at (0,3.35) {(before)};
        
        \draw[->] (-1.45,-0.15) -- (-1.45,3.0);
        \node[left] at (-1.45,2.9) {$h_R$};
        \draw (-1.52,1.8) -- (-1.38,1.8);
        \node[left] at (-1.55,1.8) {$a$};

        \filldraw (0,1.8) circle (1.5pt);
        \node[above right=1pt] at (0,1.8) {$v$};

        \draw (0,1.8) -- (0,2.8);

        \draw (0,1.8) -- (-1.0,0.35);
        \draw (0,1.8) -- (0,-0.1);
        \draw (0,1.8) -- (1.0,0.45);

        \node[left=2pt] at (-0.56,0.97) {$e_1$};
        \node[right=2pt] at (0.015,0.88) {$e_2$};
        \node[right=2pt] at (0.63,1.00) {$e_3$};
    \end{scope}

    \node at (3.55,1.55) {$\leadsto$};

    \begin{scope}[xshift=7.1cm]
        \node at (0,3.35) {(after)};
        
        \draw[->] (-1.45,-0.15) -- (-1.45,3.0);
        \node[left] at (-1.45,2.9) {$h_R$};
        \draw (-1.52,1.8) -- (-1.38,1.8);
        \node[left] at (-1.55,1.8) {$a$};
        \draw (-1.52,1.16) -- (-1.38,1.16);
        \node[left] at (-1.55,1.16) {$h_R(w)$};

        \filldraw (0,1.8) circle (1.5pt);
        \node[above right=1pt] at (0,1.8) {$v'$};

        \draw (0,1.8) -- (0,2.8);

        \draw (0,1.8) -- (-0.48,1.16);
        \filldraw (-0.48,1.16) circle (1.5pt);
        \node[left=3pt] at (-0.48,1.16) {$w$};
        \draw (-0.48,1.16) -- (-1.0,0.35);

        \draw (-0.48,1.16) -- (0,-0.1);

        \draw (0,1.8) -- (1.0,0.45);

        \node[left=2pt] at (0.0,1.80) {$e_1^{\uparrow}$};        \node[left=2pt] at (-0.7,0.70) {$e_1^{\downarrow}$};
        \node[right=2pt] at (-0.1,0.50) {$e_2'$};
        \node[right=2pt] at (0.63,1.00) {$e_3$};
    \end{scope}
\end{tikzpicture}
\caption{Local downward branch-sliding move. One downward branch \(e_2\) is detached
from the vertex \(v\) and reattached at a nearby point \(w\) on another downward branch
\(e_1\). The original vertex \(v\) loses one downward branch, while the new vertex
\(w\) has type \((u,\ell)=(1,2)\).}
\label{fig:branch_sliding_move}
\end{figure}

To iterate the local move, define the combinatorial defect of a Reeb graph
\(R=(G,h_R)\) by
\[
\Delta(R)\coloneqq
\sum_{v\in V(R)}
\Bigl(
\max\{\ell(v)-2,0\}
+
\max\{u(v)-2,0\}
+
\max\{\min\{u(v),\ell(v)\}-1,0\}
\Bigr).
\]
Then \(\Delta(R)=0\) if and only if
\[
\min\{u(v),\ell(v)\}\le 1
\qquad\text{and}\qquad
\max\{u(v),\ell(v)\}\le 2
\quad\text{for every vertex }v.
\]

\begin{prop}[Density of the combinatorially generic class]\label{prop:dense_generic_class}
For every Reeb graph \(R=(G,h_R)\) and every \(\varepsilon>0\), there exists a
combinatorially generic Reeb graph \(R_\varepsilon\) such that
\[
d_{FD}(R,R_\varepsilon)\le \varepsilon.
\]
Consequently,
\[
d_\Delta(R,R_\varepsilon),
d_I(R,R_\varepsilon)
\le 3\varepsilon.
\]
\end{prop}

\begin{proof}
We argue by induction on \(\Delta(R)\).

If \(\Delta(R)=0\), then \(R\) already satisfies the degree condition
\[
\min\{u(v),\ell(v)\}\le 1,
\qquad
\max\{u(v),\ell(v)\}\le 2
\quad\text{for every }v.
\]
Since the graph is finite, one may perturb the vertex values by an arbitrarily small
amount so as to make them pairwise distinct without changing the combinatorics.
By the standard stability of the functional distortion distance under sup-norm perturbations of
the function on a fixed space, this changes the distance by at most the
size of the perturbation. Hence \(R\) can be approximated arbitrarily well, in
particular arbitrarily well in \(d_{FD}\), by a combinatorially generic graph.

Assume now \(\Delta(R)>0\). Then there exists a vertex \(v\) such that either
\[
\ell(v)\ge 3,\qquad\text{or}\qquad u(v)\ge 3,\qquad\text{or}\qquad (\ell(v),u(v))=(2,2).
\]
Choose \(0<\eta\le\varepsilon/2\) sufficiently small that the hypotheses of
\Cref{lem:local_degree_reduction_fd} hold, and apply that lemma with parameter
\(\eta\). We obtain a Reeb graph \(R'\) such that
\[
d_{FD}(R,R')\le \eta\le \varepsilon/2
\]
and
\[
\Delta(R')\le\Delta(R)-1.
\]
Indeed, the local move decreases by \(1\) at least one, and at most two, of the
positive summands contributing to \(\Delta\), while none of the remaining summands
increase. The move creates one new vertex, of type \((1,2)\) or \((2,1)\), which
contributes \(0\) to \(\Delta\).

By the induction hypothesis applied to \(R'\) with parameter \(\varepsilon/2\),
there exists a combinatorially generic graph \(R_\varepsilon\) such that
\[
d_{FD}(R',R_\varepsilon)\le \varepsilon/2.
\]
By the triangle inequality,
\[
d_{FD}(R,R_\varepsilon)\le \varepsilon.
\]

Finally, since \(d_\Delta,d_I\le 3d_{FD}\), we obtain
\[
d_\Delta(R,R_\varepsilon),
d_I(R,R_\varepsilon)
\le 3\varepsilon.
\]
\end{proof}

This density statement is the ingredient used in the main comparison section to pass
from the combinatorially generic case to arbitrary Reeb graphs.

\printbibliography

\end{document}